\documentclass[10pt]{amsart}
\usepackage{amssymb,latexsym,amsthm}
\usepackage[centertags]{amsmath}
\usepackage{amsfonts}
\usepackage{graphicx}
\usepackage{color}

\newcommand{\abs}[1]{\left\lvert{#1}\right\rvert}
\newcommand{\norm}[1]{\left\|{#1}\right\|}

\DeclareMathOperator{\GL}{\rm{GL}}

\DeclareMathOperator{\inter}{\rm{int}}
\DeclareMathOperator{\bd}{\partial}
\DeclareMathOperator{\cl}{cl}
\DeclareMathOperator{\diam}{\rm{diam}}

\DeclareMathOperator{\fix}{\rm{Fix}}

\DeclareMathOperator{\cPE}{\rm{c}_{\mc{E}}}
\DeclareMathOperator{\cIB}{\rm{c}_{I}}
\DeclareMathOperator{\bdPE}{\rm{b}_{\mc{E}}}
\DeclareMathOperator{\bdIB}{\rm{b}_{I}}
\DeclareMathOperator{\iPE}{\mc{E}}

\newcommand{\mc}{\mathcal}

\newcommand{\ol}{\overline}
\renewcommand{\hat}{\widehat}
\newcommand{\til}{\widetilde}

\newcommand{\R}{\mathbb{R}}\newcommand{\N}{\mathbb{N}}
\newcommand{\Z}{\mathbb{Z}}\newcommand{\Q}{\mathbb{Q}}
\newcommand{\T}{\mathbb{T}}\newcommand{\C}{\mathbb{C}}
\newcommand{\D}{\mathbb{D}}

\newcommand{\sm}{\setminus}

\newcommand{\id}{\mathrm{Id}}
\newcommand{\deck}{\operatorname{Deck}}

\newcommand{\ie}{i.e.\ }
\newcommand{\eg}{e.g.\ }

\newtheorem{theorem}{Theorem}[section]
\newtheorem{corollary}[theorem]{Corollary}
\newtheorem{lemma}[theorem]{Lemma}
\newtheorem{proposition}[theorem]{Proposition}

\newtheorem{claim}{Claim}
\newtheorem{subclaim}{Claim}[claim]
\newtheorem*{theorem*}{Theorem}
\newtheorem*{claim*}{Claim}

\newtheorem{mthm}{Theorem}

\theoremstyle{definition}
\newtheorem{definition}[theorem]{Definition}
\newtheorem{example}[theorem]{Example}

\theoremstyle{remark}
\newtheorem{remark}[theorem]{Remark}

\title{Prime ends rotation numbers and periodic points}
\author{Andres Koropecki}
\address{Universidade Federal Fluminense, Instituto de Matem\'atica, Rua M\'ario Santos Braga S/N, 24020-140 Niteroi, RJ, Brasil}
\email{ak@id.uff.br}
\thanks{The first author was partially supported by CNPq-Brasil.}
\author{Patrice Le Calvez}
\address{Institut de Math\'ematiques de Jussieu, UMR 7586 CNRS, UPMC, case 247? 4 place Jussieu, 75252 Paris Cedex, France}
\email{lecalvez@math.jussieu.fr}

\author{Meysam Nassiri}
\address{School of Mathematics, Institute for Research in Fundamental Sciences (IPM), PO Box 19395-5746, Tehran, Iran}
\email{nassiri@ipm.ir}

\begin{document}

\begin{abstract}
We study the problem of existence of a periodic point in the boundary of an invariant domain for a surface homeomorphism. In the area-preserving setting, a complete classification is given in terms of rationality of Carath\'eordory's prime ends rotation number, similar to Poincar\'e's theory for circle homeomorphisms. In particular, we prove the converse of a classic result of Cartwright and Littlewood. This has a number of consequences for generic area preserving surface diffeomorphisms. For instance, we extend previous results of J. Mather on the boundary of invariant open sets for $C^r$-generic area preserving diffeomorphisms. 
Most results are proved in a general context, for homeomorphisms of arbitrary surfaces with a weak nonwandering-type hypothesis. This allows us to prove a conjecture of R. Walker about co-basin boundaries, and it also has applications in holomorphic dynamics.
\end{abstract}

\maketitle

\setcounter{tocdepth}{1}
\begin{small}
\tableofcontents
\end{small}

\section{Introduction}

The boundary of a simply connected domain on a surface can have a very complicated topology. A classic example of this fact is the pseudo-circle \cite{bing, handel-pseudo}, which is a hereditarily indecomposable continuum, and in particular nowhere locally connected. For this reason, there is a very sharp step between studying the dynamics of a circle homeomorphism and studying the dynamics of a homeomorphism on the boundary of an invariant simply connected domain.

On the circle, Poincar\'e introduced the notion of rotation number, an invariant for the dynamics, and established the following celebrated result:

\begin{quotation}
\hspace{-.5cm} \emph{There is a periodic point if and only if the rotation number is rational.}
\end{quotation}

In order to study the dynamics of a surface homeomorphism $f$ on the boundary of a simply connected domain $U$, it is natural to try to define a similar invariant and replicate Poincar\'e's results. This can be done by means of Carath\'eodory's \emph{prime ends} compactification \cite{caratheodory, caratheodory2, mather-caratheodory} which is obtained by adding a circle to $U$ with an appropriate topology so that the resulting space $\hat{U}=U\sqcup \mathbb{S}^1$ is homeomorphic to the closed unit disk $\ol{\mathbb{D}}$. The homeomorphism $f$ extends to a homeomorphism $\hat{f}$ of $\hat{U}$, therefore inducing a homeomorphism $\hat{f}|_{\mathbb{S}^1}$ on the circle of prime ends. One may in this way associate to $f$ and $U$ a \emph{prime ends rotation number} $\rho(f, U)$ as the Poincar\'e rotation number of $\hat{f}|_{\mathbb{S}^1}$.

It is already noted in \cite{cartwright-littlewood} that a rational $\rho(f,U)$ does not necessarily imply that there is a periodic point in $\bd U$, and neither the converse is true. However, in the area-preserving setting, the situation is different, as is shown by Cartwright and Littlewood in \cite{cartwright-littlewood}:
\begin{quotation}
\emph{If $f$ is area-preserving, $U$ is relatively compact, and the prime ends rotation number is rational, then there is a periodic point in $\bd U$.}
\end{quotation}
Similar results, showing that the rationality of the rotation number leads to the existence of periodic points with the area-preserving hypotheses replaced by weaker conditions near the boundary of $U$ are also studied in \cite{cartwright-littlewood}. Since then, several contributions were made in the same direction, \eg \cite{alligood-yorke, barge1, barge2, barge3, JEMS}.  

However, little is known about the converse problem. Namely, if $\rho(f,U)$ is irrational, what can we say about the dynamics on $\bd U$? This question is already mentioned in \cite{cartwright-littlewood}, and simple examples show that even in the area-preserving case there can be a fixed point in $\bd U$ (see Examples \ref{exm:invers2} and \ref{exm:invers}). However, the examples are very specific and have exactly one fixed point and no other periodic points.

This  is the subject of this paper. We present a complete classification theorem in this direction, which also gives topological information on the boundary of $U$ in the case that the rotation number is irrational.
Such a classification has applications both in the context of $C^r$-generic area preserving diffeomorphisms, and in holomorphic dynamics.

\subsection{Main results}
Note that if $f$ is area-preserving on a closed surface, then $f$ is nonwandering (\ie every nonempty open set intersects one of its forward iterates). Since in general only the latter (weaker) condition is used, we state our results for nonwandering maps. 

All the examples of a nonwandering homeomorphism of a closed surface where $\rho(f,U)$ is irrational and there is a periodic point in $\bd U$ are on the sphere, and they all have the property that $\bd U$ is a non-separating continuum with exactly one fixed point and no other periodic point. Such examples exist even in the $C^\infty$ category (see Example \ref{exm:invers2}). The next theorem shows that this is the \emph{only} possibility, and provides additional topological information (see Section \ref{sec:prelim} for definitions).
 
\begin{mthm} \label{thm:intro-main-gen-NW} 
 Let $f\colon S\to S$ be a nonwandering orientation-preserving homeomorphism of a closed orientable surface, and $U$ an open $f$-invariant simply connected set whose boundary has more than one point and such that $\rho(f,U)$ is irrational. Then one of the following holds:
\begin{itemize}
\item[(i)] $\bd U$ is a contractible annular continuum without periodic points;
\item[(ii)] $S$ is a sphere, $U$ is dense in $S$, and $S\sm U$ is a non-separating continuum containing a unique fixed point and no other periodic points.
\end{itemize}
\end{mthm}

We would like to emphasize that without any assumption of nonwandering type this result is not valid, as shown by an example of Walker \cite{walker} (see Example \ref{ex:walker}). However, since most of the arguments are local in a neighborhood of the boundary of $U$, it is possible to relax the nonwandering condition to a much weaker one. We call it the $\bd$-nonwandering condition (see Definition \ref{def:bdnw}). We will not define it here, but one should keep in mind that it is weaker than requiring that some neighborhood of the boundary of $U$ be contained in the nonwandering set of $f$ (see Proposition \ref{pro:bdnw-gen}), and in particular it holds if $f$ is area-preserving. But in addition, this condition holds in other contexts, for instance when $f$ is a holomorphic diffeomorphism with irrational rotation number. Most of our results are proved under the $\bd$-nonwandering hypothesis. 

We also mention that in case (i) of Theorem \ref{thm:intro-main-gen-NW}, the rotation number of any invariant probability measure supported in $\bd U$ (defined using an annular neighborhood) is equal to $\rho(f,U)$ (see \cite{franks-annulus, matsumoto}).

The proof of Theorem \ref{thm:intro-main-gen-NW} is based on the following 

\begin{mthm} 
\label{thm:intro-main}
Let $f\colon \R^2\to \R^2$ be an orientation preserving homeomorphism and $U\subsetneq \R^2$ be an open $f$-invariant simply connected set.  If $f$ is $\bd$-nonwandering in $U$ and $\rho(f,U)\neq 0$, then there are no fixed points for $f$ on the boundary of $U$.
Moreover, if $U$ is unbounded then there are no fixed points for $f$ in the complement of $U$.
\end{mthm}

In fact, using the above theorem we obtain a more general version of Theorem \ref{thm:intro-main-gen-NW}, which holds for surfaces which are not necessarily compact:

\begin{mthm}\label{thm:intro-main-gen} 
Let $f\colon S\to S$ be an orientation preserving homeomorphism of an orientable surface $S$ of finite type, and $U\subset S$ an open $f$-invariant topological disk such that $S\sm U$ has more than one point. Assume further that $f$ is $\bd$-nonwandering in $U$ and $\rho(f,U)$ is irrational. Then exactly one of the following holds:
\begin{itemize}
\item[(i)] $\bd U$ is aperiodic, and $\ol{U}$ is compact and contractible;
\item[(ii)] $S$ is a sphere, and the only periodic point of $f$ in $\bd U$ is a unique fixed point;
\item[(iii)] $S$ is a plane, and $\bd U$ is unbounded and aperiodic.
\end{itemize}
Moreover, there is a neighborhood $W$ of $\bd U$, which can be chosen as an annulus in case (i), a disk in case (ii) and the complement of a closed disk in case (iii), such that every connected component of $S\sm \bd U$ contained in $W$ is wandering.
\end{mthm}

A main ingredient in the proofs is the following technical theorem, which has its own character.
It relates the genus of the surface, the prime ends rotation number of an invariant open disk, and the existence of a translation arc intersecting its boundary.  

An arc $\gamma$ is called $N$-translation arc for a homeomorphism $f$, if it is a simple arc connecting a point $x$ to its image $f(x)$ and the concatenation of its first $N$ iterates is also a simple arc.

\begin{mthm}
\label{thm:intro-arclemma}
For any $g\in \N\cup\{0\}$ and $0\neq \alpha \in \R/\Z$, there is $N_{\alpha,g}\in \N$ such that the following holds.
Let $S$ be an orientable surface of genus $g$,  and $f\colon S\to S$ an orientation preserving homeomorphism. Suppose that $U\subset S$ is an invariant open topological disk such that $S\sm U$ has more than one point, $f$ is $\bd$-nonwandering in $U$, and $\rho(f,U)=\alpha$. Then there is a compact set $K\subset U$ such that every $N_{\alpha,g}$-translation arc in  $S\sm K$ is disjoint from $\bd U$.
\end{mthm}

For a motivation on this result, see the introduction of Section \ref{sec:arclemma}. We only mention here that a typical situation where $N$-translation arcs appear is when one considers the stable or unstable manifold of a fixed point of saddle type. 

The prime ends compactification can be defined for more general (non-simply connected) sets. Theorems \ref{thm:intro-main-gen-NW}, \ref{thm:intro-main-gen} and \ref{thm:intro-arclemma} are generalized to this context as Corollary \ref{coro:general}, Theorem \ref{th:general}, and Theorem \ref{th:arc-general}, respectively.

\subsection{Consequences for area-preserving surface diffeomorphisms}

The above results have a number of consequences on generic properties of area-preserving surface diffeomorphisms. See Section \ref{sec:mather}.
Here, we mention two of them. 

\begin{mthm}\label{t:mather}
Let $f$ be a $C^r$ generic area preserving diffeomorphism on a closed orientable surface, $r\geq 1$. Then there are no periodic points on the boundary of any periodic complementary domain.
\end{mthm}

Recall that a {\it complementary domain} is a connected component of the complement of a continuum (\ie of a compact connected set with more than one point). The generic condition in this theorem is given explicitly in Section \ref{sec:mather}.

This theorem solves a problem that was extensively studied by J. Mather \cite{mather-area}. He proved that the (Carath\'eodory) rotation number of any \emph{topological end} of a complementary domain $U$ is irrational under certain conditions (the so-called Moser genericity) which hold generically in any $C^r$ topology if $r$ is sufficiently large (\eg $r\geq 16$).  This gave some information on the boundary of $U$, for instance implying that any periodic point in $\bd U$ is necessarily a hyperbolic saddle $p$ with its stable and unstable manifolds also contained in $\bd U$. It is known that this is not possible if the surface is the sphere or the torus, due to results of Pixton \cite{pixton} and Oliveira \cite{oliveira} which imply that $p$ would have a homoclinic intersection. However, for surfaces of higher genus, analogous results are not known (in fact, they are false under the specific generic conditions of Mather's Theorems), and therefore it is not possible to rule out the existence of the point $p\in \bd U$ by the same methods.

The use of Theorems \ref{thm:intro-main-gen-NW} and \ref{thm:intro-arclemma} provide a novel approach to the same problem, allowing for a proof that does not require the (generic) existence of homoclinic intersections and thus allowing to obtain a proof on surfaces of any genus.

Let us remark that Theorem \ref{t:mather} was mistakenly used in \cite{xia} and \cite{koro-nassiri} (see \S\ref{sec:xia} and Remark \ref{rem:erratum}), and therefore it fills a gap in the respective articles (however, the gap in \cite{koro-nassiri} can be avoided by using a different result; see \cite{kn-erratum}).

\subsection{Other consequences}

The $\bd$-nonwandering condition also holds whenever the dynamics induced in the circle of prime ends is transitive (see \S\ref{sec:bdnw}). This observation, together with Theorem \ref{thm:intro-main-gen} and a result from \cite{barge2}, leads to the following result which proves a conjecture of Walker \cite{walker}:
\begin{mthm}\label{th:main-walker} Suppose $f\colon S^2\to S^2$ is an orientation-preserving homeomorphism, and $K\subset \R^2$ is a continuum that irreducibly separates the plane into exactly two invariant components. If the dynamics induced in the prime ends from one side of $K$ is conjugate to an irrational rotation by angle $\alpha$, then the prime ends rotation number from the opposite side of $K$ is also equal to $\alpha$.
\end{mthm}
Indeed, if the rotation numbers differ, then the main theorem from \cite{barge2} implies that there are many periodic points in $K$, but since $K$ is the boundary of a $\bd$-nonwandering invariant domain with irrational rotation number, this contradicts Theorem \ref{thm:intro-main-gen}. 

Theorem \ref{th:main-walker} proves Conjecture 2 of \cite{walker}, namely that the prime ends dynamics from the two sides of a \emph{cobasin boundary} (\ie a set such as $K$) cannot correspond to two non-conjugate irrational rotations. 

Another context in which similar ideas can be applied is in the setting of holomorphic maps. Suppose $U\subset \C$ and $f$ is a homeomorphism defined in a neighborhood of $\ol{U}$ and holomorphic in $U$, leaving $U$ invariant, and such that $\rho(f,U)$ is irrational. By means of the Riemann mapping theorem and the Schwartz refleaction principle, one may show that the homeomorphism induced by $f|_U$ on the circle of prime ends is conjugate to that of an analytic diffeomorphism \cite{perez-marco}. This means that the dynamics cannot be Denjoy type, implying automatically that the $\bd$-nonwandering condition holds. Thus, one may apply our results in this setting. In particular, one may obtain purely topological proofs of some of the results from \cite{risler}.

\subsection{Extension of the theory of prime ends} 
Another result introduced in this paper, which is of independent interest, is a natural extension of the theory of prime ends. In Section \ref{sec:prime}, a precise description is given of what happens with the prime ends compactification of an open set when of the underling surface changes. Although this forces us to introduce some cumbersome notations, it is necessary to our proof, and it allows us to consider the prime ends compactification of sets which are not relatively compact (as opposed to what is usually done in the literature).

\subsection{Structure of the article}

This article is organized as follows: Section \ref{sec:prelim} introduces preliminary notations and definitions, and some general results that will be used later. In Section \ref{sec:prime} we formally define the notion of prime ends compactification for arbitrary simply connected sets, and we establish results analogous to the ones in \cite{mather-area} and \cite{mather-caratheodory} about the topology of prime ends. The proofs are given in the Appendix to avoid bloat. We also define the prime ends rotation number and establish certain invariance of this number when the underlying surface is changed. In \S\ref{sec:bdnw} we define precisely the notion of $\bd$-nonwandering, and we show that it is weaker than other conditions, for instance $f$ being nonwandering or preserving a finite area. In \S\ref{sec:CL} we prove the result of Cartwright and Littlewood stated in the beginning of the introduction, but under the $\bd$-nonwandering condition.

Section \ref{sec:arclemma} is devoted to the main technical result of the article, which is Theorem \ref{thm:intro-arclemma}. The proof is by contradiction, and the main idea is that if the theorem does not hold one can construct a large enough disjoint pairs of loops with intersection index $\pm 1$ (contradicting the finite genus). 

Section \ref{sec:converse} contains a proof of Theorem \ref{thm:intro-main}. The main idea of the proof is to combine Theorem \ref{thm:intro-arclemma} with arguments from Brouwer theory which, together with a theorem of Jaulent on maximal unlinked sets, lead to the existence of translation arcs near a fixed point under certain conditions. 

In Section \ref{sec:main-gen} we prove Theorems \ref{thm:intro-main-gen-NW} and \ref{thm:intro-main-gen}, and a version of the same theorems for arbitrary surfaces of finite genus is also stated. The proof is based on applying Theorem \ref{thm:intro-main} on a lift of $f$ to the universal covering; however additional nontrivial arguments are required to deal with the fact that $f$ is not necessarily homotopic to the identity, and to prove the contractibility of the boundary of $U$.

Section \ref{sec:general} generalizes the results from Section \ref{sec:main-gen} to an arbitrary open set $U$: assuming there is a regular end (or ideal boundary point, as defined in \ref{sec:prelim}) of $U$ which is fixed by $f$ and has irrational prime ends rotation number, one obtains information about the \emph{impression} of the given end analogous to the one given for the boundary of $U$ in the simply connected case. The main idea of the proof consists in modifying the surface to reduce the problem to one in which $U$ is simply connected. Theorem \ref{thm:intro-arclemma} is also stated in this context.

Section \ref{sec:mather} is devoted to applications of the main results of this article, proving Theorem \ref{t:mather}, which is in fact a simple corollary of the more general Theorem \ref{th:mather-extended} (see \S\ref{sec:mather-complementary}). Other applications are also mentioned.

Finally, the examples in Section \ref{sec:examples} show that the theorems fail without a $\bd$-nonwandering condition. We also exhibit examples with an irrational prime ends rotation number but exhibiting a fixed point (on the sphere), and an example showing that the number $N_{\alpha, g}$ depends on both $\alpha$ and $g$ in Theorem \ref{thm:intro-arclemma}.

\section{Preliminaries}
\label{sec:prelim}

Let $S$ be a surface (which is always assumed Hausdorff, connected and orientable). When $U\subset S$ and it is clear that $S$ is the underlying surface, we denote by $\ol{U}$ the closure of $U$ and $\bd U$ the boundary of $U$ in $S$. If $W\subset S$ and $U\subset W$, we write $\cl_W U$ and $\bd_W U$ for the closure and boundary (respectively) of $U$ in $W$ with the restricted topology.

A set is called an open (resp. closed) \emph{topological disk} if it is homeomorphic to the open (resp. closed) unit disk in $\R^2$. Similarly, an open (resp. closed) \emph{topological annulus} is a set homeomorphic to the annulus $\mathbb{S}^1\times (0,1)$ (resp. $\mathbb{S}^1\times [0,1]$). 

We consider an arc $\gamma$ to be a continuous map from an interval in $\R$ to $S$, and $\gamma$ is called a simple arc if it is injective, except possibly at its endpoints. If $\gamma$ is an arc, we abuse notation and also denote its image by $\gamma$; it should be clear from the context which one is the case. We say that $\gamma\colon [a,b]\to S$ is a \emph{loop} if its initial point and its end point coincide. We say that a loop is \emph{essential} in $S$ if it is not homotopic (with fixed endpoints) to a point, and an open set $U\subset S$ is essential in $S$ if it contains some loop which is essential in $S$.

A \emph{translation arc} (or \emph{$1$-translation arc}) is a compact simple arc in $S$ joining a non-fixed point $x$ to its image $f(x)$, and such that $f(\gamma)\cap \gamma = \{f(x)\}$ or $\{x, f(x)\}$ (where the latter case holds only when $f^2(x)=x$). More generally, an \emph{$N$-translation arc} is a translation arc such that, additionally, $f^k(\gamma)\cap \gamma=\emptyset$ for $2\leq k\leq N-1$, and $f^N(\gamma)\cap \gamma$ is either empty or $\{x\}$ (the latter holding only when $f^{N+1}(x)=x$).
This is equivalent to saying that the path $\prod_{0\leq k\leq N}f^k(\gamma)$ obtained by concatenation is still a simple arc.  
If $\gamma$ is an $N$-translation arc for every $N\geq 1$, we say that $\gamma$ is an \emph{$\infty$-translation arc}. This is equivalent to saying that one of the arcs (and so every arc) $\prod_{k<0}f^k(\gamma)$, $\prod_{k\geq 0}f^{-k}(\gamma)$, or $\prod_{k\in\Z}f^k(\gamma)$ is a simple arc.

Let us recall a classic result of Brouwer (see, for instance, \cite{fathi}). 

\begin{proposition}[Brouwer's lemma on translation arcs]\label{pro:brouwer-trans} Any translation arc of a fixed-point free orientation preserving homeomorphism of $\R^2$ is an $\infty$-translation arc. 
\end{proposition}

\subsection{Continua}\label{sec:continua}

A \emph{continuum} is a compact connected set with more than one point. 
We say that the continuum $K\subset S$ is 
\begin{itemize}
\item \emph{cellular} if there is a sequence $(D_n)_{n\in \N}$ of closed topological disks such that $\inter(D_{n+1})\subset D_n$ and $K = \bigcap_{n\in \N} D_n$;
\item \emph{annular} if there is a sequence $(A_n)_{n\in \N}$ of closed topological annuli such that $A_{n+1}$ is contained in the interior of $A_n$ and separates its boundary components, and $K=\bigcap_{n\in \N} A_n$;
\item \emph{contractible in $S$} if there is a continuous map $H\colon K\times [0,1]\to S$ and $x_0\in S$ such that $H(x,0) = x$ and $H(x,1) = x_0$ for all $x\in K$.
\end{itemize}

Let us note the following equivalent properties.
\begin{itemize}
\item $K$ is contractible in $S$ if and only if it has a neighborhood $U\subset S$ homeomorphic to an open disk;
\item $K$ is cellular if and only if it is contractible and non-separating;
\item $K$ is annular if and only if it has a neighborhood $U$ homeomorphic to an open annulus such that $U\sm K$ has exactly two components, both essential in $U$. 
\end{itemize}

\subsection{Topological ends (ideal boundary points)}

If $S$ is a non-compact surface, a \emph{boundary representative} of $S$ is a sequence $P_1\supset P_2\supset\cdots$ of connected unbounded (\ie not relatively compact) open sets in $S$ such that $\bd_S P_n$ is compact for each $n$ and for any compact set $K\subset S$, there is $n_0>0$ such that $P_n\cap K=\emptyset$ if $n>n_0$ (here we denote by $\bd_S P_n$ the boundary of $P_n$ in $S$). Two boundary representatives $(P_i)_{i\in \N}$ and $(P_i')_{i\in \N}$ are said to be equivalent if for any $n>0$ there is $m>0$ such that $P_m\subset P_n'$, and vice-versa. The \emph{ideal boundary} $\mathrm{b_I}S$ of $S$ is defined as the set of all equivalence classes of boundary representatives. Denote by $\cIB S$ the space $S\cup \bdIB S$ with the topology generated by open sets of $S$ together with sets of the form $V \cup V'$, where $V$ is an open set in $S$ such that $\bd_S V$ is compact, and $V'$ denotes the set of elements of $\bdIB S$ which have some boundary representative $(P_i)_{i\in \N}$ such that $P_i\subset V$ for all $i\in \N$. We call $\cIB{S}$ the \emph{ends compactification} or \emph{ideal completion} of $S$.

Any homeomorphism $f\colon S\to S$ extends to a homeomorphism $\til{f}\colon \cIB {S}\to \cIB {S}$ such that $\til{f}|_{S} = f$. If $S$ is orientable and has finite genus, then $\cIB{S}$ is a closed orientable surface of the same genus, and $\bdIB S$ is a totally disconnected closed set. See \cite{richards} for details. 

If $U$ is an open subset of a surface $S$, and $\mathfrak{p}\in \bdIB U$ is an ideal boundary point, we define its \emph{impression} in $S$ by 
$$Z(\mathfrak{p})= \bigcap \{ \cl_S (V\cap U) : V\text{ open in } \cIB U, \, \mathfrak{p}\in V\}.$$
Note that $Z(\mathfrak{p})\subset \bd_S U$. 

Assuming $S$ has finite genus and $U\subset S$, we say that $\mathfrak{p}$ is a $\emph{regular}$ ideal boundary point of $U\subset S$ if $\mathfrak{p}$ is isolated in $\bdIB U$ and $Z(\mathfrak{p})$ has more than one point.

Given a regular end $\mathfrak{p}$ of $U$, we say that $A\subset U$ is a $\mathfrak{p}$-collar if $A\cup\{\mathfrak{p}\}\subset \cIB U$ is a closed topological disk.

\subsection{Some general results}

Let us state a result that allows us to make `surgery' arguments in several places. It can be seen as a refined version of Schoenflies' theorem. Its proof can be found in \cite[\S 2]{epstein}. We  give a simplified statement, since we consider boundaryless surfaces only.

\begin{theorem}[Epstein, \cite{epstein}]\label{th:epstein} Let $\gamma_1,\gamma_2\colon \mathbb{S}^1\to S$ be homotopic simple loops in a surface $S$ on which none of the two loops bounds a disk. Then there is an isotopy $(h_t)_{t\in [0,1]}$ from the identity on $S$ to some map $h=h_1$ such that $h\circ \gamma_1 = \gamma_2$, and the isotopy is fixed (pointwise) outside some compact set.
\end{theorem}

We also recall a useful extension theorem, a proof of which can be found in \cite[Theorem 2.1]{leroux-extension} (see also \cite{hamilton}).
\begin{theorem}[Extension Theorem]\label{th:extension} Let $W$ be a connected open subset of $\R^2$ and $H\colon W\to W'\subset \R^2$ a homeomorphism with a unique fixed point $z_0$. Then there is a homeomorphism $H'\colon \R^2\to \R^2$ which coincides with $H$ in a neighborhood of $z_0$ and such that $z_0$ is the unique fixed point of $H'$.
\end{theorem}

Finally, we state a result for future reference. By an \emph{aperiodic} set we mean one which does not contain a periodic point. 

\begin{theorem}[\cite{koro}]\label{th:koro} If $f\colon S\to S$ is a nonwandering homeomorphism of a closed orientable surface and $K\subset S$ is an aperiodic invariant continuum, then $K$ is annular.
\end{theorem}


\section{Prime ends compactifications}
\label{sec:prime}

From now on, whenever we say that $S$ is a surface this will mean that $S$ is a connected orientable surface of finite genus. 

Let us recall some facts and definitions from Carath\'eodory's prime ends theory. For more details of the classic theory, we refer the reader to \cite{mather-caratheodory} and \cite{mather-area}; however, we will need some work to generalize the theory to our context and to find a relation between prime ends compactifications on different spaces. In particular, we define prime ends for non-relatively compact sets and obtain the classic results in that context.

Let $U\subset S$ be a simply connected open subset of the surface $S$ such that $S\sm U$ has more than one point. An \emph{end-cut} of $U$ in $S$ is an arc $\gamma\colon [0,1)\to U$ such that $\lim_{t\to 1} \gamma(t)=x$ for some $x\in \bd U$ (so that $\gamma$ extends continuously to an arc $[0,1]\to \ol{U}$ with $\gamma(1)=x\in \bd U$).  A point $z\in \bd U$ is \emph{accessible} (from $U$) if it is the endpoint of some end-cut in $U$.

A \emph{cross-cut} of $U$ in $S$ is the image of a simple arc $\gamma\colon (0,1)\to U$ that extends to an arc $\ol{\gamma}\colon[0,1]\to \ol{U}$ joining two points of $\bd U$, and such that each of the two components of $U\sm \gamma$ has some boundary point in $\bd U\sm \ol{\gamma}$. 
Note that the endpoints of a cross-cut are accessible points. A \emph{cross-section} of $U$ in $S$ is any connected component of $U\sm \gamma$ for some cross-section $\gamma$ of $U$ in $S$. Note that each cross-cut corresponds to exactly two cross-sections, which are topological disks.

A \emph{chain} for $U$ in $S$ is a sequence $\mc{C} = (D_n)_{n\in \N}$ of cross-sections such that $D_i \subset D_j$ for all $i\geq j\geq 1$ and $\ol{\bd_U D_i} \cap \ol{\bd_U D_j}=\emptyset$ for all $i\neq j$.  If $D$ is any cross-section of $U$, we say that the chain $\mc{C}$ \emph{divides} $D$ if $D_i\subset D$ for all sufficiently large $i$. 
If $\mc{C}'=(D_n')_{\in \N}$ is another chain, we say that $\mc{C}$ \emph{divides} $\mc{C}'$ if for each $n>0$ there is $m$ such that $D_m \subset D_n'$. We say that $\mc{C}$ and $\mc{C}'$ are equivalent if $\mc{C}$ divides $\mc{C}'$ and $\mc{C}'$ divides $\mc{C}$. A chain $\mc{C}$ is called a \emph{prime chain} if $\mc{C}$ divides $\mc{C}'$ whenever $\mc{C}'$ is a chain that divides $\mc{C}$. An equivalence class of prime chains is called a \emph{prime end} of $U$. 

\begin{remark} Our definition of prime end is basically the same as the one in \cite{mather-caratheodory}, except that we use cross-sections instead of arbitrary topological disks. This makes no difference in practice, since one can (as Mather does) easily show that a chain of the more general type is always divided by a chain as defined here, and vice-versa. Moreover, in \cite{mather-caratheodory} it is required that $U$ be relatively compact, but this will not be necessary in our case.
\end{remark}

Denote by $\bdPE(U) = \bdPE(U,S)$ the set of all prime ends of $U$, and $\cPE(U) = \cPE(U,S) = U\cup \bdPE(U)$. For a cross-section $D$ of $U$, we say that the prime end $p\in \bdPE(U)$ \emph{divides} $D$ if some (hence any) chain representing $p$ divides $D$. Denoting by $\iPE_{U,S}(D)$ the set of all prime ends that divide $D$, we can topologize $\cPE(U)$ by defining a basis of open sets consisting of all sets of the form $D\cup \iPE_{U,S}(D)$ for some cross-section $D$ of $U$, together with all open subsets of $U$.

\begin{remark} Note that even though a cross-section is a subset of $U$, its definition depends not only on $U$ but also on the surface $S$: if $S'\subset S$ is an open set containing $U$, some cross-sections of $U$ as a subset of $S$ may cease to be cross-sections when $U$ is regarded as a subset of $S'$ (this happens when an endpoint of the corresponding cross-cut lies outside $S'$). This is why we use the cumbersome notation $\cPE(U,S)$ when there is need to emphasize what the underlying surface is. 
\end{remark}

The next theorem gives a useful relation between $\cPE(U,S)$ and $\cPE(U,S')$, while it also extends what is known in the relatively compact case to the general case. Since its proof is somewhat lengthy and not central to this article, we present its proof in the Appendix.

\begin{theorem}[Prime ends compactification]\label{th:prime-equiv} The following properties hold:
\begin{enumerate}
\item[(1)] $\bdPE(U,S)$ is homeomorphic to a circle $\mathbb{S}^1$, and $\cPE(U,S)$ is homeomorphic to the closed disk $\ol{\D}$. 
\item[(2)] Let $S_0\subset S$ be an open connected set such that $U\subsetneq S_0$. Then the inclusion $i\colon U \to \cPE(U,S_0)$ extends to a monotone surjection $i_*\colon \cPE(U,S)\to \cPE(U,S_0).$
\item[(3)] If $S\sm S_0$ is totally disconnected, then the inclusion $i\colon U\to \bdPE(U,S_0)$ extends to a homeomorphism $i_*\colon \cPE(U,S) \to \cPE(U,S_0)$.
\end{enumerate}
\end{theorem}

\subsection{Accessible prime ends and criteria for primality}  Let us introduce the following notation: if $(K_i)_{i\in \N}$ is a sequence of sets, then $K_i\to x\in S$ as $i\to \infty$ if for each neighborhood $V$ of $x$ there is $i_0$ such that $K_i\subset V$ for all $i\geq i_0$. 
We recall a useful criterion for ``primality'' in the case that $U$ is relatively compact in $S$. The results are contained in \cite[Lemmas 3.1-3.4 and Proposition 3.7]{mather-area}. 

\begin{proposition}\label{pro:prime-criterion} If $U$ is relatively compact in $S$, then
\begin{itemize}
\item[(1)] A chain of $U$ in $S$ is prime if and only if there is $x\in \bd_S U$ and some equivalent chain $(D_i)_{i\in \N}$ such that $\bd_U D_i\to x$ as $i\to \infty$;
\item[(2)] If $\gamma\colon [0,1)\to U$ is an end-cut of $U$ in $S$, then there is a prime end $p\in \bdPE(U,S)$ such that $\gamma(t)\to p$ in $\cPE(U,S)$ as $t\to 1^-$.
\end{itemize}
\end{proposition}

Parts (1) and (3) of Theorem \ref{th:prime-equiv} imply the following version of the previous result in the case that $U$ is not relatively compact:

\begin{proposition}\label{pro:prime-criterion2} The following properties hold:
\begin{itemize}
\item[(1)] A chain of $U$ in $S$ is prime if and only if there is $x\in \cIB(S)$ and some equivalent chain $(D_i)_{i\in \N}$ of $U$ in $S$ such that $\bd_U D_i\to x$ in $\cIB(S)$ as $i\to \infty$;
\item[(2)] If $\gamma\colon [0,1)\to U$ is an end-cut of $U$ in $S$, then there is a prime end $p\in \bdPE(U,S)$ such that $\gamma(t)\to p$ in $\cPE(U,S)$ as $t\to 1^-$.
\end{itemize}
\end{proposition}

We say that $p\in \bdPE(U,S)$ is an \emph{accessible prime end} if there is an end-cut $\gamma\colon [0,1)\to U$ such that $\gamma(t)\to p$ in $\cPE(U,S)$ as $t\to 1^-$. In this case, the limit $z=\lim_{t\to 1^-} \gamma(t)$ in $S$ does not depend on the choice of $\gamma$, but only on the choice of $p$. Indeed, if $(D_i)_{i\in \N}$ is any chain of $U$ in $S$ representing $p$ and such that $\bd_U D_i\to x\in \cIB(S)$ as $i \to \infty$ (given by part (1) of the previous proposition), since $\gamma$ accumulates in $p$ in $\cPE(U,S)$, it must intersect any neighborhood of $p$ in $\cPE(U,S)$ and so $\gamma$ intersects $D_i$ for each $i\in \N$. This implies that $\gamma$ intersects $\bd_U D_i$ if $i$ is large enough, and since $\bd_U D_i\to x$ we conclude that  $x$ is a limit point of $\gamma$, so $z=x$. But $x$ depends only on the chosen cross-section and not on the choice of $\gamma$; thus we conclude that there is a unique possible value of $z$ (and of $x$) for the given $p$.

Therefore, any accessible prime end has an associated accessible point in $\bd U$ which is uniquely determined. Similarly, by  part (2) of the previous proposition, every accessible point $z\in \bd U$ gives rise to at least one (but possibly more than one) accessible prime end, namely the endpoint in $\cPE(U)$ of any end-cut of $U$ which converges to $z$ in $S$.

Note that the previous remarks imply in particular that if $\gamma$ is a cross-cut of $U$ in $S$, then $\cl_{\cPE(U,S)}\gamma$ consists of $\gamma$ together with two points of $\bdPE(U,S)$. The next proposition summarizes some useful properties of prime ends; we omit the proofs since they are simple consequences of the definitions and the previous propositions.

\begin{proposition} \label{pro:prime-all} 
Let $D$ be a cross-section of $U$ in $S$ bounded by the cross-cut $\gamma$ in $U$.
\begin{itemize}
\item[(1)] $D$ contains some end-cut $\eta\colon [0,1)\to D$ with endpoint in $\bd U\sm \ol{\gamma}$.
\item[(2)] Let $a,b\in \bdPE(U,S)$ be the endpoints of $\gamma$ in $\cPE(U,S)$. Then the endpoint of $\eta$ in $\cPE(U,S)$ as $t\to 1^-$ is a prime end $p$ that divides $D$ and is different from $a$ and $b$.
\item[(3)] $a\neq b$, and $\iPE_{U,S}(D)$ is a nonempty open interval of the circle $\bdPE(U,S)$ bounded by $a$ and $b$.
\item[(4)] Accessible prime ends are dense in $\bdPE(U,S)$. 
\item[(5)] Accessible points of $\bd U$ are dense in $\bd U$.
\end{itemize}
\end{proposition}


\subsection{Dynamics and rotation numbers for prime ends}

If $f\colon S\to S$ is an orientation-preserving homeomorphism such that $f(U)=U$, then $f|_U$ maps cross-sections to cross-sections, and it extends to a homeomorphism $f_*\colon \cPE(U,S) \to \cPE(U,S)$. The next result, which is a direct consequence of Theorem \ref{th:prime-equiv}, says that if one considers a smaller connected surface $S_0$ such that $U\subsetneq S_0\subset S$, then the map induced by $f$ in $\cPE(U,S)$ is semi-conjugate to the map induced by $f$ in $\cPE(U,S_0)$.

\begin{corollary}\label{coro:prime-semi} If $S_0$ is an open connected set such that $U\subsetneq S_0\subset S$ and $f(S_0)=S_0$, and if $f_*$ and $f_*'$ denote the extensions of $f|_U$ to $\cPE(U,S)$ and $\cPE(U,S_0)$, respectively, then $i_*f_*=f_*'i_*$, where $i_*\colon \cPE(U,S)\to \cPE(U,S_0)$ is the monotone map extending the inclusion $i:U\mapsto \cPE(U,S_0)$ given by Theorem \ref{th:prime-equiv}.
\end{corollary}

Indeed, note that $if|_U = f|_Ui$ holds trivially, so the previous result for the extensions follows by continuity.

\begin{remark} A simple example of the previous corollary is a map $f\colon \R^2\to \R^2$ that leaves the unit disk invariant and such that $f|_{\mathbb{S}^1}$ has two fixed points: an attractor $N$ and a repellor $S$. If $S_0=S\sm K$, where $K$ is one of the two closed intervals of $\mathbb{S}^1$ limited by $N$ and $S$, then $f|_U$ induces on the circle of prime ends $\bdPE(U,\R^2)$ the same dynamics of $f|_{\mathbb{S}^1}$. However, the map induced by $f|_U$ on $\bdPE(U,S_0)$ corresponds to a circle homeomorphism with a unique fixed point, which is a saddle-node. The semi-conjugation $i_*$ collapses one of the intervals between fixed points to a single point.
\end{remark}

If $f_*$ is the homeomorphism of $\cPE(U,S)$ induced by $f|_U$, then $f_*|_{\bdPE(U,S)}$ is a homeomorphism of the circle, and we may consider its rotation number in the sense of Poincar\'e, which we denote by $\rho(f,U, S)\in \R/\Z$ (or by $\rho(f,U)$ when there is no ambiguity). If two circle homeomorphisms are monotonically semi-conjugate, then they have the same rotation number. Thus from the previous corollary we obtain the following crucial result.

\begin{corollary}\label{coro:prime-rot} If $S_0\subset S$ is open, connected, $f$-invariant, and $U\subsetneq S_0$, then $\rho(f,U,S) = \rho(f,U,S_0)$.
\end{corollary}

\subsection{Prime ends for non-simply connected sets}
\label{sec:prime-general}

Suppose that $U\subset S$ is an open connected set but not necessarily simply connected, and let $\mathfrak{p}\in \bdIB U$ be a regular ideal boundary point of $U$. We may define a $\mathfrak{p}$-cross-cut as any simple arc $\gamma$ contained in a $\mathfrak{p}$-collar which extends to a compact arc $\ol{\gamma}$ joining two different points of $Z(\mathfrak{p})$, and such that each of the two components of $U\sm \gamma$ contains some point of $Z(\mathfrak{p})\sm \ol{\gamma}$ in its boundary. If $\gamma$ is a $\mathfrak{p}$-cross-cut, then at least one of the two components of $U\sm \gamma$ is contained in some $\mathfrak{p}$-collar. Any such component is called a $\mathfrak{p}$-cross-section. Define $\mathfrak{p}$-chain and prime $\mathfrak{p}$-chain using the same definitions as in the previous sections, except that using $\mathfrak{p}$-cross-sections; and define a $\mathfrak{p}$-prime end as an equivalence class of prime $\mathfrak{p}$-chains. 

We denote the family of all $\mathfrak{p}$-prime ends by $\bdPE_{\mathfrak{p}} U$, and we set $\cPE_{\mathfrak{p}} U = U\cup \bdPE_{\mathfrak{p}} U$. As before, the topology on $\cPE_{\mathfrak{p}} U$ is defined by using as a basis of open sets the open sets of $U$ together with sets of the form $V\cup \iPE_{\mathfrak{p}}V$, where $V$ is a $\mathfrak{p}$-cross-section of $U$ and $\iPE_{\mathfrak{p}} V$  is the set of all $\mathfrak{p}$-prime ends which have a representative chain entirely contained in $V$.
Note that if $\mathfrak{p}\in \cIB U$ and if $A$ is a $\mathfrak{p}$-collar, then every $\mathfrak{p}$-prime end has a representative $\mathfrak{p}$-chain entirely contained in $A$.

If $U$ has finitely many ideal boundary points, all of which are regular, then one can do the above process simultaneously for each end of $U$ obtaining a compactification which we denote by $\cPE U$ and is a compact surface with boundary. We use the notation $\bdPE U = \cPE U \sm U = \bigcup_{\mathfrak{p}\in \bdIB U} \bdPE_{\mathfrak{p}} U$. 

The results from the previous section can be naturally extended to this setting. In particular, $\bdPE_{\mathfrak{p}} U$ is homeomorphic to $\mathbb{S}^1$ for each regular ideal boundary point $\mathfrak{p}$. To avoid further technicality, we will only make use of this fact in the case that $U$ is relatively compact, and for such case this is proved in \cite{mather-area} and \cite{mather-caratheodory}.

If $f\colon S\to S$ is an orientation preserving homeomorphism and $\mathfrak{p}$ is a regular ideal boundary point of $U$ fixed by $f$, then $f|_U$ extends to a homeomorphism $f_*$ of $\cPE_{\mathfrak{p}} U$, and we may define the rotation number $\rho(f,\mathfrak{p})$ as the classical rotation number of the circle homeomorphism $f_*|_{\bdPE_{\mathfrak{p}}U}$.

\subsection{The $\bd$-nonwandering condition}\label{sec:bdnw}

The main hypothesis for many of our results is the following (see Proposition \ref{pro:bdnw-gen} for simpler properties that imply this, for example the preservation of a finite area).

\begin{definition}\label{def:bdnw} Let $f\colon S\to S$ be a homeomorphism and $U\subset S$ a simply connected open $f$-invariant set. We will say $U$ is \emph{$\bd$-nonwandering} for $f$ (or $f$ is $\bd$-nonwandering in $U$) if, for every $n\in \N$, there is some compact set $K\subset U$ such that $U\sm K$ does not contain the forward or backward $f^n$-orbit of any $f^n$-wandering cross-section. In other words,  if $D$ is a cross-section of $U$ such that $f^{nk}(D)\cap D=\emptyset$ for all $k\neq 0$, then $f^{nk}(D)$ intersects $K$ for some positive and some negative value of $k$.
\end{definition}

\begin{remark} It can be assumed that the set $K$ is a closed topological disk, by increasing its size. We will make this assumption most of the time.
\end{remark}

\begin{remark} We do not know if Definition \ref{def:bdnw} using only $n=1$ implies the definition as stated here. However, in view of Proposition \ref{pro:bdnw-gen}, this does not make a difference in the usual setting (e.g. nonwandering or area-preserving maps).
\end{remark}
We can extend the definition to the non-simply connected setting as well.

\begin{definition}\label{def:bdnw-2} If $U\subset S$ is a connected open $f$-invariant set and $\mathfrak{p}\in \bdIB U$ a fixed regular end of $U$. We will say $f$ is \emph{$\bd$-nonwandering at $\mathfrak{p}$} if for each $n\in \N$ there is a $\mathfrak{p}$-collar $A\subset U$ such that $A$ does not contain the forward or backward orbit of any $f^n$-wandering $\mathfrak{p}$-cross-section.
\end{definition}

The next proposition implies, in particular, that $f$ is $\bd$-nonwandering in $U$ if $f|_U$ is nonwandering.

\begin{proposition}\label{pro:bdnw-gen} Let $f\colon S\to S$ be a homeomorphism and $U\subset S$ a simply connected open $f$-invariant set. Then $f$ is $\bd$-nonwandering, provided that there is a compact set $K\subset U$ such that any one of the following holds:
\begin{enumerate}
\item[(1)] $U\sm K$ does not contain the forward or backward orbit of any $f$-wandering open set.
\item[(2)] There is an $f$-invariant Borel measure (not necessarily finite) such that $\mu(U\sm K)<\infty$ and $\mu(D)>0$ for any cross-section $D$, or
\item[(3)] $U\sm K$ is contained in the nonwandering set of $f$.
\item[(4)] The dynamics induced by $f|_U$ on $\bdPE(U,S)$ is transitive.
\end{enumerate}
\end{proposition}
\begin{proof}
Suppose that (1) holds and $f$ is not $\bd$-nonwandering. Then there is $n\in \N$ such that any neighborhood of $\bd U$ contains the forward or the backward $f^n$-orbit of some $f^n$-wandering cross-section. Let $K'$ be a compact set such that $K\subset f^k(K')\subset U$ for all $k\in \Z$ with $\abs{k}\leq n$. Then we may assume that there is a cross-section $D_0$ of $U$ such that $f^{kn}(D)_0\cap D_0=\emptyset$ for all nonzero $k\in \Z$, and $f^{kn}(D_0)\subset U\sm K'$ for all integers $k> 0$ (the case $k< 0$ is analogous). 

Note that our choice of $K'$ implies that $f^k(D_0)\subset U\sm K$ for all $k>0$, so by (1) we must have $f^{k_0}(D_0)\cap D_0\neq \emptyset$ for some $k_0\in \N$. Let $D_1=f^{k_0}(D_0)\cap D_0$. Again (1) implies that $f^{k_1}(D_1)\cap D_1\neq \emptyset$ for some $k_1\in \N$. Inductively, we may define  $k_0, k_1, \dots, k_{n}$ and a sequence $\emptyset\neq D_{n+1}\subset D_{n} \cdots \subset D_1\subset D_0$, such that $D_{i+1}=f^{k_i}(D_i)\cap D_i$. We can find $0\leq i<j\leq n$ such that, if $l_i=\sum_{m=0}^{i-1} k_m$, then $l_i=l_j\, (\mathrm{mod}\, n)$, so that $l_j=l_i+mn$ for some integer $m>0$. But it follows from our definition that $D_j\subset f^{l_j-l_i}(D_i)$, and since $D_j$ and $D_i$ are both subsets of $D_0$ we conclude that $f^{mn}(D_0)=f^{l_j-l_i}(D_0)$ intersects $D_0$, contradicting the fact that $f^{km}(D_0)\cap D_0=\emptyset$ for all $k>0$.

If (2) holds, fix $n\in \N$ and let $D$ be an $f^n$-wandering cross-section, then $\{f^{kn}(D):k> 0\}$ is an infinite family of pairwise disjoint sets of the same positive measure, so they cannot be all contained in $U\sm K$, which has finite measure. Similarly for $\{f^{kn}(D):k< 0\}$, proving that $f$ is $\bd$-nonwandering in $U$.

If (3) holds then (1) clearly  holds. Finally, if (4) holds, then the induced dynamics by $f^n$ on $\bdPE(U,S)$ is minimal for any $n\in \N$. This implies that there are no $f^n$-wandering cross-sections at all, so $f$ is $\bd$-nonwandering in $U$.
\end{proof}

The next result will also be helpful in most of our proofs.

\begin{proposition}\label{pro:bdw} Let $f\colon S\to S$ be an orientation-preserving homeomorphism of an orientable surface $S$ of finite genus, $U\subset S$ an open invariant topological disk, and $K\subset S\sm U$ is a closed $f$-invariant set such that $\bd U\sm K\neq \emptyset$. If $f$ is $\bd$-nonwandering in $U$, then $f|_{S\sm K}$ is also $\bd$-nonwandering in $U$. If, in addition, $K$ is totally disconnected, then the converse also holds.
\end{proposition}

\begin{proof} 
The first implication is direct, since a wandering cross-section of $U$ in $S\sm K$ is also a wandering cross-section of $U$ in $S$. For the converse, it suffices to show that if $K$ is totally disconnected, then any wandering cross-section $D$ of $U$ in $S$ contains some cross-section of $U$ in $S\sm K$. But this follows from part (3) of Theorem \ref{th:prime-equiv}.
\end{proof}

\subsection{Cartwright-Littlewood result for rational rotation numbers}\label{sec:CL}

This theorem is essentially proved in \cite{cartwright-littlewood}, but the proof is on the sphere and the hypotheses are slightly different.

\begin{theorem}[Cartwright-Littlewood] Let $f\colon S\to S$ be an orientation-preserving homeomorphism of a surface $S$ of finite genus, and $U\subset S$ an open, connected, relatively compact invariant set with a regular end $\mathfrak{p}$ fixed by $f$. If $\rho(f,\mathfrak{p})=0$ and $f$ is $\bd$-nonwandering at $\mathfrak{p}$, then there is a fixed point in $Z(\mathfrak{p})$.
\end{theorem}
\begin{proof} First note that since $U$ is relatively compact, $\cl_S U=\cl_{\cIB S} U$, so that we may assume that $S$ is closed by working with $\cIB S$ instead of $S$.

Let $(D_i)_{i\in \N}$ be a chain representing a fixed prime end in $\bdPE_{\mathfrak{p}}U$. By Proposition \ref{pro:prime-criterion}, we may choose it such that $\bd D_i \to x$ for some $x\in \bd U$. We will show that $x$ is a fixed point.

Since the chain represents a fixed prime end, for each $j\in \N$ there is $i>j$ such that $f(D_i)\subset D_j$. Note that  $D_i\subset D_j$ as well, because $i>j$. We claim that $\ol{\bd_U D_j}$ intersects $\ol{f^{-1}(\bd_U D_j)}$ if $j$ is large enough. Suppose on the contrary that $\ol{\bd D_j}\cap f^{-1}(\ol{\bd_U D_j})=\emptyset$.  
Since $f^{-1}(D_j)$ and $D_j$ both contain $D_i$, there are two cases: either $D_i\subset f^{-1}(\cl_U D_j)\subset D_j$ or $D_i\subset \cl_U(D_j)\subset f^{-1}(D_j)$. Suppose without loss of generality that $D_i\subset f^{-1}(\cl_U D_j)\subset D_j$. By our assumption, the cross-cut $\bd_U D_j$ has no endpoint in common with $f^{-1}(\bd_U D_j)$. From this, it is easy to obtain a cross-section $E\subset D_j\sm f^{-1}(D_j)$ (as in Figure \ref{fig:max6}), and it follows that $f^{-n}(E)\subset D_j$ for all $n\in \N$. Since $D_j$ lies inside any given $\mathfrak{p}$-collar if $j$ is large enough, this contradicts the $\bd$-nonwandering condition at $\mathfrak{p}$.

We have showed that $\ol{f^{-1}(\bd D_j)}\cap \ol{\bd D_j}\neq \emptyset$ for all sufficiently large $j$. Since $\bd D_j\to x$, we conclude that $f^{-1}(x)=x$, completing the proof.
\end{proof}



\section{A lemma on translation arcs}
\label{sec:arclemma}

The main result of this section is motivated by the following simple result.
\begin{proposition}\label{prop:pre-arc}
Let $S$ be a closed surface, $f\colon S\to S$ an area-preserving homeomorphism and $U\subset S$ an open simply connected $f$-invariant set such that $S\sm U$ has more than one point. Then there is no $\infty$-translation arc intersecting both $U$ and its complement.
\end{proposition}
\begin{proof} If there is an $\infty$-translation arc $\gamma$ intersecting $U$ and $S\sm U$, then there is some sub-arc $\sigma \subset \gamma\cup f(\gamma)$ that is a cross-cut of $U$. If $D_1$ and $D_2$ are the two components of $U\sm \sigma$, the preservation of area implies that the homeomorphism $(x,y)\mapsto (f(x),f(y))$ of $S\times S$ preserves a measure of full support, and therefore the set $D_1\times D_2$ is nonwandering for such map. This implies that there is $n>0$ such that $f^n(D_1)\cap D_1\neq \emptyset$ and $f^n(D_2)\cap D_2\neq \emptyset$, which in turn implies that $f^n(\sigma) \cap \sigma\neq \emptyset$, contradicting the fact that $\sigma \cup f^n(\sigma)$ is contained in the simple arc $\bigcup_{k\in \Z} f^k(\gamma)$. 
\end{proof}

We observe that in this proposition it is essential that the translation arc $\gamma$ intersects both $U$ and $\bd U$. For example, if $U$ is a disk bounded by the stable manifold of a fixed hyperbolic saddle with a homoclinic connection, any simple arc in the homoclinic connection joining a point to its image gives an $\infty$-translation arc contained in $\bd U$. The idea of the next result is to remove the hypothesis that $\gamma$ intersects $U$, by requiring that the rotation number on the prime ends circle of $U$ be nonzero. To do that, we prove a finite version of the previous result, namely that for some $N>0$ there is no $N$-translation arc intersecting both $U$ and its complement. Note that any small perturbation of an $N$-translation arc supported on the complement of a neighborhood of the endpoints of the arc is still an $N$-translation arc, a fact that is no longer true for $\infty$-translation arcs. This is why having a version of the previous proposition for $N$-translation arcs (with the restriction that the rotation number is nonzero) allows us to remove the condition that $\gamma$ intersects $U$.

The proof of the finite version of the previous result is much more delicate, and it is necessary to take into account both the prime ends rotation number and the genus of the surface in the choice of $N$. 

Additionally, one can prove the previous result using not the preservation of area but the fact that the nonwandering set of $f$ contains $U$. We will relax this condition even more by requiring a weak type of nonwandering condition near the boundary of $U$, namely the $\bd$-nonwandering condition (see Definition \ref{def:bdnw}).

The main theorem of this section is Theorem \ref{thm:intro-arclemma} in the introduction:

\begin{theorem} \label{th:arclemma}
Let $S$ be an orientable surface of genus $g<\infty$, $f\colon S\to S$ an orientation preserving homeomorphism, and $U\subset S$ an invariant open topological disk such that $S\sm U$ has more than one point. Suppose that $f$ is $\bd$-nonwandering in $U$, and $\rho(f,U)=\alpha \neq 0$. Then there is $N_{\alpha,g}\in \N$, depending only on $\alpha$ and $g$, and there is a compact set $K\subset U$ such that every $N_{\alpha,g}$-translation arc in $S\sm K$ is disjoint from $\bd U$.
\end{theorem}

\begin{remark} Let us emphasize that $N_{\alpha,g}$ necessarily depends on both $\alpha$ and $g$. See examples \ref{exm1} and \ref{exm2} in Section \ref{sec:examples}.
\end{remark}

Before proceeding to the proof, we state a direct but useful consequence.

\begin{corollary}\label{cor:arclemma} Under the hypotheses of Theorem \ref{th:arclemma}, every $\infty$-translation arc in $S\sm K$ is disjoint from $\bd U$. 
\end{corollary}

\begin{remark} It follows from the proof of Theorem \ref{th:arclemma} that, if one replace the $\bd$-nonwandering condition by the condition that there are no wandering cross-sections (which holds for instance if $f$ preserves area and $U$ has finite area), then there is no $N_{\alpha,g}$-translation arc intersecting $\bd U$ at all (\ie $K=\emptyset$).
\end{remark}

\subsection{Maximal cross-cut lemma}
We begin with a technical lemma before moving to the proof of Theorem \ref{th:arclemma}.

\begin{lemma}\label{lem:uzero} Let $f\colon S\to S$ be a homeomorphism and $U\subset S$ a simply connected open $f$-invariant set such that $S\sm U$ has more than one point and $\rho(f, U)\neq 0$. Suppose $N\geq 2$, and $\gamma$ is an $N$-translation arc with endpoints in $\bd U$ and intersecting $U$, and let $\Gamma = \bigcup_{k=0}^N f^k(\gamma)$. Then
\begin{enumerate}
\item The set $\mc{F}$ of all connected components of $U\cap \Gamma$ consists of cross-cuts of $U$ such that each $\sigma\in \mc{F}$ is disjoint from its image by $f$, and exactly one component of $U\sm \sigma$ is disjoint from its image by $f$. We denote this component by $D_\sigma$, and its closure in $U$ by $\til{D}_\sigma = D_\sigma \cup \sigma$.
\item There is a unique connected component $U_0$ of $U\sm \Gamma$ which is not disjoint from its image by $f$. Moreover, $U\sm U_0 = \bigcup_{\sigma \in \mc{F}} \til{D}_\sigma$, and $U_0$ is open and simply connected.
\item The boundary of $U_0$ in $U$ is a union of elements of a subfamily $\mc{F}^*\subset \mc{F}$, and each point of $\bd_U U_0$ is accessible from $U_0$.
\item Suppose additionally that $\Gamma\subset V$, where $V$ is a neighborhood of $\bd U$ such that $U\sm V$ is connected and intersects its image by $f$, and $V$ does not contain the forward or backward orbit of any wandering cross-section of $U$. Then there is an arc $\sigma \subset \gamma$ such that $f^k(\sigma)\in \mc{F}^*$ for all $k$ with $0\leq k\leq N$.
\end{enumerate}
\end{lemma}

\begin{figure}[ht!]
\centering
\includegraphics[height=4cm]{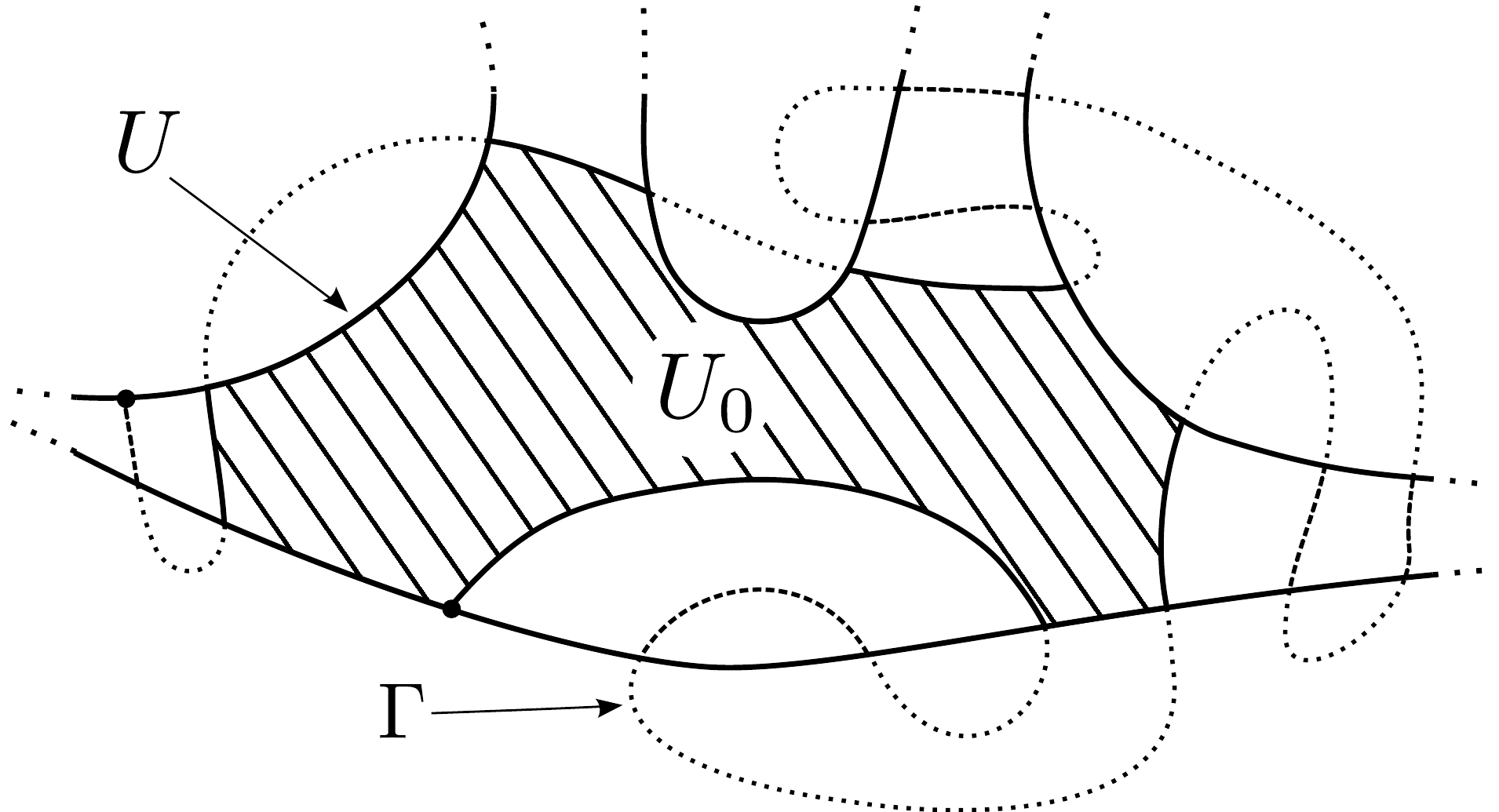}
\caption{The set $U_0$}
\label{fig:uzero}
\end{figure}

\begin{proof}
(1) Let $\mc{F}_0$ be the family of all connected components of $U\cap \gamma$. Since $\gamma$ is a translation arc, $f(\gamma)$ can only intersect $\gamma$ in its endpoints (which are in $\bd U$). Thus, $f(\sigma)$ is disjoint from $\sigma$ for any $\sigma\in \mc{F}_0$. Noting that $\bd U$ is invariant and $\mc{F} = \{f^k(\sigma) : \sigma\in \mc{F}_0,\, 0\leq k\leq N\}$, it follows that $f(\sigma)$ is disjoint from $\sigma$ for any $\sigma\in \mc{F}$. But then, one of the two components of $U\sm\sigma$ is disjoint from the image of $\sigma$. We denote that component by $D_\sigma$. It follows easily that the image of $D_\sigma$ either contains $D_\sigma$ or is disjoint from $D_\sigma$: indeed, the connected set $D_\sigma$ must be contained in one of the two connected components of $U\sm f(\sigma)$, and one of them is $f(D_\sigma)$. But $D_\sigma\subset f(D_\sigma)$ is not possible since $\rho(f,U)\neq 0$. Thus $f(D_\sigma)$ is disjoint from $D_\sigma$. Letting $\til{D}_\sigma = \sigma\cup D_\sigma = \cl_U D_\sigma$, it follows that $U\sm \til{D}_\sigma$ is not disjoint from its image. This proves (1).

(2) We now define a partial order in $\mc{F}$ as follows: $\sigma \prec \sigma'$ if $\til{D}_\sigma \subset D_{\sigma'}$. Note that whenever it makes sense, $f$ preserves the order. 

\begin{claim}\label{claim:max1} If $\sigma$ and $\sigma'$ are different elements of $\mc{F}$ and neither $\sigma \prec \sigma'$ nor $\sigma'\prec \sigma$, then $\til{D}_\sigma$ is disjoint from $\til{D}_{\sigma'}$.
\end{claim}
\begin{proof} 
The connected set $U\setminus D_{\sigma}$ is not included in $D_{\sigma'}$ because this last set is disjoint from its image by $f$ and the first one is not. It is not included in $U\setminus \tilde D_{\sigma'}$ by hypothesis (otherwise one would have $\sigma'\prec \sigma$). One deduces that $\sigma'\cap (U\setminus D_{\sigma})\neq\emptyset$. The boundary of $U\setminus D_{\sigma}$ in $U$ being $\sigma$, which is disjoint from $\sigma'$, one deduces that $\sigma'\cap \til{D}_{\sigma}=\emptyset$. This implies that the connected set $\til{D}_{\sigma}$ must  be contained in one of the two connected components of $U\setminus \sigma'$. As it is not included in $D_{\sigma'}$ (otherwise one would have $\sigma\prec \sigma')$ it is contained in the other one. In other words, it is disjoint from $\til{D}_{\sigma'}$. 
\end{proof}

\begin{claim}\label{claim:max2} For each neighborhood $V$ of $\bd U$ there are at most finitely many elements of $\mc{F}$ that are not contained in $V$. Moreover, there are finitely many $\sigma \in\mc{F}$ such that $D_\sigma$ is not contained in $V$.

\end{claim}
\begin{proof} It follows from the fact that $\Gamma$ is an embedded compact arc and elements of $\mc{F}$ are connected open subsets in the restricted topology of $\Gamma$: the set $\Gamma\sm V$ is compact, and the elements of $\mc{F}$ that are not contained in $V$ form an open covering of $\Gamma\sm V$ (in the restricted topology of $\Gamma$) by pairwise disjoint sets, so that there are finitely many such sets.

To prove the second claim, note that one may choose a neighborhood $V'\subset V$ of $\bd U$ such that $K=U\sm V'$ is connected and $f(K)$ intersects $K$.
If $\sigma\in \mc{F}$ and $\sigma\subset V'$, then one of the components of $U\sm \sigma$ contains $K$. But since $D_\sigma$ is disjoint from its image, it follows that $D_\sigma$ is disjoint from $K$, so it must be contained in $V'$. Thus there are at most finitely many $\sigma\in \mc{F}$ such that $D_\sigma \not\subset V'$.
\end{proof}

\begin{claim}\label{claim:max3} Every element of $\mc{F}$ is contained in a maximal element.
\end{claim}
\begin{proof}
Suppose not. Then one can find a sequence $\{\sigma_n\}_{n\geq 0}$ of elements of $\mc{F}$ such that $\sigma_n\prec \sigma_{n+1}$ for each $n\geq 0$. Let $V$ be a neighborhood of $\bd U$ such that $D_{\sigma_0}\not\subset V$. One deduces that $D_{\sigma_n}\not\subset V$ for any $n\geq 0$. This contradicts Claim \ref{claim:max2}.
\end{proof}

Denote by $\mc{F}^*$ the family of maximal elements of $\mc{F}$, and let
$U_0 = U\sm \bigcup_{\sigma \in \mc{F}} \til{D}_\sigma = U\sm \bigcup_{\sigma \in \mc{F}^*} \til{D}_\sigma$.
Note that if $\sigma$ and $\sigma'$ are different elements of $\mc{F}^*$, then neither $\sigma\prec \sigma'$ nor $\sigma'\prec\sigma$, so Claim \ref{claim:max1} implies that $\{\til{D}_\sigma:\sigma \in \mc{F}^*\}$ are pairwise disjoint, and from Claim \ref{claim:max2} one sees that the union of those sets is closed in $U$. This shows that $U_0$ is open. 
The connectedness and simply connectedness follows easily from the fact that $\til{D}_{\sigma}$ is a simply connected unbounded (\ie not relatively compact) set in $U$. This proves (2).

(3) Let $\sigma_0\in \mc{F}^*$ and $x\in \sigma_0$. From Claim \ref{claim:max2}, if $B$ is a small enough neighborhood of $x$, there are at most finitely many elements $\sigma_1,\dots,\sigma_k$ of $\mc{F}$ intersecting $B$ and different from $\sigma_0$. Since the cross-cuts $\sigma_i$ are closed in $U$ and pairwise disjoint, there is a smaller neighborhood $B'$ of $x$ that is disjoint from $\sigma_1\cup\cdots\cup\sigma_k$. Hence $B'$ is disjoint from all elements of $\mc{F}$ except $\sigma_0$, and since $\sigma_0$ is an embedded arc, there is a smaller neighborhood $B''\subset B'$ such that $\sigma_0\cap B''$ is a connected embedded arc separating $B''$ into exactly two components $B_1$ and $B_2$ (and therefore $x$ is accessible from each component). In particular $B_1$ and $B_2$ are connected and disjoint from $\bigcup_{\sigma\in \mc{F}^*}\sigma$, and both $\bd_U B_1$ and $\bd_U B_2$ contain $x$. Suppose $B_1$ is the component that intersects $D_{\sigma_0}$ (and thus it is contained in $D_{\sigma_0}$). We claim that $B_2\subset U_0$. Indeed, if $B_2$ intersects $\til{D}_\sigma$ for some $\sigma\in\mc{F}^*$, then $B_2\subset D_\sigma$, and $\sigma\neq \sigma_0$. Thus $\til{D}_\sigma$ is disjoint from $\til{D}_{\sigma_0}$, contradicting the fact that $x\in\cl_U B_1 \cap \cl_U B_2 \subset \til{D}_\sigma \cap \til{D}_{\sigma_0}$. Thus, $B_2$ is disjoint from $\til{D}_\sigma$ for all $\sigma\in \mc{F}^*$, implying that $B_2\subset U_0$. 

We have showed that any point of an element $\sigma_0\in \mc{F}^*$ is accessible from $U_0$. This also shows that $\bigcup_{\sigma\in\mc{F}^*}\sigma\subset \bd_U U_0$. But as we already saw, $\bigcup_{\sigma\in \mc{F}^*}\til{D}_\sigma$ is closed, so its boundary (and hence the boundary of its complement $U_0$) is contained in $\bigcup_{\sigma\in \mc{F}^*} \sigma$, and we conclude that $\bigcup_{\sigma\in\mc{F}^*}\sigma = \bd_U U_0$ as claimed.

(4) Denote by $\mc{F}_k$ those elements of $\mc{F}$ of the form $f^k(\sigma)$ with $\sigma\in \mc{F}_0$, and $\mc{F}_k^*=\mc{F}_k\cap \mc{F}^*$. Thus $\mc{F}=\bigcup_{k=0}^N \mc{F}_k$ and $\mc{F}^*=\bigcup_{k=0}^N \mc{F}^*_k$.

\begin{claim}\label{claim:max4} If $\sigma\in \mc{F}_k^*$ with $0\leq k<N$, then either $f(\sigma)\in \mc{F}_{k+1}^*$ or $f(\sigma)\prec \sigma'$ for some $\sigma'\in \mc{F}_0^*$.

Similarly, if $\sigma\in \mc{F}_k^*$ with $0<k\leq N$, then either $f^{-1}(\sigma)\in \mc{F}_{k-1}^*$ or $f^{-1}(\sigma)\prec \sigma'$ for some $\sigma'\in \mc{F}_N^*$.
\end{claim}
\begin{proof} We only prove the first claim; the other is proved with a similar argument.

Suppose $0\leq k<N$ and $\sigma\in \mc{F}_k^*$. If $f(\sigma)$ is not maximal, then $f(\sigma)\prec \sigma'$ for some $\sigma'\in \mc{F}^*_i$, where $0\leq i\leq N$. Suppose that $i>0$. Then $f^{-1}(\sigma')\in \mc{F}$, and so $\sigma = f^{-1}(f(\sigma))\prec f^{-1}(\sigma')$, which contradicts the maximality of $\sigma$. Thus the only possibility is $i=0$, as required.
\end{proof}

\begin{claim} \label{claim:max5} Either $\mc{F}_0^*\neq \emptyset$ or $\mc{F}_N^*\neq \emptyset$.
\end{claim}
\begin{proof} From Claim \ref{claim:max3} we know that $\mc{F}^*$ is nonempty, which means that there exists some maximal element. So some $\mc{F}_k^*$ is nonempty, and the previous claim then implies that either $\mc{F}_0^*$ is nonempty or $\mc{F}_N^*$ is nonempty.
\end{proof}

\begin{claim}\label{claim:max6} Suppose there is no $\sigma\in \mc{F}_0^*$ such that $f^k(\sigma)\in \mc{F}^*_k$ for all $0\leq k\leq N$. 
If $\mc{F}_0^* \neq \emptyset$ (resp. $\mc{F}_N^*\neq \emptyset$), then any neighborhood $V$ of $\bd U$ containing $\Gamma$ contains the forward (resp. backward) orbit of some cross-section of $U$.  
\end{claim}
\begin{proof}  We assume $\mc{F}_0^*\neq \emptyset$ (the other case is analogous, using $f^{-1} $ instead of $f$). For $\sigma\in \mc{F}^*_0$, let $n_\sigma$ be the first positive integer such that $f^{n_\sigma}(\sigma)\notin \mc{F}^*$. From our assumption, we have that $1\leq n_\sigma\leq N$ for each $\sigma\in \mc{F}_0^*$. Moreover, from Claim \ref{claim:max4} we must have $f^{n_\sigma}(\sigma)\prec \sigma'$ for some $\sigma'\in \mc{F}^*_0$. 
Let $$\mc{C}=\{f^k(\sigma) : \sigma \in \mc{F}_0^*,\, 0\leq k<n_\sigma\}\subset \mc{F}^*.$$ 
If $W=\bigcup_{\sigma\in \mc{C}} D_\sigma$, noting that this is a disjoint union, it follows easily from Claim \ref{claim:max2} that $\bd_U W = \bigcup_{\sigma\in \mc{C}} \sigma$. Let us see that $f(W)\subset W$: Any $\sigma\in \mc{C}$ is of the form $\sigma=f^k(\sigma_0)$ for some $\sigma_0\in \mc{F}_0^*$ and $0\leq k< n_{\sigma_0}$. If $k<n_{\sigma_0}-1$, then $f(\sigma)=f^{k+1}(\sigma_0)\in \mc{C}$, so that $f(D_\sigma)=D_{f(\sigma)}\subset W$; and if $k=n_{\sigma_0}-1$ then Claim \ref{claim:max4} implies that there is $\sigma'\in \mc{F}_0^*\subset \mc{C}$ such that $f(\sigma) = f^{n_{\sigma_0}}(\sigma_0)\prec \sigma'$, so that $f(D_\sigma)=D_{f(\sigma)}\subset D_{\sigma'}\subset W$. In both cases, $f(D_\sigma)\subset W$, showing that $f(W)\subset W$.

We also claim that $\cl_U f^N(W)\subset W$. To see this, note that if $\sigma\in \mc{F}_0^*$ then $f^{n_\sigma}(\sigma)\subset W$ and thus $f^n(\sigma)\subset W$ for any $n>n_\sigma$. Since every element of $\mc{C}$ is a positive iterate of some element of $\mc{F}_0^*$, it follows that $f^N(\sigma)\subset W$ for all $\sigma\in \mc{C}$. Recalling that $\bd_U W = \bigcup_{\sigma\in \mc{C}} \sigma$, we conclude that $f^N(\cl_U W)\subset W$ as claimed.

We will find a wandering cross-section of $U$ contained in $W$. To do this, fix $\sigma \in \mc{F}_0^*$, and suppose $D_\sigma$ is nonwandering (otherwise there is nothing to do). Then there is a smallest integer $n>0$ such that $f^n(D_\sigma)$ intersects $D_\sigma$. Since $f^n(D_\sigma)\subset W$, there is $\sigma_0\in \mc{C}$ such that $f^n(D_\sigma)\subset D_{\sigma_0}$, and so $D_\sigma$ intersects $D_{\sigma_0}$. But then from Claim \ref{claim:max1} and the maximality of $\sigma_0$ and $\sigma$, we conclude that $\sigma=\sigma_0$. Thus $f^n(D_\sigma)\subset D_\sigma$. 

Let us show that $n\geq n_\sigma$: if $n<n_\sigma$, then $f^n(\sigma)\in \mc{F}^*$ for all $0\leq k\leq n$, and together with the fact that $D_{f^n(\sigma)}=f^n(D_\sigma)\subset D_\sigma$, this implies that $f^n(\sigma)=\sigma$. But then $f^k(\sigma)\in \mc{F}^*$ for all $k\in \Z$, contradicting our hypothesis. Thus $n\geq n_\sigma$.


Note that from the fact that $\gamma$ is an $N$-translation arc with $N\geq 2$ follows that two different elements of $\mc{F}$ cannot share their two endpoints. Thus, if $\sigma'$ is the element of $\mc{F}_0^*$ such that $f^{n_\sigma}(\sigma)\prec \sigma'$, we know that $f^{n_\sigma}(\sigma)$ and $\sigma'$ do not share their two endpoints, and $f^n(D_\sigma)\subset f^{n-n_\sigma}(D_{\sigma'})$. Denote $\sigma''=f^{n-n_\sigma}(\sigma')$ and $D_{\sigma''}=f^{n-n_\sigma}(D_{\sigma'})$. Since $f^n(D_\sigma)\subset D_{\sigma''}$ and $f^n(D_\sigma)$ intersects $D_\sigma$, it follows that $D_{\sigma''}$ intersects $D_\sigma$. Since $D_\sigma$ is a connected component of $W$ and $D_{\sigma''}\subset W$, it follows that
$$f^n(D_\sigma)\subset D_{\sigma''} \subset D_\sigma.$$

\begin{figure}[ht!]
\centering
\includegraphics[height=5cm]{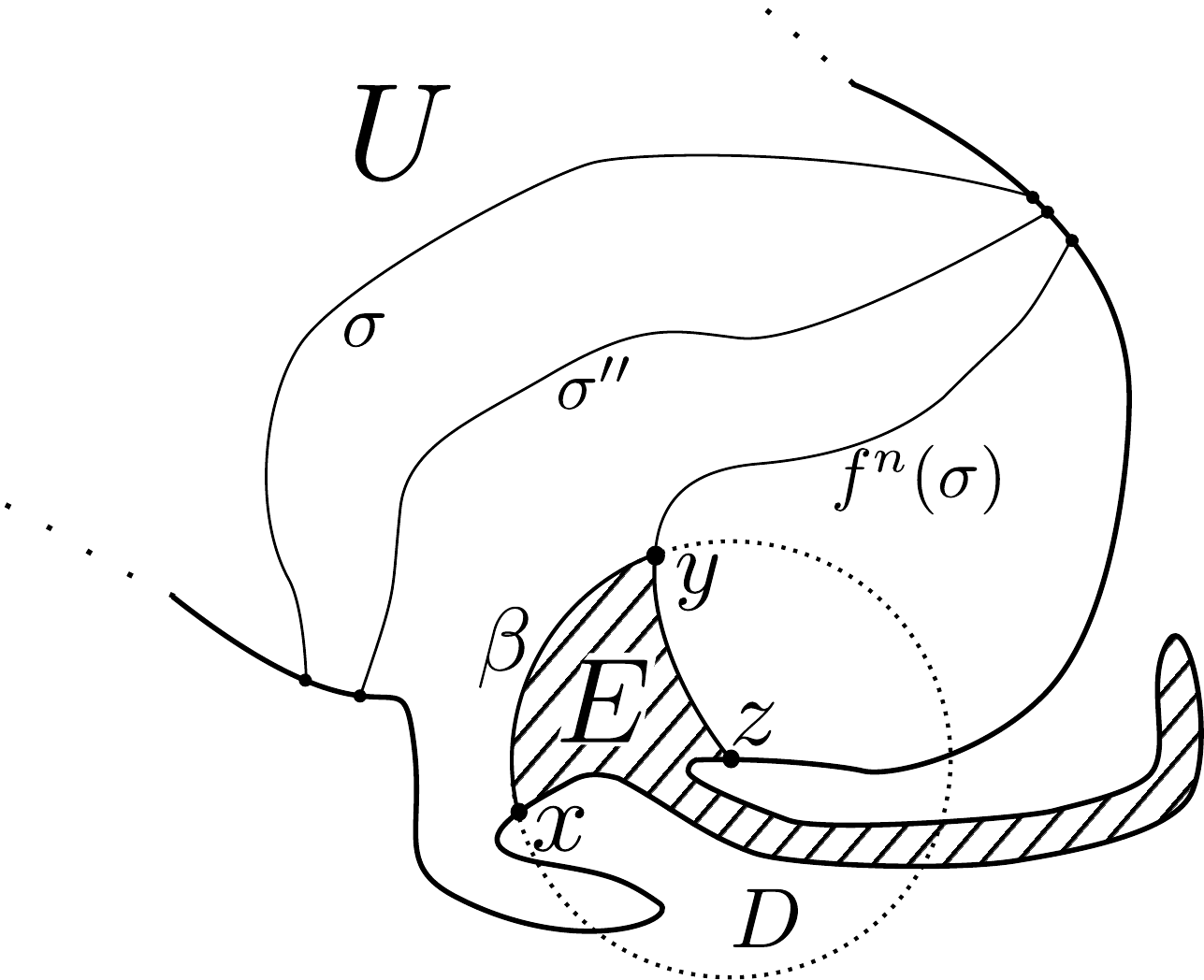}
\caption{Obtaining the wandering cross-section $E$}
\label{fig:max6}
\end{figure}

This implies that $D_\sigma\sm f^{n}(D_\sigma)$ is a wandering set, so to find a wandering cross-section of $U$ in $W$ it suffices to show that $D_\sigma\sm f^{n}(D_\sigma)$ contains some cross-section $E$ of $U$. This will be a consequence of the fact that $f^n(\sigma)$ and $\sigma''$ do not share some endpoint (see Figure \ref{fig:max6}): 
Let $z\in \bd U$ be  an endpoint of $f^{n}(\sigma)$ that is not an endpoint of $\sigma''$. Then we can choose a small disk $D$ around $z$ that is disjoint from $\sigma''$ and from the remaining endpoint of $f^{n}(\sigma)$. Since $D$ necessarily intersects $D_{\sigma''}\sm f^{n}(D_\sigma)$ and it does not contain $U$, the boundary of $D$ contains a simple arc $\beta$ joining a point $x\in \bd U$ to a point $y\in f^{n}(\sigma)$ and such that $\beta\subset D_{\sigma''}\sm f^n(\til{D}_{\sigma})$. The concatenation of $\beta$ with the sub-arc of $f^{n}(\sigma)$ joining $y$ to $z$ gives a cross-cut $\beta'$ of $U$ such that $\beta'\subset D_{\sigma''}\sm f^n(D_{\sigma})$. A simple connectedness argument shows that one of the cross-sections $E$ defined by $\beta'$ in $U$ is contained in $D_{\sigma''}\sm f^n(D_{\sigma})$, as we wanted.

%
%
%

Finally, let $V$ be a neighborhood of $\bd U$ such that $U\sm V$ is connected and not disjoint from its image by $f$, and suppose $\Gamma\subset V$. If $\gamma\in \mc{F}$, then $\gamma$ is disjoint from $U\sm V$ and so $U\sm V$ is contained in one component of $U\sm \gamma$. Since $D_\gamma$ is disjoint from its image, it follows that $U\sm V$ is disjoint from $D_\gamma$. Since this is true for all $\gamma\in \mc{F}$, in particular $U\sm V$ is disjoint from $W$, \ie $W\subset V$. Thus the wandering cross-section that we found earlier is contained in $V$. This proves the claim.
\end{proof}
Part (4) of the lemma follows directly from Claim \ref{claim:max6}, completing the proof of the lemma.
\end{proof}

\subsection{Proof of Theorem \ref{th:arclemma}}

Let us begin by explaining what is $N_{\alpha,g}$. Note that we are not interested in this proof in searching the optimal value of this constant. For every $\xi\in\R/\Z$ write $\Vert \xi\Vert =\inf_{t+\Z=\xi} \vert t\vert$. We will set $N_{\alpha,g}=(2g+1)q+2$, where the integer $q$ is chosen as the smallest positive integer such that
$$\begin{cases} \Vert q\alpha\Vert=0& \text{if $\alpha$ is rational}, \cr   (2g+1)\Vert q\alpha\Vert\leq \Vert\alpha\Vert & \text{if $\alpha$ is irrational.}\end{cases}$$
This choice of $q$ guarantees the following:
\begin{claim*}
If $F$ is an orientation preserving homeomorphism of $\R/\Z$ of rotation number $\alpha$ and if $(\xi_i)_{i\in\Z}$ is an orbit of $F$ that is not periodic, then one can choose an orientation of $\R/\Z$ such that the sequence
\begin{equation}\label{eq:orbit}
(\xi_0, \xi_q, \dots ,\xi_{2gq}, \xi_{(2g+1)q}, \xi_1, \xi_{q+1}, \dots ,\xi_{2gq+1}, \xi_{(2g+1)q+1},\xi_0)
\end{equation} is cyclically ordered in the sense that the open intervals defined by two consecutive terms are nonempty and pairwise disjoint. 
\end{claim*}

This claim should be clear in the irrational case, since the fact that $\alpha$ is irrational implies by a classic result of Poincar\'e that $F$ is monotonically semiconjugate to the rigid rotation by $\alpha$, so that the relative ordering of an orbit of $F$ in the circle coincides with the relative ordering of points of an orbit of the rigid rotation by $\alpha$ (and so the condition on $q$ easily implies that the claim holds using the usual orientation of the circle; see Figure \ref{fig:alpha}).

\begin{figure}[ht]
\begin{minipage}[b]{0.49\linewidth}
\centering
\includegraphics[width=\linewidth]{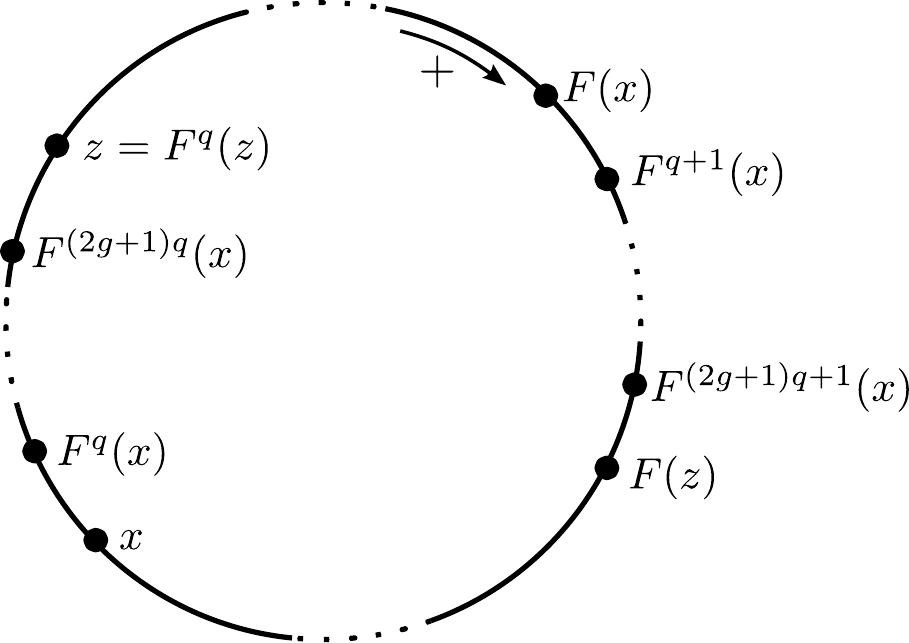}
\caption{Rational case}
\label{fig:alpha-Q}
\end{minipage}
\hfill
\begin{minipage}[b]{0.49\linewidth}
\centering
\includegraphics[width=.87\linewidth]{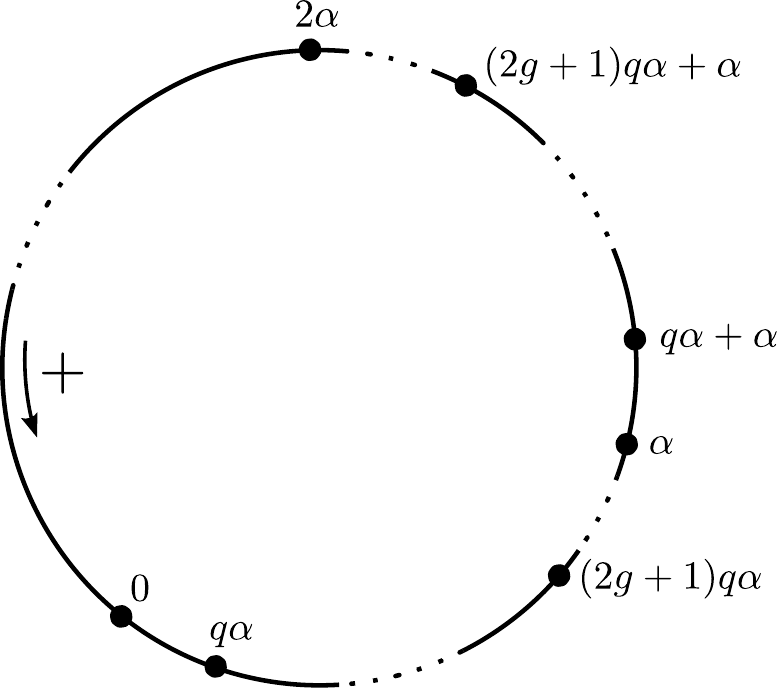}
\caption{Irrational case}
\label{fig:alpha}
\end{minipage}
\end{figure}

On the other hand, if $\alpha =p/q$ is rational (with $q\in \N$ smallest possible), then every periodic point has period $q$ and the non-periodic point $x=\xi_0$ is such that $F^{kq}(x)$ converges monotonically to a periodic point $z$ as $k\to  \infty$. In particular, since $q\geq 2$, the claim is easily seen to hold, by choosing the orientation of the circle depending on the direction of convergence of $f^{kq}(x)$ to $z$ (\ie from the left or from the right). See Figure \ref{fig:alpha-Q}.

Let us suppose that the hypotheses of Theorem \ref{th:arclemma} are satisfied, and let $K_0\subset U$ be a compact set such that $U\sm K_0$ does not contain the forward or backward orbit of any wandering cross-section of $U$. Choose a larger compact set $K$ such that $K_0\subset f^n(K)\subset K$ for $0\leq n\leq N_{\alpha,g}$. We may further assume that $K$ is connected and not disjoint from its image.

From now on, we assume that there exists an $N$-translation arc $\gamma$ in $U\sm K$ such that $\gamma$ meets $\partial U$, where $N\geq N_{\alpha,g}$, and we will derive a contradiction. Let us consider the simple arc $\Gamma=\bigcup_{0\leq n\leq N} f^k(\gamma)$. Observe that every sub-arc $\gamma'$ of $\Gamma$ that joins a point $x\in \gamma$ to $f(x)$ is an $(N-1)$-translation arc. By hypothesis, one may choose $\gamma'$ such that it meets $\partial U$ at a point $x'$ distinct from its endpoints. If $D$ is a topological disk  that contains $x'$ in its interior and that is sufficiently small, then  $f^n(\gamma'\cup D)\cap (\gamma'\cup D)\setminus\{f(x)\}=\emptyset$ for every positive $n\leq N-1$. This implies that every arc $\gamma''$ that joins $x$ to $f(x)$ and is included in $\gamma'\cup D$ is an $N-1$ translation arc. As $x'$ belongs to $\partial U$, one may construct such an arc $\gamma''$ that contains $x'$ but also meets $U$ inside $D$. Let us consider $\Gamma''=\bigcup_{0\leq n\leq N-1} f^k(\gamma'')$. The sub-arc of $\Gamma''$ that joins $x'$ to $f(x')$ is an $(N-2)$-translation arc $\gamma'''$ that meets $U$ and whose ends are on $\partial U$. Let us define $\Gamma'''=\bigcup_{0\leq n\leq N-2} f^k(\gamma''')$.

Let us apply Lemma \ref{lem:uzero} to $\Gamma'''$ by keeping the notations. Note that the hypothesis of part (4) of that lemma holds using $V=U\sm K$ and $N-2$ instead of $N$. Thus one may find a cross-cut $\sigma\in{\mathcal F}^*$ such that $f^k(\sigma)\in{\mathcal F}^*$ for all $k$ with $0\leq k\leq N-2$. 

Since $\sigma$ is a cross-cut of $U$, its closure in $\cPE(U)$ is $\sigma \cup \{\xi_0, \eta_0\}$, where $\xi_0, \eta_0$ are prime ends of $U$. Let $\til{f}$ be the extension of $f|_U$ to $\cPE(U)$. We claim that $\xi_0$ is not $\til{f}$-periodic. To see this, assume that $\xi_0$ is periodic, and let  $q$ be the smallest positive integer such that $\til{f}^q(\xi_0)=\xi_0$. Then  $\alpha = \rho(\til{f}|_{\bdPE(U)})=p/q\, (\text{mod }\Z)$ for some $p\in \Z$, and so by our definition, $N\geq N_{\alpha, g} = (2g+1)q+2 \geq q+2$. Since $\xi_0$ is an accessible prime end, it has a corresponding accessible point $\xi_0'$ in $\bd U$. The fact that $\til{f}^q(\xi_0)=\xi_0$ implies that $f^q(\xi_0')=\xi_0'$. But $\xi_0'$ is an endpoint of $\sigma$, which is contained in the compact arc $\gamma'''$, so $\xi_0'\in \gamma'''$. Since $\gamma'''$ is an $(N-2)$-translation arc and $2\leq q\leq N-2$, it follows that $f^{q}(\gamma''')\cap \gamma'''=\emptyset$. This contradicts the fact that $\xi_0'\in \sigma$.

Thus $\xi_0$ is not periodic, and the claim at the beginning of the proof applied to $F=\til{f}|_{\bdPE(U)}$ implies that if $\xi_k = \til{f}^k(\xi_0)$, the sequence (\ref{eq:orbit}) is cyclically ordered in the circle $\bdPE(U)$ (using an appropriate orientation). If $U_0\subset U$ is the set from Lemma \ref{lem:uzero}, then $f^i(\sigma)\subset \bd U_0$ for $0\leq i\leq N-2$. Using these facts (and noting that $N-2 \geq g$), one may choose two sequences of arcs $(a_i)_{0\leq i\leq g}$ and , $(b_i)_{0\leq i\leq g}$ with the following properties (see Figure \ref{fig:arcs-ab}): 
\begin{itemize}
\item $a_i$ and $b_i$ are in $U_0$ except for their endpoints;
\item $a_i$ joins a point of $f^{2iq}(\sigma)$ to a point of $f^{(2i+1)q}(\sigma)\cap U$;
\item $b_i$ joins a point of $f^{2iq+1}(\sigma)$ to a point of $f^{(2i+1)q+1}(\sigma)\cap U$;
\item $a_i\cap a_j=\emptyset = b_i\cap b_j$ if $i\neq j$, and $a_i\cap b_j=\emptyset$ for any $i,j$;
\end{itemize}

\begin{figure}[ht!]
\centering
\includegraphics[height=5.5cm]{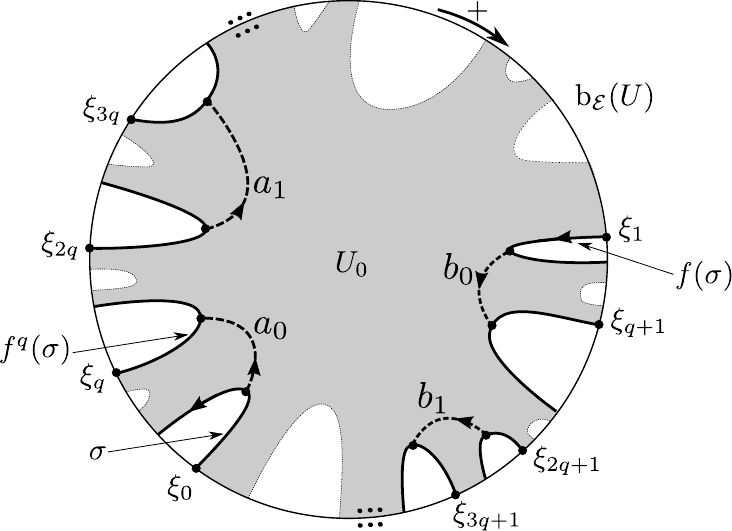}
\caption{The arcs $a_i$ and $b_i$ viewed in $\cPE(U)$}
\label{fig:arcs-ab}
\end{figure}


The arc $\Gamma'''$ is linearly ordered by its orientation. Recall that $\sigma$ is a sub-arc of $\gamma'''$, so that $\sigma$ lies between $x'$ and $f(x')$ in the ordering of $\Gamma'''$, and the sequence of arcs $$\left(\sigma, f(\sigma), f^q(\sigma), f^{q+1}(\sigma), \dots, f^{(2g+1)q}(\sigma), f^{(2g+1)q+1}(\sigma)\right)$$ are pairwise disjoint and increasingly ordered as sub-arcs of $\Gamma'''$; furthermore, they are all oriented in the same way (either positively or negatively) with respect to the orientation of $\Gamma'''$. We may close $a_i$ to a simple loop $a_i''$ by letting $a_i''=a_i\cup a_i'$, where $a_i'\subset \Gamma'''$ is a simple arc joining the endpoint of $a_i$ to its starting point (oriented accordingly). Similarly, we define a simple loop $b_i''=b_i\cup b_i'$ where $b_i'\subset \Gamma'''$ is a simple arc joining the endpoint of $b_i$ to its initial point and oriented accordingly. The loops $a_0'',\dots, a_g''$ are pairwise disjoint, and so are the loops $b_0'',\dots, b_g''$. Moreover, if $i\neq j$ then $a_i''\cap b_j''=\emptyset$, while $\emptyset\neq a_i''\cap b_i''\subset \Gamma'''$ and $\emptyset\neq a_i'\cap b_i'\subset \Gamma'''$.

\begin{figure}[ht!]
\centering
\includegraphics[width=\linewidth]{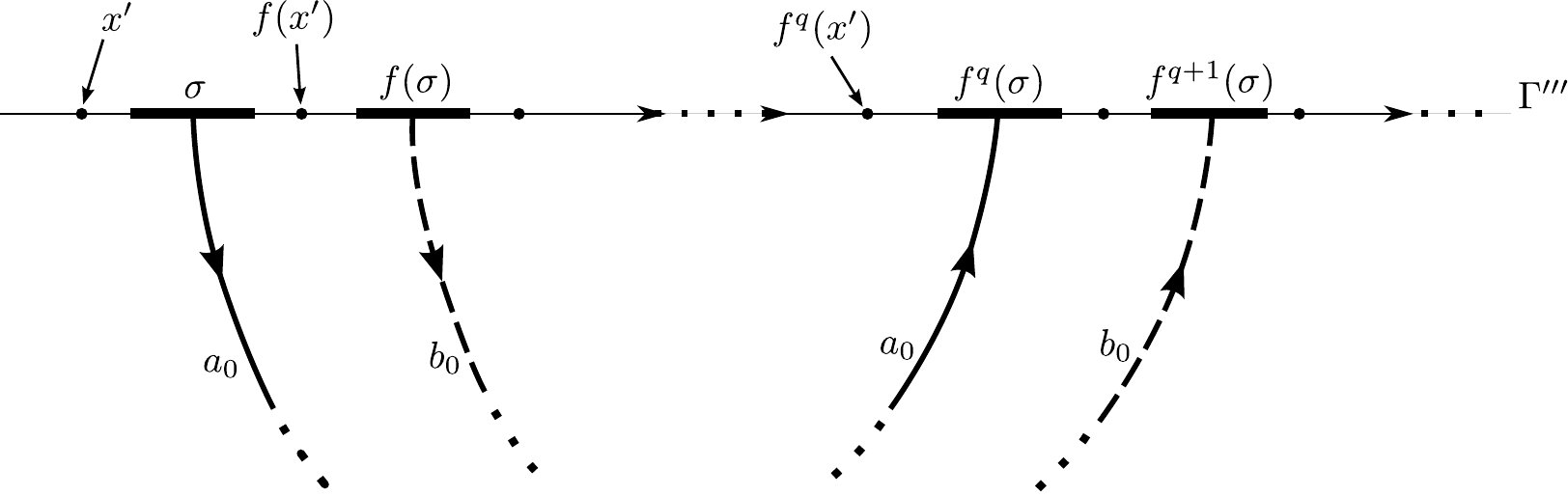}
\caption{The arcs $a_i$ and $b_i$ viewed in $\Gamma'''$}
\label{fig:gamma3}
\end{figure}

Note that by construction, the intersection index of $a_i$ with $\Gamma'''$ at the initial point of $a_i$ is opposite to the corresponding index at the final point of $a_i$, and $a_i$ does not intersect $\Gamma'''$ anywhere else (and similarly with $b_i$). This means that in a neighborhood of $\Gamma'''$, the curves $a_i$ and $b_i$ lie on the same side of $\Gamma$ (see Figure \ref{fig:gamma3}). From these facts one concludes that the algebraic intersection number of $a_i''$ and $b_j''$ is $0$ if $i\neq j$ and $\pm 1$ if $i=j$. This implies that one can construct $g+1$ pairwise disjoint $1$-punctured tori $T_0,\dots, T_g$ such that $a_i''\cup b_i'' \subset T_i$ for each $i$, contradicting the fact that $S$ has genus $g$.

A more algebraic way of saying the same thing is to note that, if $[a_i'']$ and $[b_i'']$ denotes the respective homology classes of $a_i''$ and $b_i''$ in $H_1(S,\Q)$ and $[a_i'']\wedge [b_i'']$ denotes the intersection pairing, then $[a_i'']\wedge[b_j'']$ is $0$ if $i\neq j$ and $\pm 1$ if $i=j$, while $[a_i'']\wedge [a_j''] = 0 = [b_i'']\wedge [b_j'']$ for $0\leq i\leq j \leq g$. It follows that $\{[a_0''],\dots, [a_g''], [b_0''],\dots, [b_g'']\}$ is linearly independent in $H_1(S, \Q)$, concluding that $\dim H_1(S, \Q) \geq 2g+2$. 
But since $g$ is the genus of $S$ and $S$ is orientable, one should have $\dim H_1(S,\Q) = 2g$. This contradiction completes the proof of the theorem.
\qed

\section{Converse of a theorem of Cartwright and Littlewood}
\label{sec:converse}

The main theorem of this section is the following:

\begin{theorem}\label{th:main}
Let $f\colon \R^2\to \R^2$ be an orientation preserving homeomorphism and $U\subsetneq \R^2$ an open $f$-invariant simply connected set. If $f$ is $\bd$-nonwandering in $U$ and $\rho(f,U)\neq 0$, then there is no fixed point of $f$ in the boundary of $U$. 
\end{theorem}

Before proceeding to the details, we mention an easy consequence:
\begin{corollary}\label{coro:main}  Under the hypotheses of the above theorem, if $U$ is unbounded then $\fix(f)\subset U$.
\end{corollary}

Note that the two results above are precisely Theorem \ref{thm:intro-main} in the introduction.

\subsection{Maximal unlinked sets and Brouwer theory} Let us state some results and definitions that we need to prove Theorem \ref{th:main}.

If $f\colon S\to S$ is a homeomorphism isotopic to the identity of an orientable surface, let us say that a set $X\subset \fix(f)$ is \emph{unlinked} if $X$ is closed and $f|_{S\sm X}\colon S\sm X\to S\sm X$ is isotopic to the identity. We say that $X$ is \emph{maximal unlinked} if it is unlinked and there is no unlinked set of fixed points containing $X$ properly.

\begin{remark}\label{rem:unlinked}
If $S=\R^2$ and $f$ preserves orientation, then any set consisting of two or less fixed points of $f$ is unlinked. To see this, note that every orientation preserving homeomorphism of $\R^2$ is isotopic to the identity \cite{kneser}, and let $(f_t)_{t\in[0,1]}$ be an isotopy from $\id$ to $f$. For every fixed point $z$, one gets an isotopy  $(f'_t)_{t\in[0,1]}$ that fixes $z$, by defining $f'_t=g_t\circ f_t$, where $g_t$ is the affine translation that sends $f_t(z)$ on $z$. If $z'$ is another fixed point,  one gets an isotopy  $(f''_t)_{t\in[0,1]}$ that fixes $z$ and $z'$ by defining $f''_t=h_t\circ f_t$, where $h_t$ is the affine map that sends $f_t(z)$ on $z$ and $f_t(z')$ on $z'$.
\end{remark}

The next theorem is due to O. Jaulent. We only state the part that we will use.
\begin{theorem}\cite{jaulent}\label{th:jaulent} Let $S$ be an orientable surface, $f\colon S\to S$ a homeomorphism isotopic to the identity, and $X_0$ is an unlinked set of fixed points of $f$. There exists a maximal unlinked set $X\subset \fix(f)$ containing $X_0$. Moreover, if $(f_t)_{t\in [0,1]}$ is an isotopy from $\id_{S\sm X}$ to $f|_{S\sm X}$ and $z\in \fix(f)\sm X$, then the loop $(f_t(z))_{t\in [0,1]}$ is homotopically nontrivial in $S\sm X$. 
\end{theorem}

We will also use the following classic lemma.

\begin{lemma}\label{lem:gettrans}  Let $h\colon \R^2\to \R^2$ be an orientation preserving homeomorphism, and $\gamma_0$ a simple arc joining a point $z$ to its image $h(z)$ and containing no fixed point of $h$. Then any neighborhood of $\gamma_0$ contains a translation arc $\gamma$ for $h$ such that $z\in \gamma$. 
\end{lemma}
\begin{proof} Let $V$ be a given neighborhood of $\gamma_0$. Reducing $V$, we may assume that it is simply connected and contains no fixed point of $h$. Choose a homeomorphism $\phi\colon \D\to V$ such that $\phi(0)=z$. For $t\in [0,1]$ define $V_t = \phi(B_t(0))$, where $B_t(0)$ is the open disk of radius $t$ centered at the origin. If $\gamma_0$ is not already a translation arc, we may choose $\tau$ to be the smallest number such that $\ol{V}_\tau$ intersects $\ol{h(V_\tau)}$. Clearly $0<\tau<1$, and $V_\tau\cap h(V_\tau) =\emptyset$. This means that there is $w\in \bd V_\tau$ such that $h(w)\in \bd V_\tau$, and moreover $w\neq h(w)$. Since $\ol{V}_\tau$ is homeomorphic to a disk and $z\in V_\tau$, we can choose a simple arc $\gamma$ in $V_\tau$ joining $w$ to $h(w)$ and going through $z$. It is clear from the defintion that $h(\gamma)\cap \gamma\subset \{w,h(w)\}$, which means that $\gamma$ is a translation arc.
\end{proof}

The next key lemma is a consequence of Brouwer's lemma on translation arcs (see, for instance, \cite{fathi}): 

\begin{lemma}\label{lem:brouwerlink} Let $h\colon \R^2\to \R^2$ be an orientation preserving homeomorphism, and $\gamma$ a translation arc for $h$. Then for each $N\in \N$, either $\gamma$ is an $N$-translation arc, or $\bigcup_{k=0}^N h^k(\gamma)$ contains a loop that is homotopically nontrivial in $\R^2\sm \fix(h)$.
\end{lemma}
\begin{proof} Suppose $\gamma\colon [0,1]\to \R^2$ joins $x$ to $h(x)$, let $W$ be the connected component of $M\sm \fix(h)$ containing $x$. Since $\gamma$ is a translation arc, it is disjoint from $\fix(f)$, so it is contained in $W$. This implies that $W$ is invariant. Let $\til{\pi}\colon \til{W}\to W$ be its universal covering map. Choose $\til{x}\in \til{\pi}^{-1}(x)$, let $\til{\gamma}$ be the lift of $\gamma$ with $\til{\gamma}(0)= \til{x}$, and let $\til{h}\colon \til{W}\to \til{W}$ be the lift of $h|_W$ such that $\til{h}(\til{x})=\til{\gamma}(1)$.

Then clearly $\til{\gamma}$ is a translation arc for $\til{h}$, and since $\til{W}\simeq \R^2$ and $\til{h}$ has no fixed points, Proposition \ref{pro:brouwer-trans} implies that $\til{\gamma}$ is an $\infty$-translation arc for $\til{h}$. Let $\til{\Gamma}_n$ be the concatenation of $\til{\gamma},\til{h}(\til{\gamma}),\dots, \til{h}^N(\til{\gamma})$. If $\til{\pi}$ is injective on $\til{\Gamma}_N$, then it is clear that $\gamma$ is an $N$-translation arc for $h$. Otherwise, there is a sub-arc $\til{\sigma}\subset \til{\Gamma}_n$ that joins two different points of $\til{\pi}^{-1}(z)$, for some $z\in W$. This implies that $\til{\pi}(\til{\sigma})$ is a loop that is  homotopically nontrivial in $W$ (and thus in $\R^2\sm \fix(h)$, as claimed).
\end{proof}

\subsection{Accessible fixed prime ends}

Before proceeding to the proof of Theorem \ref{th:main} in its general form, we state a result which does not require the $\bd$-nonwandering condition.

\begin{proposition}\label{pro:main-acc} Let $f\colon \R^2\to \R^2$ be an orientation preserving homeomorphism and $U\subsetneq \R^2$ a simply connected $f$-invariant set. If there is an \emph{accessible} fixed point in $\bd U$, then $\rho(f,U)=0$.
\end{proposition}

\begin{remark} The proposition above is essentially Corollary 2 of \cite{cartwright-littlewood}, if one translates the statement to their setting. Here, we present a short direct proof for the sake of completeness.

\end{remark}

\begin{proof}[Proof of Proposition \ref{pro:main-acc}]
Suppose there is an accessible fixed point $z_1\in \bd U$, so there is an arc $\gamma\colon [0,1)\to U$ such that $\gamma(t)\to z_1$ in $\R^2$ as $t \to 1^-$. Let $\til{z}_1\in \cPE(U)$ be the prime end such that $\gamma(t)\to \til{z}_1$ in $\cPE(U)$ as $t\to 1^-$, and denote by $\til{f}$ be the extension of $f|_U$ to $\cPE(U)$. Assume for contradiction that $\rho(f,U)\neq 0$. Then $\til{f}(\til{z}_1)\neq \til{z}_1$. Letting $\gamma'(t) = f(\gamma(t))$ we obtain an arc in $U$ such that $\gamma'(t) \to z_1$ in $\R^2$ and $\gamma'(t)\to \til{f}(\til{z}_1)\neq \til{z}_1$ in $\cPE(U)$ as $t\to 1^-$. By reducing $\gamma$, we may assume that $\gamma$ and $\gamma'$ are disjoint, and by joining the initial point of $\gamma$ to the initial point of $\gamma'$ by a simple arc $\eta$ in $U$ that is disjoint from $\gamma$ and $\gamma'$ except at its endpoints, we obtain a simple arc $\sigma = \gamma\cup \eta \cup \gamma'$ in $U$ that extends in $\cPE(U)$ to an arc joining $\til{z}_1$ to $\til{f}(\til{z}_1)$. 
See Figure \ref{fig:access}.

\begin{figure}[ht!]
\centering
\includegraphics[width=\linewidth]{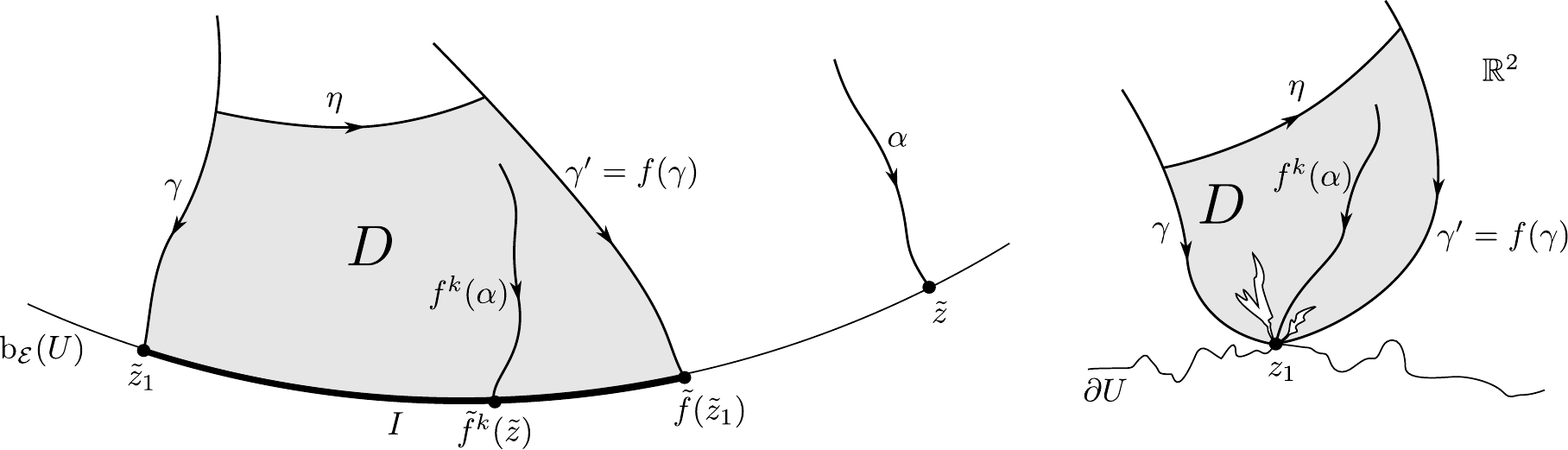}
\caption{View in $\cPE(U)$ (left) and $\R^2$ (right)} 
\label{fig:access}
\end{figure}

Let $D\subset \R^2$ be the open disk bounded by $\sigma\cup \{z_1\}$. 
The boundary of $D\cap U$ in $\cPE(U)$ consists of $\sigma$ together with one of the two open intervals $I$ of $\bdPE(U)$ determined by the points $\til{z}_1$ and $\til{f}(\til{z}_1)$ (moreover, $D$ is a neighborhood of $I$ in $\bdPE(U)$).  Since $I$ is a fundamental domain of the map $\til{f}|_{\bdPE(U)}$, which has nonzero rotation number, it follows that there is $N>0$ such that for any prime end $\til{z}\in \bdPE(U)$ there is an integer $k$ such that $0\leq k\leq N$ and either $\til{f}^k(\til{z})\in I$ or $\til{f}^k(\til{z})=\til{z}_1$.

 Let $z\neq z_1\in \bd U$ be an accessible point, and $\til{z}$ an accessible prime end associated to $z$, so that there is an arc $\alpha\colon [0,1)\to U$ such that $\alpha(t)\to z$ in $\R^2$ and $\alpha(t)\to \til{z}$ in $\cPE U$ as $t\to 1^-$. Then as we just mentioned, there is $k$ with $0\leq k\leq N$ such that either $\til{f}^k(\til{z})\in I$ or $\til{f}^k(\til{z})=\til{z}_1$. Let us first show that the latter case does not hold: if $\til{f}^k(\til{z})=\til{z}_1$, since the accessible point associated to $\til{z}_1$ is precisely $z_1$, it follows that $f^k(\alpha(t))\to z_1$ (in $\R^2$) as $t\to 1^-$, and since $z_1$ is fixed by $f$ it follows that $\alpha(t)\to z_1$ as $t\to 1^-$, contradicting the assumption that $z\neq z_1$.

Thus we must have $\til{f}^k(\til{z})\in I$. Since $f^k(\alpha(t))=\til{f}^k(\alpha(t))\in D$ if $t<1$ and is close enough to $1$, it follows that $f^k(z)$, which is the limit of $f^k(\alpha(t))$ in $\R^2$ as $t\to 1^-$, belongs to $\ol{D}\cap \bd U$. Moreover, since $f^k(z)\neq z_1$, and $z_1$ is the only element of $\bd{D}\cap\bd U$, we conclude that $f^k(z)\in D$. Since $0\leq k \leq N$, it follows that $z\in \bigcup_{k=0}^N f^{-k}(D)$. Let $W = \bigcup_{k=0}^N f^{-k}(D)$. We just showed that every accessible point of $\bd U$ other than $z_1$ belongs to $W$. Since accessible points are dense in $\bd U$, it follows that $\bd U\subset \ol{W}$. Thus, if $\mc{O}$ denotes the unbounded connected component of $\R^2\sm \ol{W}$, we conclude that $U$ is disjoint from $\mc{O}$. But $\bd\mc{O}\subset \bd W\subset U\cup \{z_1\}$ (because $\bd{D}\subset U\cup \{z_1\}$), and clearly $\bd\mc{O}\neq \{z_1\}$, so $U\cap \bd\mc{O}\neq \emptyset$, contradicting the fact that $U\cap \mc{O}=\emptyset$. This completes the proof.
\end{proof}

\subsection{Proof of Theorem \ref{th:main}}
The proof is by contradiction. Assume from now on that there is a fixed point $z_1\in \bd U$.
We begin by reducing the problem to a simpler setting.

\setcounter{claim}{0}
\begin{claim} We may assume that $f$ has a unique fixed point $z_0$ in $U$.
\end{claim}
\begin{proof}
To see this, we will find a map $f'$ that coincides with $f$ in a neighborhood of $\bd U$ (and hence satisfies the hypotheses of the theorem) but has a unique fixed point in $U$. 

Note that since $\rho(f,U)\neq 0$, the extension of $f$ to the prime ends compactification of $U$ has no fixed point in the circle of prime ends.
This implies that $f$ has at least one fixed point $z_0$ in $U$, by Brouwer's fixed point theorem applied to $\cPE(U,S)$. Moreover, one can find a closed topological disk $\ol{D}_0\subset U$ whose interior $D_0$ contains all the fixed points of $f|_U$. Let $\ol{D}_1\subset U$ be a closed disk containing $D_0\cup f(D_0)$.

Since $U$ is homeomorphic to $\R^2$, we may consider the one-point compactification $U_* = U\sqcup\{z_*\}\simeq \mathbb{S}^2$, in which $f|_U$ induces a homeomorphism $f_*$ by fixing $z_*$. Let $M = U_*\sm \{z_0\}$, which is homeomorphic to $\R^2$. Since $W=U_*\sm \ol{D}_1$ is a connected neighborhood of $z_*$, and $z_*$ is the unique fixed point of $f_*|_M$ in $W$, we see from the Extension Theorem \ref{th:extension} that there is a map $F\colon M\to M$ which fixes $z_*$, coincides with $f_*$ in some neighborhood $V$ of $z_*$ (in $M$), and has no other fixed point in $M=U_*\sm \{z_0\}$. Note that since $V$ is a neighborhood of $z_*$, the set $K=M\sm V$ is a compact subset of $U$. The map $F$ extends to a homeomorphism $F_*\colon U_*\to U_*$ by fixing $z_0$, and so $F_*|_U$ is a homeomorphism which coincides with $f$ in $U\sm K$ and has $z_0$ as its unique fixed point. In particular defining $f'(z) = F_*(z)$ for $z\in K$ and $f(z)$ for $z\in S\sm K$ we have  map with the required properties.
\end{proof}

Thus from now on we assume that $z_0\in U$ is the unique fixed point of $f$ in $U$. Recall that we are assuming that there is a fixed point in $\bd U$. We will consider two separate cases.

\medskip

\noindent \textbf{Case I.} Every connected component of $\fix(f)\sm\{z_0\}$ is unbounded.

\smallskip

In this case, if $V$ is the connected component of $(\R^2\sm \fix(f))\cup \{z_0\}$ containing $z_0$, then $V$ is a simply connected invariant set whose boundary consists of fixed points of $f$, and $U\subset V$. 
Consider the prime ends compactification $\mathrm{c}_{\mc{E}}V$. Since any accessible point from $V$ of $\bd V$ is fixed by $f$, Proposition \ref{pro:main-acc} implies that the extension $\til{f}$ of $f|_V$ to $\cPE V$ has a fixed point in the prime ends circle $\bdPE V$. Since $\bd U\cap \fix(f)$ is not empty (by our assumption), it follows that $\cl_{\cPE V} U\cap \bdPE V\neq \emptyset$. Since $\til{f}|_{\bdPE V}$ is an orientation preserving circle homeomorphism with a fixed point, every orbit converges to a fixed point, so the closed, nonempty, $\til{f}|_{\bdPE V}$-invariant set $\cl_{\cPE V} U\cap \bdPE V$ contains a fixed point of $\til{f}$. Thus we may choose $\til{z}_1\in \cl_{\cPE V} U\cap \bdPE V$ such that $\til{f}(\til{z}_1)=\til{z}_1$. 

By Theorem \ref{th:arclemma}, there is $N>0$ and a compact set $K\subset U$ such that there is no $N$-translation arc contained in $\R^2\sm K$ and intersecting $\bd U$. Note that $f$ has no fixed points in $V$ other than $z_0$, so by increasing the size of $K$ we may further assume that $K$ is a closed topological disk in $U$ containing $z_0$. 

\begin{figure}[ht!]
\centering
\includegraphics[height=5cm]{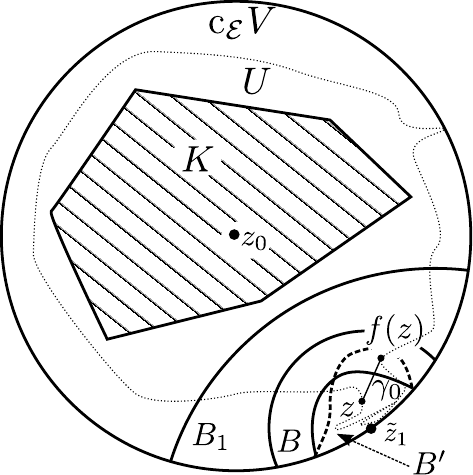}
\caption{Picture for case I}
\label{fig:case1}
\end{figure}
 
Let $B_1\subset V\sm K$ be a cross-section of $V$ such that $z_0\notin B_1$ and $\til{z}_1$ lies in the open interval of prime ends of $V$ determined by $B_1$; \ie $\til{z}_1\in \inter_{\cPE V} (\cl_{\cPE V} B_1)$ (see Figure \ref{fig:case1}). Choose a smaller cross-section $B\subset B_1$ of $V$ such that $\til{z}_1\in \inter_{\cPE V}(\cl_{\cPE V} B)$ and $f^k(B)\subset B_1$ for all $k$ with $0\leq k\leq N$, and finally let $B'\subset B$ be a smaller cross-section of $V$ such that $\til{z}_1\in \inter_{\cPE V}(\cl_{\cPE V} B')$ and $f(B')\subset B$. Note that $\bd U \cap B'\neq \emptyset$, since otherwise $U$ would contain $B'$, and that would imply that there is some fixed point accessible from $U$ (namely, any point of $\bd V$ accessible from $B'$) contradicting the fact that $\rho(f,U)\neq 0$ (and Proposition \ref{pro:main-acc}). Thus we may choose $z\in \bd U \cap B'$ and an arc $\gamma_0\subset B$ joining $z$ to $f(z)$. Since there are no fixed points in $B$, we know that $\gamma_0$ does not contain a fixed point, and by Lemma \ref{lem:gettrans} there is a arc $\gamma \subset B$ for $f|_V$ such that $z\in \gamma$ (so $\gamma$ intersects $\bd U$). Since $\gamma\subset \R^2\sm K$, Theorem \ref{th:arclemma} implies that $\gamma$ is not an $N$-translation arc, thus we conclude from Lemma \ref{lem:brouwerlink} (applied to $f|_V$) that $\bigcup_{k=0}^N f^k(\gamma)$ contains a loop that is homotopically nontrivial in $V\sm \fix(f|_V) = V\sm\{z_0\}$. This contradicts the fact that $\bigcup_{k=0}^N f^k(\gamma)\subset B_1$, which is simply connected and disjoint from $z_0$. This contradiction shows that Case I is not possible.

\medskip

\noindent \textbf{Case II.} There is a bounded connected component of $\fix(f)\sm\{z_0\}$.

\smallskip

In this case, we may further assume that there is a bounded component of $\fix(f)$ intersecting $\bd U$. Indeed, let $C$ be a bounded connected component of $\fix(f)$. If $C$ intersects $\bd U$, there is nothing to do. Otherwise, there is another connected component $C'$ of $\fix(f)$ that intersects $\bd U$ (since we are assuming that there is a fixed point $z_1$ in $\bd U$).  We may choose any $p_0\in C$ and consider the space $M=\R^2\sqcup \{\infty\}\sm\{p_0\}$, which is homeomorphic to $\R^2$. In this new space, $f$ induces a homeomorphism $f'$ for which all our hypotheses hold (the $\bd$-nonwandering condition in $U$ and the condition on the prime ends rotation number hold for $f'$ thanks to Proposition \ref{pro:bdw} and Corollary \ref{coro:prime-rot}). Since $C'$ is a compact subset of $\R^2$ intersecting $\bd U$ and $p_0\notin C'$, it follows that $C'$ is a connected component of $\fix(f')$ intersecting $\bd_M U$, and it is bounded in $M$ (because it is compact). Thus we may replace $f$ by $f'$ and have the required assumption.

\begin{claim} We may assume that $\fix(f)$ is totally disconnected.
\end{claim}

\begin{proof}

Let $V$ be the connected component of $(\R^2\sm \fix(f))\cup\{z_0\}$ containing $U$, and $K=\R^2\sm V$. Clearly $V$ is an $f$-invariant open set.  It follows from the hypothesis on the rotation number and Proposition \ref{pro:main-acc} that an accessible point from $U$ in $\bd U$ cannot be fixed by $f$, so it does not belong to $K$. Thus $\bd U\sm K\neq \emptyset$, and by Corollary \ref{coro:prime-rot} we have that $\rho(f, U, \R^2)= \rho(f,U, V)$. 

From the classification of noncompact surfaces \cite{richards}, $V$ is homeomorphic to $\R^2\sm E$ for some totally disconnected closed set $E$ (this can be seen by fixing $\mathfrak{p}\in \bdIB V$ and noting that $\cIB(V) \sm \{\mathfrak{p}\}$ is homeomorphic to $\R^2$ and $\bdIB(V)\sm \{\mathfrak{p}\}$ is totally disconnected). Let $h\colon V\to \R^2\sm E$ be a homeomorphism, and consider the map  $hfh^{-1}\colon \R^2\sm E\to \R^2\sm E$. A standard continuity argument shows that $hfh^{-1}$ extends to a homeomorphism $f'\colon \R^2\to \R^2$ such that $\fix(f')\sm\{h(z_0)\} = E$. Moreover, $U'=h(U)$ is a simply connected open $f'$-invariant set with a fixed point of $f'$ in its boundary, namely the image by $h$ of any bounded connected component of $\fix(f)$ that intersects $\bd U$ (we know that there is at least one such component by our assumption in Case II). Again by Corollary \ref{coro:prime-rot}, 
$$\rho(f,U,\R^2) = \rho(f,U,V) = \rho(f',U',\R^2\sm E) = \rho(f',U',\R^2).$$
Therefore $f'$ leaves $U'$ invariant, has a nonzero rotation number in $U'$, and also $f'$ has a fixed point in the boundary of $U'$. Moreover, the first part of Proposition \ref{pro:bdw} implies that $f_V$ is $\bd$-nonwandering in $U$, so $f'|_{\R^2\sm E}$ is $\bd$-nonwandering in $U'$. Since $E$ is totally disconnected, the second part of the same proposition implies that $f'$ is $\bd$-nonwandering in $U'$. Thus the same hypotheses that we have on $f$ hold for $f'$, and since $\fix(f')$ is totally disconnected, we may use $f'$ instead of $f$, proving the claim.
\end{proof}

Let $z_1\in \bd U\cap \fix(f)$. Since the set $\{z_0, z_1\}\subset \fix(f)$ is unlinked (see Remark  \ref{rem:unlinked}), Theorem \ref{th:jaulent} implies that there exists a closed set $X\subset \fix(f)$ containing $z_0$ and $z_1$ such that $X$ is maximal unlinked, and a corresponding isotopy $\mc{I} = (f_t\colon \R^2\sm X \to \R^2\sm X)_{t\in [0,1]}$ from $f_0=\id_{\R^2\sm X}$ to $f_1=f|_{\R^2\sm X}$. Moreover, since $X\subset \fix(f)$ is totally disconnected, the isotopy extends to $\R^2$ fixing all points of $X$; thus from now on we assume that $\mc{I} = (f_t)_{t\in [0,1]}$ with $f_t\colon \R^2 \to \R^2$, and $f_t(x) = x$ for all $t\in [0,1]$, $x\in X$. 

Let $X_1 = X\sm \{z_0\}$, and let $M = \R^2\sm X_1$. Note that since $X$ is totally disconnected and closed, $M$ is connected. Let us denote by $\mc{I}_M$ the restricted isotopy $(f_t|_M)_{t\in [0,1]}$. 
Let $\hat{\pi}\colon \hat{M}\to M$ be the universal covering of $M$ and $\hat{\mc{I}} = (\hat{f}_t)_{t\in [0,1]}$ the isotopy from $\id_{\hat M}$ to $\hat{f}=\hat{f}_1$ that lifts $\mc{I}|_M$. It is easy to see that $\hat{f}$ commutes with the elements of the group of covering transformations $\deck(\hat{\pi})$. Since the isotopy $\mc{I}|_M$ fixes $z_0$, it follows that any element of $\hat{\pi}^{-1}(z_0)$ is fixed by $\hat{\mc{I}}$, and in particular by $\hat{f}$. 

\begin{claim} \label{claim:main3} If $\hat{U}$ is a connected component of $\hat{\pi}^{-1}(U)$, then $\hat{U}$ is $\hat{f}$-invariant, unbounded and simply connected. Moreover, $\hat{\pi}|_{\hat{U}}$ is a homeomorphism onto $U$ that conjugates $\hat{f}|{\hat{U}}$ to $f|_U$, $\hat{f}$ is $\bd$-nonwandering in $\hat{U}$, and $\rho(f,U) = \rho(\hat{f},\hat{U})$.
\end{claim}
\begin{proof} It is invariant because $\hat{\pi}^{-1}(U)$ is $\hat{f}$-invariant and $\hat{U}$ is a connected component of that set, which necessarily contains an element of $\hat{\pi}^{-1}(z_0)\subset \fix(\hat{f})$. That $\hat{\pi}|_{\hat{U}}$ is injective follows from the fact that $U$ is simply connected. Since $\bd U$ contains $z_1\in X_1$, it follows that $\hat{U}$ is unbounded. The fact that $f|_U$ is conjugated to $\hat{f}|_{\hat{U}}$ is obvious.

To prove the last two claims, note that the homeomorphism $h=\hat{\pi}_{\hat{U}}\colon \hat{U}\to U$ has the special property that it maps cross-cuts to cross-cuts. This implies that $\hat{f}$ is $\bd$-nonwandering in $\hat{U}$, and it also gives a natural way to extend $h$ to a homeomorphism between the prime ends compactifications of $\hat{U}$ and $U$, which by continuity should conjugate the respective extensions of $\hat{f}|_{\hat{U}}$ and $f|_{U}$. In particular the two extended maps have the same rotation number in their prime ends circles.
\end{proof}

\begin{claim}\label{claim:main4} For any $z_1\in X_1 = X\sm \{z_0\}$, there is a neighborhood $B$ of $z_1$ in $\R^2$ such that $$\hat{\pi}(\fix(\hat{f}))\cap B= \emptyset.$$
\end{claim}
\begin{proof} 
We may choose a disk $D$ in $M$ containing $z_1$, disjoint from $z_0$ and such that the boundary of $D$ is disjoint from $X$ (because $X$ is totally disconnected). Fix a neighborhood $B$ of $z_1$ such that for any $z$ in $B$, the arc $(f_t(z))_{t\in[0,1]}$ is contained in $D$. To show that $B$ has the required property, suppose for contradiction that $z\in B$ is the image by $\hat{\pi}$ of a fixed point of $\hat{f}$. Then $\gamma=(f_t(z))_{t\in [0,1]}$ is a loop in $D$, and it must be homotopically trivial in $M$, because it lifts to a trivial loop in the universal covering. Let us show that $\gamma$ is homotopically trivial in $M\sm\{z_0\}$ as well: we know that there is a homotopy $(\gamma_s)_{s\in [0,1]}$ in $M$ from $\gamma$ to $z$ fixing the base point $z$. Let $r:\R^2\sm D \to \bd D$ be a retract (i.e. $r$ is continuous, $r(x) \in \bd D$ for all $x$ in $\R^2\sm D$ and $r(x)=x$ for $x\in \bd D$), and extend it to $\R^2$ by letting $r(x)=x$ in $D$. Then $(r\circ \gamma_s)_{s\in [0,1]}$ is a homotopy from $\gamma$ to $z$ which never leaves $\ol{D}$. Thus $\gamma$ is homotopically trivial in $M\sm \{z_0\}=\R^2\sm X$. But since  $z\in \fix(f)\sm X$ (because $z\in M$, $z\neq z_0$ and $M\cap X=\{z_0\}$), the last part of Theorem \ref{th:jaulent} says that $\gamma$ must be homotopically nontrivial in $\R^2\sm X$. This contradiction proves the claim.
\end{proof}

Fix a connected component $\hat{U}$ of $\hat{\pi}^{-1}(U)$.

\begin{claim} \label{claim:main5}
For any neighborhood $V$ of $z_1\in \bd U\cap X$, there exists a translation arc $\hat{\gamma}$ for $\hat{f}$ such that $\hat{\pi}(\hat{\gamma})\subset V$ and $\hat{\gamma}$ intersects $\bd \hat{U}$.
\end{claim}
\begin{proof}
Reducing $V$ if necessary, we may assume that $V\subset B$, where $B$ is the set from Claim \ref{claim:main4}, so that $\hat{\pi}(\fix(\hat{f}))\cap V= \emptyset$. Since the isotopy $\mc{I}$ fixes elements of $X$ (particularly $z_1$), if $\hat{z}\in \bd \hat{U}$ is chosen such that $z=\hat{\pi}(\hat{z})$ is close enough to $z_1$, one has that $\gamma_0=(f_t(z))_{t\in [0,1]}$ is contained in $V\sm X$. If $\hat{\gamma}_0=(\hat{f}_t(\hat{z}))_{t\in [0,1]}$, then $\hat{\gamma}_0$ is the lift of $\gamma_0$ with initial point $\hat{z}$ and its final point is $\hat{f}(\hat{z})$ (recall that $\hat{f} = \hat{f}_1$). The arc $\hat{\gamma}_0$ might not be simple, but it necessarily contains a simple arc $\hat{\gamma}_1$ joining $\hat{z}$ to $\hat{f}(\hat{z})$, which also satisfies $\hat{\pi}(\hat{\gamma}_1)\subset V\sm X$. Moreover, $\hat{\gamma}_1$ contains no fixed points of $\hat{f}$, since $\hat{\pi}(\fix(\hat{f}))$ is disjoint from $V$.  Thus from Lemma \ref{lem:gettrans} there is, in any neighborhood of $\hat{\gamma}_1$, a translation arc $\hat{\gamma}$ for $\hat{f}$ such that $\hat{z}\in \hat{\gamma}$. In particular, one may choose $\hat{\gamma}$ such that $\hat{\pi}(\hat{\gamma})\subset V$, as required.
\end{proof}

To finish the proof, note that from Claim \ref{claim:main3} one has that $\hat{U}$ is $\hat{f}$-invariant, simply connected and $\bd$-nonwandering for $\hat{f}$, and $\alpha = \rho(\hat{f},\hat{U})\neq 0$. Thus, from Theorem \ref{th:arclemma} applied to $\hat{f}$ and $\hat{U}$ we obtain a constant $N=N_{\alpha,0}$ and a compact set $K\subset \hat{U}$ such that no $N$-translation arc of $\hat{f}$ intersecting $\bd\hat{U}$ is contained in $\hat{M}\sm K$. Let $z_1\in \bd U\cap X$ and consider the set $B$ given by Claim \ref{claim:main4}. We assume additionally that $B$ is simply connected and disjoint from $\pi^{-1}(K)$ (replacing it by a smaller neighborhood of $z_1$ if necessary). 
Note that $\hat{\pi}^{-1}(B)\subset \hat{M}\sm \fix(\hat{f})$, and each of its connected components is simply connected.

Choose a neighborhood $V$ of $z_1$ such that $f^k(V)\subset B$ for $0\leq k\leq N$, and let $\hat{\gamma}\subset V$ be the translation arc for $\hat{f}$ given by Claim \ref{claim:main5}.
Let $\hat{\Gamma}_N=\bigcup_{k=0}^N \hat{f}^k(\hat{\gamma})$. Since $\hat{\pi}(\hat{\Gamma}_N)\subset B$, we have that $\hat{\Gamma}_N$ is contained in a connected component of $\hat{\pi}^{-1}(B)$, which is simply connected and disjoint from $\fix(\hat{f})$. In particular, any loop contained in $\hat{\Gamma}_N$ is homotopically trivial in $\hat{M}\sm \fix(\hat{f})$. 
Thus, the only possibility in Lemma \ref{lem:brouwerlink} applied to $\hat{f}$, is that $\hat{\gamma}$ be an $N$-translation arc.  Since $\hat{\gamma}$ intersects $\bd\hat{U}$ and is contained in $\hat{M}\sm K$, we get a contradiction from Theorem \ref{th:arclemma}. This proves that Case II does not hold, completing the proof of the theorem. 
\qed

\subsection{Proof of Corollary \ref{coro:main}}

Suppose there is a fixed point $z_1$ in $\R^2\sm U$, and $U$ is unbounded. Then we consider the new surface $M=(\R^2\cup\{\infty\})\sm \{z_1\}$ which is still homeomorphic to $\R^2$, and the homeomorphism defined from $f$ by setting $f(\infty)=\infty$. By Proposition \ref{pro:bdw} and Corollary \ref{coro:prime-rot}, the hypotheses of Theorem \ref{th:main} hold for this new map, so there are no fixed points in the boundary of $U$ in $M$. But this contradicts the fact that $\infty$ is in the boundary of $U$ in $M$ (because $U$ was assumed to be unbounded).
\qed


\section{Generalizations to arbitrary surfaces}
\label{sec:main-gen}

If $S$ is a closed surface of positive genus and $U\subset S$ is a simply connected invariant set with irrational prime ends rotation number for a homeomorphism $f$, we wonder not only whether there can be fixed points in $\bd U$ but also whether $U$ can be too complicated from the point of view of the topology of the ambient space. Specifically, can the closure of $U$ be non-contractible in $S$? Without any additional hypothesis, the answer is yes. For example, $\bd U$ could be a Denjoy-type continuum (\ie the minimal set of the suspension flow on $\T^2$ of a Denjoy example in the circle). See Example \ref{ex:walker2}. The next theorem, which is Theorem \ref{thm:intro-main-gen} in the introduction, shows that if $f$ is $\bd$-nonwandering in $U$ then this is no longer the case, and gives additional dynamical information.

\begin{theorem}\label{th:main-gen} Let $f\colon S\to S$ be an orientation preserving homeomorphism of an orientable surface $S$ of finite type, and $U\subset S$ an open $f$-invariant topological disk such that $S\sm U$ has more than one point. Assume further that $f$ is $\bd$-nonwandering in $U$ and $\rho(f,U)$ is irrational. Then exactly one of the following holds:
\begin{itemize}
\item[(i)] $\bd U$ is aperiodic, and $\ol{U}$ is compact and contractible;
\item[(ii)] $S$ is a sphere, and the only periodic point of $f$ in $\bd U$ is a unique fixed point;
\item[(iii)] $S$ is a plane, and $\bd U$ is unbounded and aperiodic.
\end{itemize}
Moreover, there is a neighborhood $W$ of $\bd U$, which can be chosen as an annulus in case (i), a disk in case (ii) and the complement of a closed disk in case (iii), such that every connected component of $S\sm \bd U$ contained in $W$ is wandering. 
\end{theorem}

\begin{remark} For closed surfaces, only cases (i) and (ii) are possible, and in both cases $\bd U$ is contractible.
Observe also that the last part of the theorem implies that all but one of the connected components of $S\sm \ol{U}$ are wandering in case (i), and all connected components of $S\sm \ol{U}$ are wandering in the remaining cases. In particular, in cases (ii) and (iii), there are no periodic points in $S\sm U$ except possibly a unique fixed point.
\end{remark}

For the case of surfaces of finite genus (but possibly not of finite type) one can easily obtain the following statement, by applying the previous theorem to the ends compactification of $f$.

\begin{corollary}\label{coro:main-gen-gen} If in the hypotheses of the prevoius theorem one replaces the condition that $S$ be of finite type by the weaker condition that $S$ has finite genus only, then one of the following cases holds:
\begin{itemize}
\item[(i)] $\bd U$ is aperiodic, and exactly one connected component of $S\sm \ol{U}$ is nonwandering (hence invariant);
\item[(ii)] $S$ has genus $0$, the set of periodic points of $f$ has at most one element (thus necessarily fixed), and every connected component of $S\sm \ol{U}$ is wandering.
\end{itemize}
\end{corollary}

%

Finally, in the case that $f$ is nonwandering on a closed surface one has a simpler statement, which is Theorem \ref{thm:intro-main-gen-NW} in the introduction.

\begin{corollary}\label{coro:main-gen} Let $f\colon S\to S$ be a nonwandering homeomorphism of a closed surface, and $U$ an open $f$-invariant simply connected set with irrational prime ends rotation number. Then one of the following holds:
\begin{itemize}
\item[(i)] $\bd U$ is a contractible annular continuum without periodic points;
\item[(ii)] $S$ is a sphere, $U$ is dense in $S$, and $S\sm U$ is a cellular continuum containing a unique fixed point and no other periodic points.
\end{itemize}
\end{corollary}
\begin{proof}
If the first case in Theorem \ref{th:main-gen} holds, then $\bd U$ is an aperiodic invariant continuum, and so by Theorem \ref{th:koro} it is an annular continuum as well, so the first case of the corollary holds. 
On the other hand, if the second case of Theorem \ref{th:main-gen} holds, then since $f$ is nonwandering and all components of $S\sm \ol{U}$ are wandering, it follows that $S\sm\ol{U}=\emptyset$, so that $\bd U = S\sm U$ is nonseparating (\ie a cellular continuum).
\end{proof}

\begin{remark} Examples \ref{exm:invers} and \ref{exm:invers2} show that the case with a fixed point cannot be excluded, even in the smooth area-preserving setting. 
\end{remark}

\subsection{Proof of Theorem \ref{th:main-gen}}
\setcounter{claim}{0}
We begin considering closed surfaces only.

%

%

\begin{claim}\label{claim:gen1} If $S$ is a closed surface, then one of cases (i) or (ii) holds.
\end{claim}
\begin{proof}
We consider first the case where $S$ is a sphere. Then $S\sm U$ is a non-separating $f$-invariant continuum. We claim that there is a fixed point $z_0\in S\sm U$. This is a consequence of Cartwright-Littlewood's theorem, but alternatively one can prove it as follows: Since $\rho(f,U)\neq 0$, if $D\subset U$ is a closed topological disk such that $\bd D$ is a loop disjoint from $\fix(f)$ and homotopic to the circle $\bdPE(U,S)$ in $\cPE(U,S)\sm \fix(f)$, then the fixed point index of $f$ in $U$ is $i(f,D)=1$. Since $f$ preserves orientation and $S$ is a sphere, Lefschetz' formula implies that $i(f,D)+i(f, S\sm D) = 2$, and in particular there is a fixed point in $S\sm D$. But there are no fixed points in $U\sm D$, so there must be a fixed point in $S\sm U$ as claimed. 

Let $S'=S\sm \{z_0\} \simeq \R^2$. Then $\bd_{S'} U$ has more than one point, and by Corollary \ref{coro:prime-rot}, $\rho(f|_{S'}, U, S') = \rho(f,U, S)$. Also, $f|_{S'}$ is $\bd$-nonwandering in $U$ by Proposition \ref{pro:bdw}. Hence by Theorem \ref{th:main} applied to $f^n|_{S'}$ for each $n\in \N$ we conclude that there is no periodic point in $\bd_{S'}U$. 

Suppose first that $z_0\notin \bd_S U$. Then $\bd_S U=\bd_{S'}U$, and so there are no periodic points in $\bd_S U$. This means that $\bd_S U$ is aperiodic. Since $z_0\in \inter S\sm \ol{U}$, it follows that $\ol{U}\neq S$, and since $S$ is a sphere this easily implies that $\ol{U}$ is contractible, so case (i) holds.

Now suppose that $z_0\in \bd_S U$. Then $U$ is unbounded in $S'$, and Corollary \ref{coro:main} implies that there is no periodic point in $S'\sm U$. Thus the unique periodic point of $f$ in $S\sm U$ is $z_0$. This proves that case (ii) holds, completing the proof of the claim in the case that $S$ is a sphere.

Now assume that $g$ is a surface of genus $g>1$. The case $g=1$ is similar, and is explained at the end of the proof.
Let $\hat{\pi}\colon \D\to S$ be the universal covering of $S$. We may assume that $\D$ is endowed with the hyperbolic metric and the group of covering transformations $\deck(\pi)$ consists of isometries of $\D$. 
There is some lift $\hat{f}$ of $f$ that fixes a connected component of $\hat{U}$ of $\hat{\pi}^{-1}(U)$. Claim \ref{claim:main3} from the proof of Theorem \ref{th:main} remains valid in the present setting, with the same proof, so we have that $\hat{U}$ is $\bd$-nonwandering for $\hat{f}$ and $\rho(\hat{f},\hat{U}) = \rho(f,U)$. From Theorem \ref{th:main} applied to $\hat{f}^n$ for each $n\in \N$, we conclude that there are no periodic points of $\hat{f}$ in $\bd \hat{U}$. Let us divide the remainder of the proof of Claim \ref{claim:gen1} into two sub-claims.

\begin{subclaim}\label{claim:nielsen} $\cl_{\ol{\D}} \hat{U} \cap \mathbb{S}^1= \emptyset$ (\ie $\hat{U}$ is bounded in $\D$).
\end{subclaim}
\begin{proof}
By the theory of Nielsen, the map $\hat{f}$ extends to a the boundary circle $\mathbb{S}^1$, defining a homeomorphism $F$ from $\ol{\D}=\D\cup \mathbb{S}^1$ to itself. Moreover, there is $n\in \N$ such that $\hat{f}^n$ has a fixed point in $\bd D$ (this is for instance the content of \cite[Theorem 1]{nielsen}\footnote{An english translation of \cite{nielsen} can be found in \cite{nielsen-collected}}). We may extend $F$ to the whole plane $\C$ continuously, for instance by letting $F(re^{it}) = rF(e^{it})$ for $r\geq 1$, $t\in \R$. 

Therefore, the map $F^n\colon \C\to \C$ leaves invariant the open $\bd$-nonwandering topological disk $\hat{U}\subset \D$ and also the disk $\D$, and $F^n$ has a fixed point in $\mathbb{S}^1$. Since $F^n|_{\mathbb{S}^1}$ is a circle homeomorphism with a fixed point, the orbit of any point in $\mathbb{S}^1$ converges to a fixed point of $F^n|_{\mathbb{S}^1}$. 

Assume for contradiction that the claim does not hold, \ie $\bd_{\ol{\D}}\hat{U}\cap \mathbb{S}^1\neq \emptyset$. Since the latter is a closed $F^n$-invariant set, it follows from the previous observation that $\bd_{\ol{\D}}\hat{U}$ contains a fixed point of $F^n$ (which lies in $\mathbb{S}^1$). From Corollary \ref{coro:prime-rot} we have that
$\rho(F^n,\hat{U}, \C)=\rho(\hat{f}^n, \hat{U}, \D)$ which is just $n\rho(\hat{f},\hat{U},\D)$, which in turn is equal to $n\rho(f,U)$, an irrational rotation number. 

Let us show that $F$ is $\bd$-nonwandering in $U$ (and hence so is $F^n$ for any $n\in \N$). Since we already know that $F|_{\hat{U}} = \hat{f}$ is $\bd$-nonwandering in $U$, it suffices to show that any wandering cross-section of $U$ in $\C$ contains some cross-section of $U$ in $\D$. Let $D$ be any cross-section of $U$ in $\C$ so that $\gamma = \bd_U D$ is a cross-cut of $U$ in $\C$. 
Observe that $\bd_{\C} D \subset \cl_{\C}\gamma \cup \mathbb{S}^1\cup \bd_{\D} U$. Moreover, if $\bd_{\C} D \sm \cl_{\C}\gamma$ intersects $\bd_{\D} U$ at some point $z\in \D$, then we may easily obtain a cross-section $D'$ of $U$ in $\D$ which is contained in $D$, by an argument already used in Section \ref{sec:prime}: letting $C$ be a small enough circle around $z$ (so that it bounds a disk disjoint from $\cl_{\C}\gamma$), one has that any connected component of $D\cap C$ is a cross-cut of $U$ contained in $D$, and one of its cross-sections must be contained in $D$, as required.

Thus we need to show that $\bd_{\C} D \sm \cl_{\C}\gamma$ always intersects $\bd_{\D} U$.  If this is not the case, then $\bd_{\C} D \sm \cl_{\C}\gamma\subset \mathbb{S}^1$, so that $D$ is bounded by $\cl_{\C}\gamma\cup \mathbb{S}^1$. This implies that both endpoints of $\cl_{\C}\gamma$ are in $\mathbb{S}^1$ and they are different (otherwise $D$ would not be a cross-section of $U$ in $\C$). Thus $D$ is a cross-section of $\D$.  But from the fact that $S$ is a closed surface it follows that given $z\in S$, any cross-section of $\D$ contains infinitely many elements of $\pi^{-1}(z)$, contradicting the fact that $\pi(\hat{U})$ is simply connected.
%

We have thus obtained an orientation-preserving homeomorphism $F^n$ of the plane with a $\bd$-nonwandering invariant open topological disk $\hat{U}$ which has a fixed point in the boundary, and $\rho(F^n, \hat{U}, \C)\neq 0$. This contradicts Theorem \ref{th:main}, completing the proof of Claim \ref{claim:nielsen}.
\end{proof}

Thus the set $K_0=\cl_{\ol{\D}} \hat{U}$ is a compact subset of $\D$. Let $K$ be the ``filling'' of $K_0$, \ie the union of $K_0$ with all the bounded (in $\D$) connected components of $\D\sm K_0$. 

\begin{subclaim}\label{claim:contract} $\hat{\pi}|_K$ is injective. In other words, $K\cap T(K)=\emptyset$ for any Deck transformation $T\neq \id$.
\end{subclaim}

\begin{proof} Consider the family of sets $H = \{T(K): T\in \deck(\pi),\,  T(K)\cap K\neq \emptyset\}$.  The elements of $H$ are permuted by $\hat{f}$. Moreover, since $\deck(\pi)$ acts properly on $\D$, there are finitely many elements in $H$. Therefore, there is $k\in \N$ such that $\hat{f}^k$ fixes every element of $H$. Assume for contradiction that $H$ contains some element $T(K)$ with $T\neq \id$. Since $\hat{U}$ is disjoint from $T(\hat{U})$, we have that $\hat{U}\cap \bd T(\hat{U})=\emptyset$ as well. Since $\bd T(K) \subset \bd T(\hat{U})$, we deduce that $\hat{U}\cap \bd T(K)=\emptyset$. The fact that $\deck(\pi)$ acts properly on $\D$ also implies that $T^{-1}(K)\not\subset K$, so that $K\not\subset T(K)$. This implies that $\bd K$ intersects $\D\sm T(K)$, so that $\hat{U}\not\subset T(K)$ (because $\bd K \subset \bd \hat{U}$). The latter, together with the fact that $\hat{U}$ is disjoint from $\bd T(K)$, implies that $\hat{U}$ is disjoint from $T(K)$. Observe that since $T(K)$ intersects both $K$ and $\D\sm K$, we have that $\bd K\subset \bd\hat{U}$ intersects $T(K)$. 

Since $K$ is a non-separating continuum, $\D\sm T(K)$ is homeomorphic a one-punctured plane $\R^2\sm \{z_0\}$. The map $\til{g}$ induced by $\hat{f}^k$ in $\R^2\sm \{z_0\}$ extends to $\R^2$ by fixing $z_0$ (because $\hat{f}^k$ fixes $T(K)$), and the set $\til{U}$ corresponding to $\hat{U}$ after this identification is a $\bd$-nonwandering invariant topological disk for $\til{g}$ in $\R^2$ due to Proposition \ref{pro:bdw}. Moreover, by Corollary \ref{coro:prime-rot}
$$\rho(\til{g},\til{U}, \R^2)  = \rho(\hat{f}^k|_{\D\sm T(K)},\hat{U},\D\sm T(K)) = \rho(\hat{f}^k,\hat{U},\D) = n\rho(\hat{f},\hat{U},\D)\neq 0.$$ Thus Theorem \ref{th:main} implies that $\til{g}$ has no fixed points in $\bd\til{U}$, contradicting the fact that $z_0$ is in the boundary of $\til{U}$ (which follows from the fact that $\bd\hat{U}$ intersects $T(K)$). This contradiction proves Claim \ref{claim:contract}.
\end{proof}

Since $K$ is compact and connected, and we showed that it projects injectively to $S$, it follows that $\pi(K)$ is contractible in $S$. Since $\ol{U}\subset \pi(K)$, we conclude that $\ol{U}$ is contractible in $S$ as well. Finally, as we already explained before Claim \ref{claim:nielsen}, by Theorem \ref{th:main} there are no periodic points of $\hat{f}$ in $\bd \hat{U}$, and since $\pi|_{\cl{\hat{U}}}$ is a homeomorphism onto $\ol{U}$ and $\pi{\hat{f}}=f\pi$, we conclude that there are no periodic points of $f$ in $\bd U$ as well. This completes the proof of Claim \ref{claim:gen1} when $S$ has genus $g>1$.

It remains to consider the case where the genus of $S$ is $g=1$, that is, when $S\simeq \T^2$. But in this case,  the same proof used for the case $g>1$ works, after observing the following facts: any lift $\hat{f}\colon \R^2\to \R^2$ of a homeomorphism $f\colon \T^2\to \T^2$ has the form $\hat{f}(x)=Ax + \phi(x)$ where $A\in \GL(2,\Z)$ and $\phi$ is a $\Z^2$-periodic map. This implies that if one compactifies $\R^2$ with a ``circle at infinity'' $S_\infty\simeq \mathbb{S}^1$ with the topology where a neighborhood system of $re^{it}$, for $r>0$, is given by the open sets $V_{\epsilon,M}=\{re^{is}$ where $r>M$ and $\abs{s-t}<\epsilon\}$. Since $\hat{f}=A+\phi$, it follows easily that the map $\hat{f}$ extends to $\R^2\sqcup S_\infty \simeq \ol{\D}$ continuously, and this extension depends only on $A$ (in fact, if $v=re^{it}$, then the extension maps $v$ to $Av/\norm{Av}$).

Moreover, since $A\in \GL(2,\Z)$, one has from classical results that some power of $A$ has a real eigenvector. Therefore there is $n>0$ and $v\in S_\infty$ such that $A^nv/\norm{A^nv}=v$, and so the extension of $\hat{f}^n$ to $\R^2\sqcup S_\infty\simeq \ol{\D}$ has a fixed point in the boundary circle. Thus in Claim \ref{claim:nielsen} we may use these facts instead of the result of Nielsen, and the rest of the proof works without any modification for this case. This completes the proof of Claim \ref{claim:gen1}.
\end{proof}

\begin{claim}\label{claim:gen2} If $S$ is closed, then there a neighborhood $W$ of $\bd U$, which is an annulus in case (i) and a disk in case (ii), such that every connected component of $S\sm \bd U$ contained in $W$ is wandering.
\end{claim}
\begin{proof}
By Claim \ref{claim:gen1}, one of cases (i) or (ii) holds. Assume first that case (ii) holds, so $S$ is a sphere, and there is a fixed point $z_1\in \bd U$. We will show that every connected component of $S\sm \ol{U}$ is wandering, which implies that the disk $W=S\sm \{z_0\}$ satisfies the claim for any $z_0\in U$. 
Suppose on the contrary that there is a nonwandering component $V$ of $S\sm \ol{U}$. Since $\ol{U}$ is invariant, it follows that $f^n(V)=V$ for some $n$. Since $\ol{U}$ is connected and $S$ is a sphere, $V$ is simply connected. Choose $x_0\in V$ and let $h\colon S\to S$ be a homeomorphism $h\colon S\to S$ which is the identity outside a compact subset of $V$ and such that $h(f^n(x_0))=x_0$. Then $f'=hf^n$ is a new map which coincides with $f$ in a neighborhood of $\ol{U}$ and $f'^n(x_0)=x_0$. But this contradicts Theorem \ref{th:main} applied to $f'^n|_{S\sm \{x_0\}}$. This proves the claim when case (ii) holds.

Now suppose that case (i) holds, so $\ol{U}$ is contractible. This means that there is a closed topological disk $B\subset S$ that contains $\ol{U}$ in its interior. Choose any point $z_0\in U$ and let $W=B\sm \{z_0\}$. To prove that $W$ satisfies the claim, assume for contradiction that there is some nonwandering component $V_1$ of $S\sm \ol{U}$ such that $V_1\subset W$. Then $V_1$ is periodic, and we may assume $f(V_1)=V_1$ by considering a power of $f$ instead of $f$ (while keeping the remaining hypotheses).

Let us show that it suffices to consider the case where $S$ is a sphere. If $S$ is not a sphere, then its universal covering space $\hat{S}$ is homeomorphic to $\R^2$. If $\pi\colon \hat{S}\to S$ is the covering projection, there is a connected component $\hat{U}$ of $\pi^{-1}(U)$ and a lift $\hat{f}\colon \hat{S}\to \hat{S}$ of $f$ such that $\hat{f}(\hat{U})=\hat{U}$. Let $\hat{B}$ be the connected component of $\pi^{-1}(B)$ containing $\hat{U}$, so that $\pi|_{\hat{B}}$ is a homeomorphism onto $B$, and let $\hat{V}_1$ be the connected component of $\pi^{-1}(V_1)$ in $B$, so that $\hat{f}(\hat{V}_1)=\hat{V}_1$. 
Let $S'=S\sqcup\{\infty\}$ be the one-point compactification of $\hat{S}$ (so $S'$ is a sphere), $f'$ the map induced on $S'$ by $\hat{f}$, and $U'=\hat{U}$. The facts that $\pi f'=f'\pi$ holds on $\cl_{\hat{S}}\hat{U}$ and $\pi|_{\cl_{\hat{S}}\hat{U}}$ is a homeomorphism onto $\ol{U}$ imply that $f'$ is $\bd$-nonwandering in $U'$ and $\rho(f',U')$ is irrational. Moreover, if $V_1'=\hat{V}_1$ and $V_2'$ is the connected component of $S'\sm \cl_{\hat{S}}\hat{U}$ containing $\infty$, then $V_1'$ and $V_2'$ are both $f'$-invariant connected components of $S'\sm \cl_{S'} U'$. Thus we are in the same setting as in the beginning of the proof, but on the sphere.

Thus we assume that $S$ is already the sphere and we will find a contradiction. Choose any point $z_i\in V_i$ and let $B_i$ be a closed topological disk containing $z_i$ and $f(z_i)$. If $\gamma_i\colon [0,1]\to \bd B_i$ is a parametrization of $\bd B_i$ (oriented positively), then $\gamma_i$ is homotopic to $f\circ \gamma_i$ in $V_i\sm \{f(z_i)\}$ (because $f$ preserves orientation), and so by Theorem \ref{th:epstein} applied to $V_i\sm \{f(z_i)\}$ there is a homeomorphism $h_i\colon V_i\sm\{f(z_i)\}$ such that $h_i\circ f\circ \gamma_i = \gamma_i$ (note that this means that $h_if$ fixes each point of $\bd B_i$), and $h_i$ is fixed outside a compact subset of $V_i\sm \{f(z_i)\}$. 

We may define $h\colon S\to S$ by $h(z) = h_i(z)$ if $z\in V_i$, and $h(z)=z$ if $z\notin V_1\cup V_2$, which is a homeomorphism of $S$. Then $f'=hf$ coincides with $f$ on $\ol{U}$ (in particular $\rho(U,f')$ is irrational and $f'$ is $\bd$-nonwandering in $U$) and has a pointwise fixed circle $\bd B_i$ in $V_i$ for $i\in \{1,2\}$. Consider the new surface $S'= S\sm (\inter B_1\cup \inter B_2)/{{\sim}}$, where ${{\sim}}$ is the relation that identifies $\bd B_1$ with $\bd B_2$ by $\gamma_1(t)\sim\gamma_2(1-t)$ for $x\in \mathbb{S}^1$ (the orientation of $\gamma_2$ was reversed so that the identification gives an orientable surface again, see Figure \ref{fig:gluing}). Then $S'$ is a torus, and after the identification, $V_1\sm \inter(B_1)\cup V_2\sm \inter(B_2)$ becomes a topological annulus $A$ in $S'$ which is a connected component of $S'\sm \ol{U}$. 

\begin{figure}[ht!]
\centering
\includegraphics[width=.8\textwidth]{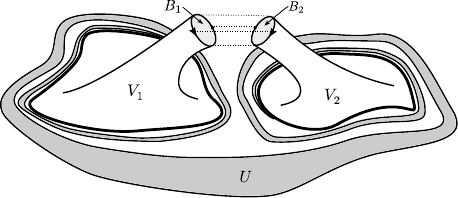}
\caption{Identifying the boundaries of $B_1$ and $B_2$ to obtain a torus}
\label{fig:gluing}
\end{figure}

Let $g'\colon S'\to S'$ be the map induced by $f'$. Observe that since $\bd U$ is aperiodic for $f$, it is aperiodic for $g'$ as well. Since $g'$ is $\bd$-nonwandering in $U$ and $\rho(g',U)$ is irrational, Claim \ref{claim:gen1} applied to $g'$ and $S'$ implies that $\bd U$ is contractible in $S'$. But that means that $\bd U$ is contained in some closed topological disk $B$, and so one of the connected components of $S'\sm \ol{U}$ is a punctured torus and the remaining ones are open topological disks. But as we already mentioned, $A$ is a connected component of $S'\sm \ol{U}$ homeomorphic to an annulus; so we arrived to a contradiction, proving the claim.
\end{proof}

Claims 1 and 2 complete the proof of the theorem in the case that $S$ is closed. To finish the proof we need to consider the case where $S$ is a non-closed surface of finite type. In this case, the ideal boundary $\bdIB(S)$  is finite, so if $S'=\cIB(S)$ and $f'\colon S'\to S'$ is the extension of $f$, then $\bdIB(S)$ is a finite $f'$-invariant set, which must therefore consist of finitely many periodic orbits. Again by Corollary \ref{coro:prime-rot} and Proposition \ref{pro:bdw} the hypotheses of the theorem remain valid for $U$ and $f'$; in particular since $S'$ is a closed surface, what we already proved implies that one of the cases (i) or (ii) holds for $f'$.

 Suppose first that case (i) holds for $S'$ and $f'$. Then $\bd_{S'} U$ is aperiodic and $\cl_{S'}U$ is contractible in $S'$. Since points of $\bdIB(S)$ are periodic for $f'$, it follows that $\bd_{S'}U$ is disjoint from $\bdIB(S)$. Note that this implies that $\bd_S{U} = \bd_{S'} U$, so $\bd U = \bd_S U$ is aperiodic. Also we have an annular neighborhood $W'$ containing $\bd U$ such that every connected component of $S'\sm \bd_{S'} U$ contained in $W'$ is wandering. One of the boundary components of $\bd W'$ must be contained in $U$, so $B=U\cup W'$ is an open topological disk containing $\ol{U}$. Since points of $\bdIB(S)$ are disjoint from $\bd U$ and from any connected component of $S'\sm \bd U$ contained in $B$, this means that the finite set $\bdIB(S)$ is contained in the unique connected component $V$ of $S'\sm \bd U$ not contained in $B$. It follows that there is a smaller open topological disk $B'\subset B$ such that $\ol{U}\subset B'$ and $B'$ is disjoint from $\bdIB(S)$. Thus $B'\subset S$, and $\ol{U}$ is contractible in $S$. Thus case (i) holds for $S$ and $f$. Also, letting $W = B'\sm \{z_0\}$ for some point $z_0\in U$, we obtain an annular neighborhood of $\bd U$ with the required property, completing the proof in this case.
 
 Now assume that case (ii) holds for $S'$ and $f'$. Then $S'$ is a sphere, and every component of $S' \sm \cl_{S'} U$ is wandering for $f'$. Moreover, $\bd_{S'} U$ has a unique fixed point $z_0$ and no other periodic point, and there is an open topological disk $W'$ containing $\bd_{S'}U$ such that every connected component of $S'\sm U$ contained in $W'$ is wandering. Since $U$ is invariant, $S'\sm W'$ intersects $U$, so $\bd W'\subset U$. This implies that $S'\sm \cl_{S'}U \subset W'$.
 
Consider first the case where $z_0\in \bdIB(S)$. Then, since $\bdIB(S)$ consists of periodic points of $f'$ and there are no other periodic points in $\bd_{S'}U$, it follows that $\bdIB(S)\sm \{z_0\} \subset S'\sm \cl_{S'} U$. But since all connected components of $S'\sm \cl_{S'}(U)$ are contained in $W'$, they are wandering, so they cannot contain a point of $\bdIB(S)$. This implies that $\bdIB(S)=\{z_0\}$, hence $S$ is homeomorphic to $\R^2$ and $\bd U = \bd_S U$ is aperiodic. Thus, case (iii) holds for $f$ and $S$. Moreover $S'\sm W'$ is a closed subset of $U$, so taking any closed topological disk $B\subset U$ containing $S'\sm W'$ we have that $W=S\sm B$ has the property that every connected component of $S\sm \bd U$ in $W$ is wandering, as required.

Finally, consider the case where $z_0\notin \bdIB(S)$. Then $z_0\in \bd_S U$, and $\bd_{S'}U \sm \{z_0\}$ contains no periodic points of $f'$. This implies that $\bdIB(S)\subset S'\sm \cl_{S'} U$. But every component of $S'\sm \cl_{S'} U$ is wandering (because it is contained in $W'$), so we conclude that $\bdIB(S)=\emptyset$. This means that $S=S'$ and $f'=f$, and the theorem holds.
This completes the proof of Theorem \ref{th:main-gen}
\qed

\section{Generalizations to non-simply connected sets}
\label{sec:general}

In the case that $U$ is not simply connected, we can adapt our previous results to the next somewhat technical, but potentially useful result. Recall the notation and definitions from Section \ref{sec:prelim}, \S\ref{sec:prime-general} and \S\ref{sec:bdnw}.

\begin{theorem}\label{th:general} Let $f\colon S\to S$ be an orientation-preserving homeomorphism of a closed surface $S$, and $U\subset S$ an open $f$-invariant connected set. Let $\mathfrak{p}$ be a regular fixed end of $U$. If $f$ is $\bd$-nonwandering at $\mathfrak{p}$, and $\rho(f,\mathfrak{p})$ is irrational, then one of the following holds:
\begin{itemize}
\item[(i)] $Z(\mathfrak{p})$ is aperiodic, or
\item[(ii)] $Z(\mathfrak{p})$ is contractible and has a unique fixed point and no other periodic point.
\end{itemize}
Furthemore, there is a neighborhood $W$ of $Z(\mathfrak{p})$, which is an annulus in case (i) and a disk in case (ii), such that every connected component of $S\sm Z(\mathfrak{p})$ contained in $W$ is wandering.
\end{theorem}

Again, in the nonwandering setting we have a simpler statement:
\begin{corollary}\label{coro:general} Let $f\colon S\to S$ be an orientation-preserving nonwandering homeomorphism of a closed surface $S$, and $U\subset S$ an open $f$-invariant connected set. Let $\mathfrak{p}$ be a regular fixed end of $U$. If $\rho(f,\mathfrak{p})$ is irrational, then one of the following holds:
\begin{itemize}
\item $Z(\mathfrak{p})$ is an aperiodic annular continuum, or
\item $Z(\mathfrak{p})$ is a cellular continuum with a unique fixed point and no other periodic points. 
\end{itemize}
\end{corollary}

A version of Theorem \ref{th:arclemma} can also be stated in this setting:

\begin{theorem}\label{th:arc-general}  Let $f\colon S\to S$ be an orientation-preserving homeomorphism of an orientable surface $S$ of genus $g<\infty$, and $U\subset S$ an open $f$-invariant connected set. Let $\mathfrak{p}$ be a regular fixed end of $U$ such that $f$ is $\bd$-nonwandering at $\mathfrak{p}$, and $\rho(f,\mathfrak{p})=\alpha\neq 0$. Then there is a constant $N_{\alpha,g}\in \N$ and a $\mathfrak{p}$-collar $C$ such that any $N_{\alpha,g}$-translation arc that intersects $Z(\mathfrak{p})$ also intersects $U\sm C$. 
\end{theorem}

\subsection{Proof of Theorem \ref{th:general}} 

The main idea to obtain theorem \ref{th:general} from the previous results is to use surgery to reduce the problem to one in which $U$ is simply connected. Assuming the hypotheses of the theorem, consider the ends compactification $\cIB(U)$ and the extension $f_*$ of $f|_{U}$ to $\cIB(U)$. Since $\mathfrak{p}$ is a regular end, $\mathfrak{p}$ is an isolated fixed point of $f_*$. Let $\gamma_0$ be a simple loop bounding a $\mathfrak{p}$-collar\footnote{recall from \S\ref{sec:prime-general}: this means that $D$ is a subset of $U$ bounded by a simple loop such that $D\cup \{\mathfrak{p}$\} is a closed topological disk in $\cIB(U)$} 
$D_0$ in $U$, and $\gamma$ another simple loop bounding a $\mathfrak{p}$-collar $D$ such that $D\cup f(D)\subset D_0$. Then $\gamma':=f(\gamma)$ is homotopic to $\gamma$ in $D_0$. Let $D_1$ be another $\mathfrak{p}$-collar such that $D_1\subset D\cap f(D)$. Then $\gamma$ and $\gamma'$ are homotopic simple essential loops in the annulus $A = \inter D_0\sm D_1$, and it follows from Theorem \ref{th:epstein} that there exists an orientation preserving homeomorphism $h\colon A\to A$ such that $h(\gamma')=\gamma$ and $h(x)=x$ for $x\in \bd A$. We may extend $h$ as the identity in $S\sm A$, and define $f' = hf$, so that $f'(\gamma)=\gamma$ and $f'$ coincides with $f$ outside $A$. In particular, the fact that $f'$ coincides with $f$ in the $\mathfrak{p}$-collar $D_1$ implies that $f'$ is $\bd$-nonwandering at $\mathfrak{p}$ and $\rho(f',\mathfrak{p})=\rho(f,\mathfrak{p})$ is irrational. 

Let us assume that $\gamma$ is non-separating in $S$ (the remaining case is similar and explained at the end of the proof). Let $S_0 =S\sm \gamma$ (for the separating case, one would choose the connected component of $S\sm \gamma$ containing $D_1$). Let $S'= \cIB S_0$, and denote by $g'\colon S'\to S'$ the extension of $f'|_{S_0}$ to $S'$. If $A_i$ denotes the connected component of $\ol{A}\sm \gamma$ bounded by $\gamma_i$, for $i\in \{0,1\}$, then each $A_i$ is a collar of one of a different end $\mathfrak{e}_i\in \bdIB(S_0)$, and $\bdIB(S_0)=\{\mathfrak{e}_0,\, \mathfrak{e}_1\}$.  So the sets $A_i'=A_i\cup \{\mathfrak{e}_i\}$ are topological disks which are neighborhoods of $\mathfrak{e}_i$, for $i\in \{0,1\}$. Observe that $g'$ coincides with $f$ on $S'\sm (A'_0\cup A'_1)$.

\begin{figure}[ht!]
\centering
\includegraphics[width=\textwidth]{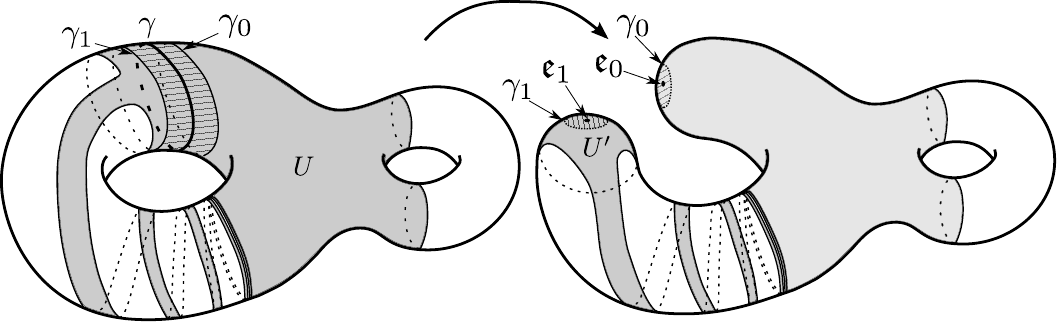}
\caption{$U'$ and the ends compactification of $S_0$}
\label{fig:surgery}
\end{figure}

Since the connected component of $S_0\sm Z(\mathfrak{p})$ containing $D_1$ is precisely $\inter D$, which is $f'$-invariant and contains a unique end of $S_0$ (namely, $\mathfrak{e}_1$), it follows that the connected component of $S'\sm Z(\mathfrak{p})$ containing $D_1$ is a topological disk $U' = D\cup \{\mathfrak{e}_1\}$ such that $\bd_{S'} U' = Z(\mathfrak{p})$ (the latter being true because $D_1$ is a $\mathfrak{p}$-collar). Clearly $\rho(g', U') = \rho(f, \mathfrak{p})$ and $U'$ is $\bd$-nonwandering for $g'$. 
We now apply Theorem \ref{th:main-gen} to $g'$ and $U'$, and translate the information to $f$ and $U$.

Suppose first that case (i) in Theorem \ref{th:main-gen} holds for $g'$ and $U'$. Then $\bd_{S'} U'=Z(\mathfrak{p})$ is aperiodic (so case (i) holds for $f$ and $U$), and there is an annular neighborhood $W'$ of $Z(\mathfrak{p})$ such that any connected component of $S'\sm Z(\mathfrak{p})$ contained in $W'$ is wandering. Note that the connected components of $S'\sm Z(\mathfrak{p})$ containing $\mathfrak{e}_1$ and $\mathfrak{e}_2$ cannot be contained in $W'$, so by reducing $W'$ we may assume that it does not contain $\mathfrak{e}_1$ or $\mathfrak{e}_2$. Thus $W'$ is also an annular neighborhood of $Z(\mathfrak{p})$ in $S_0$ (hence in $S$). Since $W'$ does not contain $U$, if $V$ is a connected component of $S\sm Z(\mathfrak{p})$ contained in $W'$, then $V$ is disjoint from $U$. Thus $V$ is also a connected component of $S'\sm Z(\mathfrak{p})$, and therefore $V$ is wandering for $f'$. Since $f=g'$ in $S\sm U$, it follows that $V$ is wandering for $f$, so the final claim of the theorem holds with $W=W'$.

Now suppose that case (ii) from Theorem \ref{th:main-gen} holds for $U'$ and $g'$. Then $S'$ is a sphere, and there is exactly one fixed point in $\bd_{S'}U'=Z(\mathfrak{p})$. Moreover, all connected components of $S'\sm \cl_{S'} U'$ are wandering. This is not possible in this case, since the connected component of $S'\sm \cl_{S'}U'$ containing $\mathfrak{e}_0$ is invariant; thus case (ii) cannot hold.

Finally, we need to consider the case where $\gamma$ is separating in $S$. We sketch the proof since it is similar to the previous case. For this case, we choose $S_0$ as the connected component of $S\sm \gamma$ that contains $D_1$ (see Figure \ref{fig:surgery-sep}). 
\begin{figure}[ht!]
\centering
\includegraphics[width=\textwidth]{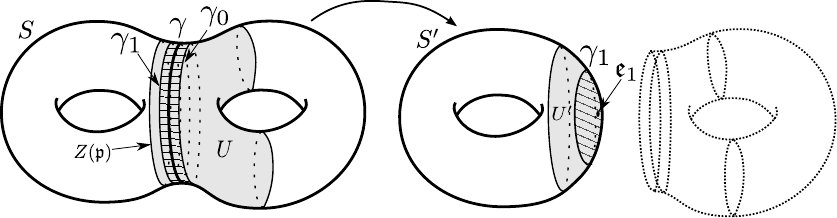}
\caption{Case with $\gamma$ separating $S$}
\label{fig:surgery-sep}
\end{figure}
In this case, $S_0$ has a unique end $\mathfrak{e}_1$, and if $S'=\cIB(S_0) = S_0\cup \{\mathfrak{e}_1\}$, then the connected component $U'$ of $S_0$ that contains $\mathfrak{e}_1$ is a topological disk with $\rho(g', U')=\rho(f,\mathfrak{p})$ and $\bd_{S'}U' = Z(\mathfrak{p})$ (where $g'$ is defined as before). Applying Theorem \ref{th:main-gen} again for $g'$ and $U'$ leaves us with two cases; if case (i) holds for $g'$, then virtually the same argument as before shows that case (i) holds for $f$. On the other hand, if case (ii) holds for $g'$ then one has that $S'$ is a sphere and all connected components of $S'\sm \ol{U}$ are wandering. This implies that $S'\sm \{\mathfrak{e}_1\} = S_0$ is a topological disk containing $Z(\mathfrak{p})$, so $Z(\mathfrak{p})$ is contractible in $S$, and all connected components of $S\sm Z(\mathfrak{p})$ contained in $S_0$ are wandering; thus case (ii) and the final claim of the theorem hold with $W=S_0$.
This completes the proof of the theorem.
\qed

\subsection{Proof of Corollary \ref{coro:general}}

The nonwandering condition implies that $f$ is $\bd$-nonwandering at $\mathfrak{p}$. If the first case from Theorem \ref{th:general} holds, then $Z(\mathfrak{p})$ is an aperiodic invariant continuum contained in an annulus $V$, and all connected components of $S\sm Z(\mathfrak{p})$ contained in $V$ are wandering. Since $f$ is nonwandering, it follows that no connected component of $S\sm Z(\mathfrak{p})$ is contained in $V$. This easily implies that $Z(\mathfrak{p})$ separates $V$ into exactly two annular connected components, so it is annular (alternatively, Theorem \ref{th:koro} implies directly that $Z(\mathfrak{p})$ is annular).

If the second case from \ref{th:general} holds, then $Z(\mathfrak{p})$ has a unique fixed point, no other periodic point, and has a neighborhod $D$ homeomorphic to a disk such that every component of $S\sm Z(\mathfrak{p})$ contained in $D$ is wandering. This means that $D\sm Z(\mathfrak{p})$ has no relatively compact (in $D$) connected components, and since $Z(\mathfrak{p})$ is compact we conclude that $Z(\mathfrak{p})$ is non-separating in $S$. This implies that $Z(\mathfrak{p})$ is cellular, completing the proof.
\qed

\subsection{Proof of Theorem \ref{th:arc-general}}

We only sketch the proof since it is very similar to the proof of Theorem \ref{th:general}. The surgery from the proof of Theorem \ref{th:general} allows us to obtain a simply connected set $U'$ on a surface $S'=\cIB(S_0)$, where $S_0$ is a connected component of $S\sm \gamma$ for some loop $\gamma$ bounding a $\mathfrak{p}$-collar. Moreover, $f$ coincides with $g'$ outside an annular neighborhood $A\subset U$ of $\gamma$, and $\rho(g',U')=\rho(f,\mathfrak{p})=\alpha$. Note that the genus of $S'$ is not greater than the genus of $S$; thus letting $N=\max_{0\leq k\leq g} N_{\alpha,k}$ (Where $N_{\alpha,k}$ is the number from Theorem \ref{th:arclemma}) we have that there is a compact set $K\subset U'$ such that any $N$-translation arc for $g'$ intersecting $\bd U$ must also intersect $K$. Let $C'$ be a $\mathfrak{p}$-collar disjoint from $A$ such that  $K\subset U\sm C'$, and let $C$ be a smaller $\mathfrak{p}$-collar so that $f^n(C)\subset C'$ for $0\leq n\leq N$. Suppose for contradiction that there is an $N$-translation arc $\gamma$ for $f$ such that $\gamma \cap (U\sm C)=\emptyset$. Then $f^k(\gamma)$ is disjoint from $U\sm C'$ for $0\leq k\leq N$, and since $U\sm C'\subset S\sm A$, this implies that $\gamma$ is also an $N$-translation arc for $g'$ intersecting $\bd U'$. But $\gamma$ is disjoint from $U'\sm C'$, and thus disjoint from $K$, contradicting Theorem \ref{th:arclemma} due to our choice of $N$.
\qed

\section{Applications to generic area-preserving diffeomorphisms}
\label{sec:mather}

The main theorem of this section is an application of our previous theorems that generalizes Mather's results from \cite{mather-area} in two ways: it works for any $r \geq 1$ (since it does not rely on Moser stability of elliptic points), and the conclusion is much stronger (since one gets that there is no periodic point, instead of just saying that the prime end rotation number is irrational). Before stating the theorem, we need a definition.

Given a closed surface $S$, denote by $\mathcal{G}_r$ the set of all area preserving $C^r$ diffeomorphisms of $S$ satisfying the following two properties:
\begin{itemize}
\item[(G1)] All periodic points of $f$ are either hyperbolic or elliptic, and there are no saddle connections;
\item[(G2)] If $p$ is an elliptic periodic point and $U$ is a neighborhood of $p$, then there is an open disk $D$ containing $p$, contained in $U$, and bounded by finitely many pieces of stable and unstable manifolds of some hyperbolic periodic orbit $q$ intersecting transversely (see Figure \ref{fig:zehnder}).
\end{itemize}

\begin{figure}[ht!]
\centering
\includegraphics[height=4cm]{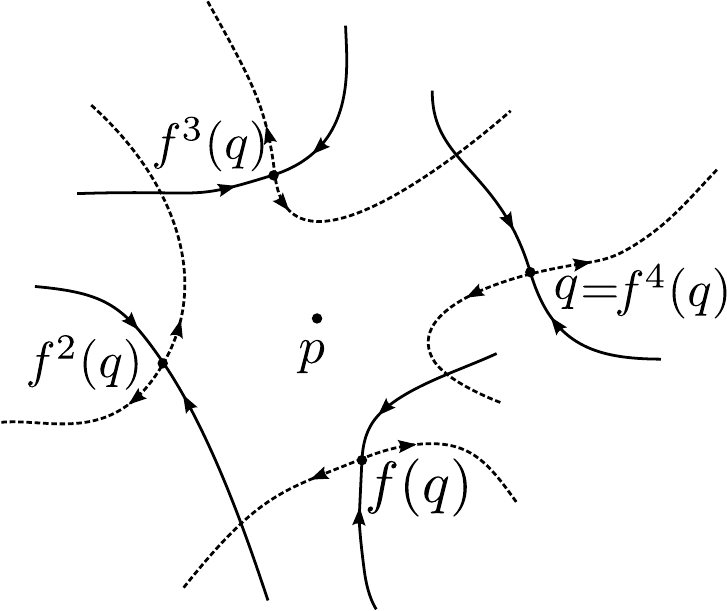}
\caption{Property (G2)}
\label{fig:zehnder}
\end{figure}

\begin{remark}  Robinson proved that, for any $r \geq 1$, property (G1) is $C^r$-generic \cite{robinson-ks}, and (G2) is $C^r$-generic due to Zehnder \cite{zehnder}. Thus $\mathcal{G}_r$ is residual in the space of area preserving $C^r$ diffeomorphisms of $S$.
\end{remark}

\begin{remark} The idea of replacing the Moser genericity condition from Mather's arguments by Condition (G2), thus removing the requirement that $r$ be large, was also used by Girard \cite{girard} to prove that $C^r$ generically ($r\geq 1$) for an area preserving diffeomorphism there is no invaraint curve of rational rotation number. 
\end{remark}

\begin{theorem}\label{th:mather-extended}
If $f$ is a $C^r$-generic area preserving diffeomorphism of a closed orientable surface $S$ (namely if $f\in \mathcal{G}_r$, $r\geq 1$), and $\mathfrak{p}$ is a regular periodic end of an open $f$-invariant connected set $U$, then $Z(\mathfrak{p})$ is an aperiodic annular continuum.
\end{theorem}

\begin{remark} For the case of the sphere, this theorem can be proved using a classic result of Pixton \cite{pixton} (combined with results of Mather \cite{mather-area}) which guarantees that under the generic hypotheses of the theorem, every hyperbolic periodic point has a homoclinic intersection. In fact, this is done in \cite{franks-lecalvez}. A similar argument can be done on the torus using a result of Oliveira \cite{oliveira}. However, for surfaces of higher genus, is it not known whether a $C^r$-generic diffeomorphism has a homoclinic intersection for each (or even for some) hyperbolic periodic point. Furthermore,  for maps in $\mathcal{G}_r$ (as in Theorem \ref{th:mather-extended}) it is actually not true that some hyperbolic periodic point has a homoclinic intersection: one may produce an example in the surface of genus $2$ which is the time-one map of a flow with exactly two hyperbolic saddles and no saddle connection or additional periodic points (where every orbit other than the saddles is dense).
\end{remark}

\subsection{Saddle hyperbolic fixed points}

We begin with a result about hyperbolic points only which is a consequence of Theorem \ref{th:arclemma}

If $p$ is a hyperbolic fixed point for a surface diffeomorphism, a \emph{branch} of $p$ is a connected component of $W^s(p)\sm{p}$ or $W^u(p)\sm{p}$. There are two stable branches and two unstable ones. If $f$ maps each branch to itself, then we say $f$ is \emph{branch-preserving} at $p$. 

\begin{theorem} \label{th:no-saddle}
Suppose that $f\colon S\to S$ is a homeomorphism of a closed orientable surface, and $U\subset S$ is an open $f$-invariant topological disk on which $f$ is $\bd$-nonwandering and $\rho(f,U)$ is nonzero. Assume further that $f$ is $C^1$ in a neighborhood of $\bd U$. Then there is no branch-preserving hyperbolic fixed point in $\bd U$.
\end{theorem}
\begin{proof}
Suppose $p\in \bd{U}$ is a branch-preserving hyperbolic fixed point. Using a local linearizing neighobhrood near $p$, one sees that for each $N\in \N$ there is a neighborhood $V_N$ of $p$ such that $V_N\sm \{p\}$ is foliated by $N$-translation arcs, so by Theorem \ref{th:arclemma} we have that $V_N$ does not meet $\bd U$ if $N$ is large enough, contradicting the fact that $p\in \bd U$. 
\end{proof}

\subsection{Proof of Theorem \ref{th:mather-extended}}

First, we show that the prime ends rotation number $\rho(f,\mathfrak{p})$ is irrational. 
This, in fact, is a theorem of Mather \cite[Theorem 5.1]{mather-area}   with the slightly different assumptions.
While Mather assumes the Moser stability of elliptic points, we instead assume Property (G2).
This changes only the part of Mather's proof in which it is shown that there are no elliptic periodic points in $Z(\mathfrak{p})$. 

To see that  $Z(\mathfrak{p})$ contains no elliptic periodic point, suppose for contradiction that $p\in Z(\mathfrak{p})$ is an elliptic periodic point. Since $Z(\mathfrak{p})$ consists of more than a single point, we may choose a neighborhood $V$ of $p$ that does not contain $Z(\mathfrak{p})$. By Property (G2), there is some hyperbolic periodic point $q$ and a disk $D$ bounded by finitely many arcs of the stable and unstable manifolds of $q$ such that $p \in D$ and $D \subset V$. We may write $\bd D=\bigcup_{i=1}^{n} W_i$ where each $W_i$ is an arc of stable or unstable manifold of an iterate of $p$, and we order them cyclically, \ie in a way that $W_i$ intersects $W_{i+1}$ for $1\leq i\leq n-1$.

By \cite[Corollary 8.3]{mather-area}, if a branch of the stable or unstable manifold of $q$ intersects $Z(\mathfrak{p})$, then the whole branch is contained in $Z(\mathfrak{p})$.  Since $Z(\mathfrak{p})$ is connected and contains a point in the complement of $D$, it follows that $\bd D$ intersects $Z(\mathfrak{p})$. Thus $W_i$ intersects $Z(\mathfrak{p})$ for some $i$. Assume without loss of generality $i=1$. Then as we mentioned \cite[Corollary 8.3]{mather-area} implies that $W_1\subset Z(\mathfrak{p})$. Since $W_2$ intersects $W_1\subset Z(\mathfrak{p})$, repeating this argument we have that $W_2\subset Z(\mathfrak{p})$, and by induction we conclude that $\bd D = \bigcup_{i=1}^n W_i\subset Z(\mathfrak{p})$. But $U$ is connected and intersects both $D$ and $S\sm D$ (because $Z(\mathfrak{p})$ is not entirely contained in $V$ and $D\subset V$). This means that $U$ intersecs $\bd D\subset Z(\mathfrak{p})$, contradicting the fact that $Z(\mathfrak{p})\subset \bd U$. Thus there are no elliptic periodic points in $Z(\mathfrak{p})$, and $\rho(f, \mathfrak{p})$ is irrational. 

The key part of the proof is Corollary \ref{coro:general}, which now implies that either $Z(\mathfrak{p})$ is an aperiodic annular continua, which is what we wanted to show, or $Z(\mathfrak{p})$ is a cellular continuum with empty interior and with a fixed point $p$. Assume the latter case. Since we showed that $p$ cannot be elliptic, condition (G1) implies that $p$ is hyperbolic.
The fact that $Z(\mathfrak{p})$ is contractible implies that there is an open topological disk $D\subset S$ such that $\bd D$ is a loop in $U$ and $Z(\mathfrak{p})\subset D$. 
Suppose that $S$ is not a sphere, so its universal covering $\hat{S}$ is homeomorphic to $\R^2$. If $\pi\colon \hat{S}\to S$ is the universal covering map, $\hat{p}\in \pi^{-1}(p)$ and $\hat{f}$ is a lift of $f$ such that $\hat{f}(\hat{p})=\hat{p}$, then there is a disk $\hat{D}\subset \hat{S}$ such that $\pi|_{\hat{D}}$ is a diffeomorphism onto a neighborhood of $Z(\mathfrak{p})$. Let $K$ be the connected component of $\pi^{-1}(Z(\mathfrak{p}))$ in $\hat{D}$. Then $K$ is $\hat{f}$-invariant because it contains a fixed point $\hat{p}$, and $K$ has empty interior.
Let $S'=\hat{S}\sqcup \{\infty\}$ be the one-point compactification of $\hat{S}$, which is a sphere, and $f'$ is the map induced by $\hat{f}$. Then $U'=S'\sm K$ is a topological disk, $\bd U'=K$, and $f'$ is $C^1$ in the neighborhood $\hat{D}$ of $K$. The fact that $\pi_{\hat{D}}\hat{f}=f'\pi_{\hat{D}}$ easily implies that $f'$ is $\bd$-nonwandering and $\rho(f', U')=\rho(f, \mathfrak{p})$. Since $f'$ has a hyperbolic fixed point $\hat{p}$ in $U'$, which we may assume branch-preserving by using $f'^2$ instead of $f$, this contradicts Theorem \ref{th:no-saddle}. 

Thus there are no periodic points in $Z(\mathfrak{p})$, and it follows from Theorem \ref{th:koro} that $Z(\mathfrak{p})$ is an annular continuum, completing the proof.
\qed

\begin{remark} After using Corollary \ref{coro:general}, which guarantees that $Z(\mathfrak{p})$ is contractible (if it contains a fixed point), it would also be possible to finish the proof using arguments such as Pixton's theorem \cite{pixton} to conclude that there is a homoclinic intersection of $p$, from which one easily obtains a contradiction (see also \cite{kn-erratum}). Note that this proof would not work without the aid of Corollary \ref{coro:general}, since an analogous of Pixton's theorem is not available on surfaces of genus greater than $1$.
\end{remark}

\subsection{Aperiodicity of the boundary of complementary domains}\label{sec:mather-complementary}

Recall that a complementary domain is any connected component of the complement of a nontrivial continuum. A complementary domain always has finitely many ideal boundary points (see, for instance, \cite[Lemma 2.3]{mather-area}), and none of its boundary components are single points. Therefore, Theorem \ref{t:mather} from the introduction follows immediately from the next result.

\begin{corollary} \label{cor:mather-extended}
Let $f$ be a $C^r$-generic area preserving diffeomorphism of a closed orientable surface $S$ (namely, let $f\in \mathcal{G}_r$, $r\geq 1$), and  $U$ a periodic open set with finitely many ends. Then the boundary of $U$ is a finite disjoint union of aperiodic annular continua and periodic orbits.
 \end{corollary}
\begin{proof}
Since $\bdIB(U)$ consists of finitely many points which are periodic by the map induced by $f$ on $\cIB(U)$, one may choose $n\in\N$ such that $f^n$ fixes every ideal boundary point of $U$, and it follows from Theorem \ref{th:mather-extended} applied to $f^n$ that for each $\mathfrak{p}\in \bdIB(U)$, if $Z(\mathfrak{p})$ has more than one point then it is an aperiodic annular continuum invariant by $f^n$. Since there are finitely many ends, is easy to see from the definition of ideal boundary that $\bd U = \bigcup_{\mathfrak{p}\in \bdIB(U)} Z(\mathfrak{p})$. Thus $\bd U$ is a finite union of aperiodic (but possibly non-disjoint) $f^n$-invariant annular continua, together with a finite set of fixed points of $f^n$ (corresponding to the ends where $Z(\mathfrak{p})$ is a sinlg epoint). In particular, $\bd U$ has finitely many connected components, and the ones that are not periodic points are aperiodic $f^n$-invariant continua. By Theorem \ref{th:koro}, this implies that each connected component of $\bd U$ that is not a periodic point is in fact an aperiodic annular continuum, completing the proof.
\end{proof}

\begin{remark}\label{rem:erratum} This result (with a minor variation in the generic conditions) is mistakenly attributed to Mather in \cite[Theorem 12]{koro-nassiri}, leaving a gap in the article. In \cite{kn-erratum}, this mistake is corrected by avoiding the use of Theorem 12 from \cite{koro-nassiri}. Using Corollary \ref{cor:mather-extended}, there is no gap and the erratum \cite{kn-erratum} is no longer necessary.
\end{remark}

\subsection{Closure of branches of stable and unstable manifolds}

The following results are proved in \cite{mather-area} under different assumptions (which are $C^r$ generic, only for large $r$).

\begin{corollary} \label{cor:mather-branch cl}
Let $f$ be a $C^r$-generic area preserving diffeomorphism of a closed surface $S$ (namely, let $f\in \mathcal{G}_r$, $r\geq 1$). Suppose $p$ is a hyperbolic periodic point. Then,
\begin{enumerate}
\item If $p$ belongs to some periodic continuum $K$, then then the stable and unstable manifolds of $p$ are contained in $K$.
\item All four stable/unstable branches of $p$ have the same closure.
\end{enumerate}
 \end{corollary}
\begin{proof}
To prove (1), let $\gamma$ be a branch of $p$.
Suppose for contradiction that $\gamma \not\subset K$. Then, $\gamma$ intersects a 
connected component $V$ of $S\sm K$.  This implies that $p\in \bd V$. But $V$ is a periodic complementary, so this contradicts Corollary \ref{cor:mather-extended}. 

To prove (2), let $\gamma_1$ and $\gamma_2$ be two branches of $p$. Then $\cl(\gamma_1)$ is a periodic continuum and $p\in K_1$, so it follows from (1) that $\gamma_2 \subset \cl(\gamma_1)$.  Similarly, $\gamma_1 \subset \cl(\gamma_2)$, completing the proof.
\end{proof}

\subsection{Density of stable manifolds}\label{sec:xia}
In \cite{franks-lecalvez} it is proved that {\it for a $C^r$ generic area-preserving diffeomorphism of the sphere, the union of stable manifolds of hyperbolic period points is dense}. One of the reasons for the interest in this type of results is that they aim in the direction of one of the oldest open problems in dynamics, i.e. the conjecture of Poincar\'e on the density of periodic points for generic area preserving diffeomorphisms. 
  
Corollary \ref{cor:mather-extended}  fills a gap in the proof of the generic density of stable manifolds for arbitrary surfaces given by Xia in \cite{xia}, (particularly, in the proof of Lemma 2.3) which is similar to the one mentioned in Remark \ref{rem:erratum}. Thus, we can state that the following theorem holds.

\begin{theorem}
For a $C^r$ generic area-preserving diffeomorphism of a closed surface, the union of stable manifolds of hyperbolic periodic points is dense.
\end{theorem}

The second claim from the main theorem of \cite{xia} states that an open set without periodic points is necessarily contained in the closure of the stable manifold of a given periodic point. This seems to be more delicate, and we are unable to complete its proof.

\section{Examples}
\label{sec:examples}

\subsection{Necessity of the $\bd$-nonwandering hypothesis}
\begin{example}[Many fixed points]\label{ex:walker} The $\bd$-nonwandering hypothesis assumed in Theorem \ref{th:main} and related results is really necessary, as an example due to Walker shows. The example, found in \cite{walker}, is of a $C^1$-diffeomorphism $f$ of the sphere, which has an invariant open simply connected set $U$ such that $\rho(f,U)$ is irrational, but $\bd U$ contains a circle of fixed points. Naturally, the example is not $\bd$-nonwandering; in fact $\bd U$ is a basin boundary (so there is a closed disk $B\subset U$ such that $f(B)\subset \inter{B}$ and $\bd U = \bigcap \ol{U}\sm f^{-n}(B)$. This implies that the backward orbit of sufficiently small cross-sections remains arbitrarily close to $\bd U$, thus failing to be $\bd$-nowandering. 

This example also shows that the theorem fails if one weakens the definition of $\bd$-nonwandering to require only that for some compact $K\subset U$, the set $U\sm K$ does not contain the forward orbit of any wandering cross-section (\ie one really must require that forward \emph{and backward} orbit of a wandering cross-section leaves $U\sm K$).
Figure \ref{fig:walker} shows what the boundary of $U$ looks like. In this case, $U$ is considered to be the unbounded connected component of the complement of the continuum depicted (together with the point at $\infty$). The boundary of $U$ consists of the circle $C$ of fixed points (the smaller circle) together with a set of ``hairs'' accumulating on $C$. The endpoint of each ``hair'' lies on a larger circle (dotted circle) and forms a Cantor set which is invariant by $f$; the dynamics on that Cantor set corresponds to that of a Denjoy example. However any other point outside $\bd U$ is repelled towards $\infty$. 
 
\begin{figure}[h] 
\includegraphics[width=5cm]{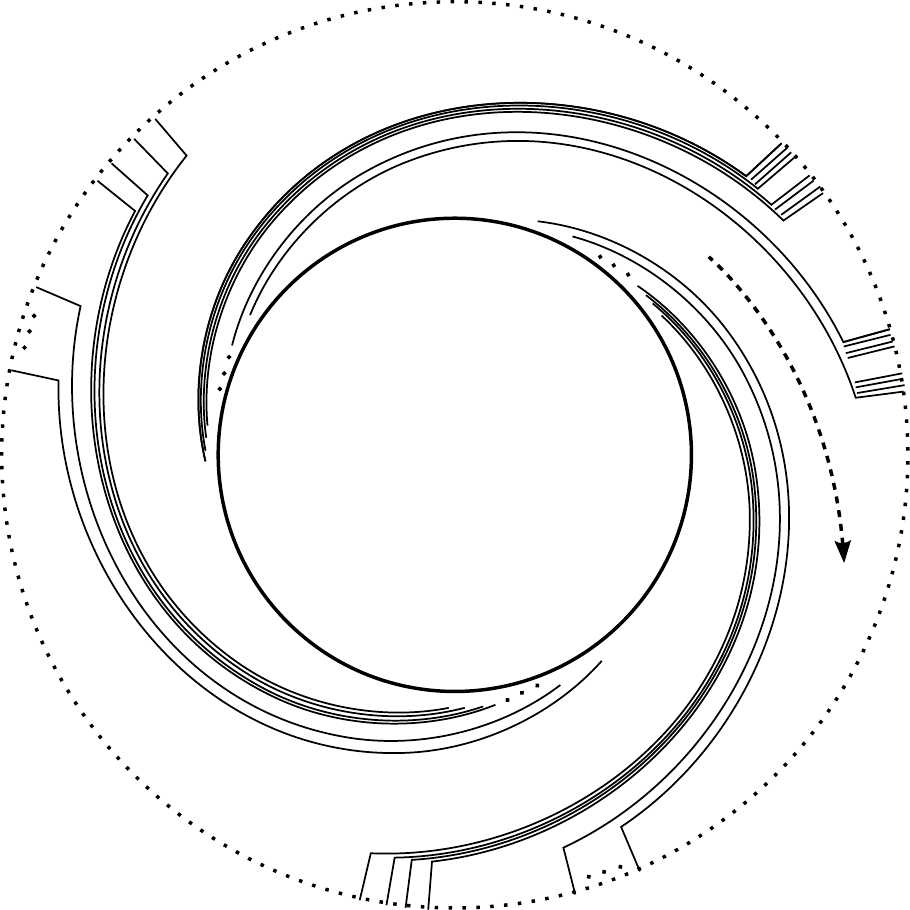}
 \caption{Walker's example}
\label{fig:walker}
\end{figure}
\end{example}

\begin{example}[Non-contractible boundary] \label{ex:walker2}
Let $f'$ be the map from Example \ref{ex:walker} and $U'$ the open set with $\rho(f',U')$ irrational. Note that $B=\ol{U}'$ is a closed topological disk bounded by a circle of fixed points. We may assume (by using an appropriate conjugation) that $B$ is the unit square $[0,1]^2$ on $\R^2$. Define a map $f\colon \T^2\to \T^2$ by $f(z) = \pi(f'(\hat{z}))$, where $\hat{z}$ is the element of $\pi^{-1}(z)$ in $[0,1)^2$. Then $f$ is a homeomorphism of $\T^2$, with two circles of fixed points, $C_1=\pi(\R\times \{0\})$ and $C_2=\pi(\{0\}\times \R)$.  Let $U = \pi(U')$. Then clearly $C_1\cup C_2\subset \bd U$, and $\rho(f,U)$ is irrational. Thus we have produced an example on $\T^2$ with an invariant open topological disk $U$ with irrational prime ends rotation number and a non-contractible boundary.
\end{example}

\subsection{On the hypotheses of Theorem \ref{th:arclemma}}
Here we give two simple examples that show that the number $N_{\alpha, g}$ from Theorem \ref{th:arclemma}) really depends on both the prime ends rotation number $\alpha$ and the genus of the surface $S$.

\begin{example}[dependence on $\alpha$] \label{exm1}
Given $N\in \N$, $N>1$, let $f$ be the rotation by $\alpha = 2\pi N$,  on  the sphere $\C \cup \{\infty\}$, and let $U$ be the unit disk.  Clearly, $U$ is $f$-invariant and its  prime ends rotation number is equal to $1/N$. Then the boundary of $U$ contains an $N$-translation arc, so the constant $N_{\alpha, 0}$ is necessarily greater than $N$.
\end{example}

\begin{example}[dependence on the genus] \label{exm2} 
Let $N$ be a positive integer. We introduce an area-preserving homeomorphism $f$ on a closed surface $S$ (of genus $n=2N$) that leaves invariant a simply connected open set $U$, and admits an $N$-translation arc outside any given compact set $K\subset U$ but intersecting both $U$ and its complement, while $\rho(f,U) = 1/2$.

Let $S_1$ be the sphere $\R^2\cup \{\infty\}$, and  
let $R$ be the rotation of $180$ degrees centered at the origin $o$.
One can easily define a volume preserving diffeomomorphism $g$ such that $g$ commutes with $R$, and $o$ is a hyperbolic fixed point with a homoclinic loop. We denote by $U$, $V$ and $V'=R(V)$ the connected components of the complement of the homoclinic loop.
See Figure \ref{fig:arc}.
We further assume that there is a $g^n$-invariant smooth closed disk $\overline{D}$ in $V$, whose orbit is the union $n$ disjoint closed disks. Let $D':=R(D)$, where $D$ the interior of $\overline{D}$.
(One may define $g$ by using the time-one map of the Hamiltonian flow associated to a suitable height function.)

\begin{figure}[h] 
\includegraphics[width=0.8\textwidth]{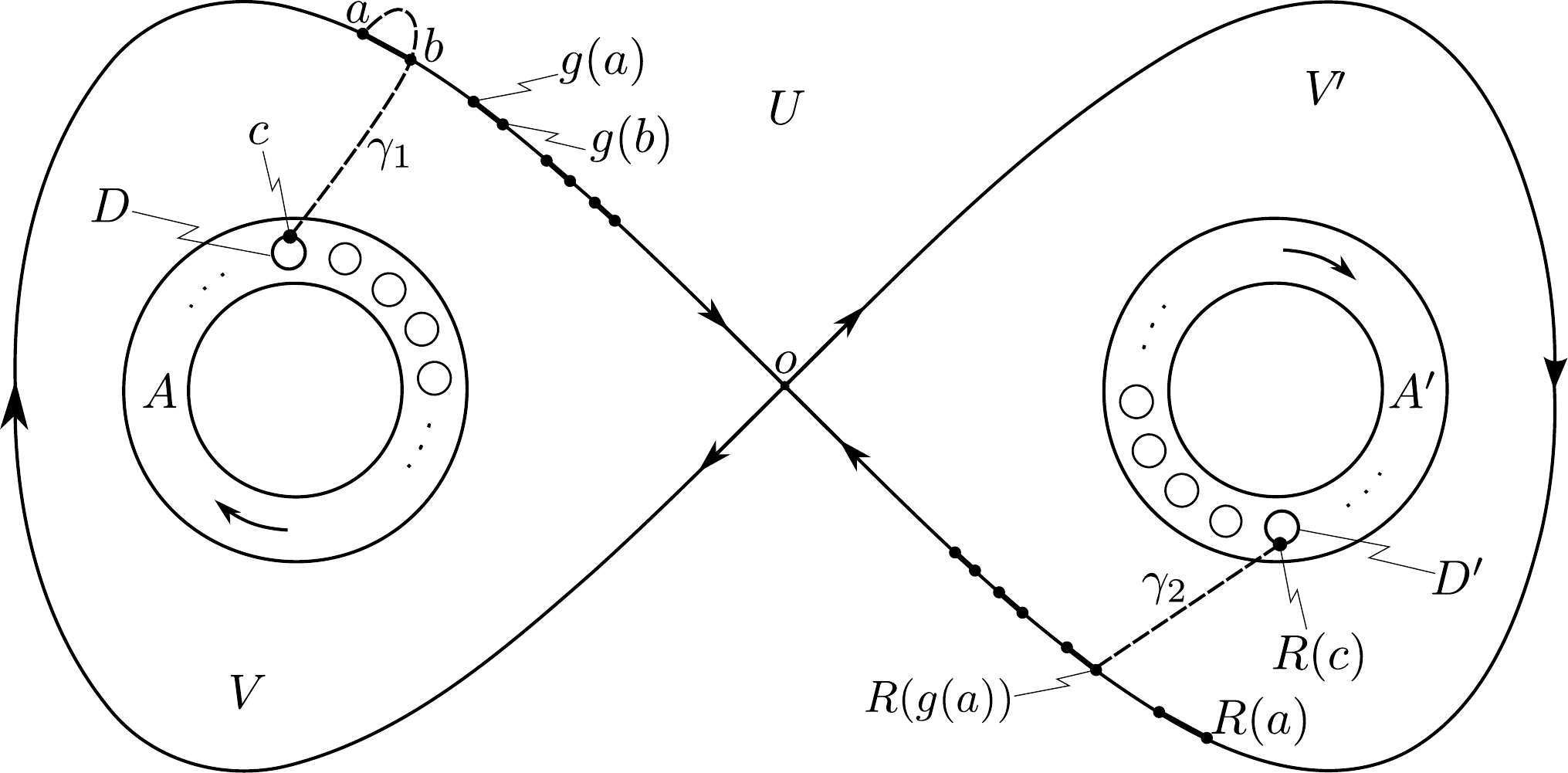}
 \caption{}
\label{fig:arc}
\end{figure}

Now, let $\gamma_1$ and $\gamma_2$ be two arcs as in Figure \ref{fig:arc}. 
We assume that $\gamma_1$ intersects $U$ and its extremal points are $a\in \partial V$ and $c\in\partial D$, and the extremal points of $\gamma_2$ are $R(c)$ and $R(g(a))$. 
We also assume that $g^i(\gamma_j)$ are pairwise disjoint for $j=1,2$ and $i=0,1, \dots, n-1$.

We will make a surgery on $S_1$ as it follows. 
Let $W$ be the $g$-orbit of $D\cup  D'$ which is the union of $2n$ disjoint disks. Let $S_0 = S_1 \setminus W$, 
$R_0:= R|_{S_0}$, $g_0:= g|_{S_0}$ and $f_0:= (R\circ g)|_{S_0}$. Then $g_0$ commutes with $R_0$ and $f_0\circ R_0 = R_0 \circ f_0 =g_0$.
Observe that $\partial S_0$ is invariant for $g$ and for $R$.
By the assumptions, $\overline{D}$ is diffeomorphic to the closed unit disk $\overline{\D}\subset \R^2$ by a diffeomorphism $\tau$ such that $\tau(c)=(0,1)$. 
Let $\psi_1 : \partial D \to \partial D$,  $\psi_1 = \tau^{-1}\circ L \circ \tau$, where $L(x,y)=(-x,y)$ on $\R^2$.  
Then $\psi_1(c)=c$.
Let $\psi: \partial S_0 \to \partial S_0$ such that $\psi(p)=g^i \circ \psi_1 \circ g^{-i} (p)$ if $p \in g^i(\partial D)$ and 
$\psi(p)=R\circ g^i \circ \psi_1 \circ g^{-i} \circ R^{-1}(p)$ if $p \in g^i(\partial D')$.
It  easy to see that on  $\partial S_0$ one has $\psi\circ R = R\circ \psi$, $\psi^2=R^2=(\psi\circ R)^2=\id$ and $\psi\circ g = g\circ \psi$.

Now, let $S$ be the quotient $S_0/(R\circ \psi)$. Then, it follows that $S$ is an orientable surface of genus $n$ and the homeomorphisms $R_0$ and $g_0$ extend to the quotient surface $S$, and so does $f_0$.
We denote by $f$ the homeomorphism on $S$ obtained from $f_0$. 
It then follows that $\gamma_1\cup\gamma_2$ projects to an arc $\gamma$ (note that $c$ and $R(c)$ are identified), which is an $N$-translation arc for $f$ that intersects the $f$-invariant disk $U$ and its complement. Also one could have chosen $\gamma$ lying outside any previously given compact set $K\subset U$, and it is easy to verify that $\rho(f,U) = 1/2\, (\mathrm{mod}\, \Z)$. 
\end{example}

\subsection{Optimality of Theorem \ref{th:main-gen} and its corollaries.}
The next examples show that the second case in Theorem \ref{th:main-gen} and Corollary \ref{coro:main-gen} cannot be excluded, even for smooth area preserving maps.

\begin{example}[{Smooth example with a fixed point}]
 \label{exm:invers2} Let us show that, at least for some irrational values of $\alpha$, there are $C^\infty$ area-preserving examples where the second case holds in Theorem \ref{th:main}.

Let $f$ be a transitive area preserving $C^\infty$ diffeomorphism on the sphere $\mathbb{S}^2$ obtained by the Anosov-Katok method so that it fixes only two points, say the north and the south poles $p,q$, and moreover almost every point has an irrational rotation number $\alpha$ for $f|_A$, where $A=S\setminus \{p,q\}$. 
Such a diffeomorphism exists (it is done explicitly in \cite{crovisier-grenoble}; we refer the reader to \cite{fk} for a general survey on the Anosov-Katok method).

Let $D$ be an open disk containing $p$ and disjoint from a neighborhood $V$ of $q$.
Let $K$ be the connected component of $\bigcap_{n\in \Z} f^n(\ol{D})$ containing $p$. Then $K$ is compact, connected and invariant. Moreover, $K$ is non-separating, because the complement of $K$ contains the orbit of $V$, which is dense and connected. Thus $U = \mathbb{S}^2\sm K$ is an open invariant topological disk. 

It is possible to show that $K$ intersects $\bd D$. This follows from the results of \cite{LY-manuscript}, which imply that if $K$ does not intersect $\bd D$ then either $p$ is ``dissipative'' (\ie every neighborhood of $0$ contains a loop around $0$ that is disjoint from its image) or $p$ ``saddle-point'' (there exists an integer $m$ such that the Lefschetz index $i(f^m,p)$ is non-positive). Indeed, $p$ cannot be dissipative because $f$ preserves area, and it cannot be a saddle-point because $i(f^m,p) = 1 = i(f^m,q)$. The latter follows from the fact that the two indices must add up to $2$ by Lefschetz' formula, and none of the two indices is greater than $1$ by a theorem of Pelikan and Slaminka \cite{pelikan} (see also \cite{lecalvez-viz}). 

Finally the fact that $\rho(f,U) = \alpha$ follows from \cite[\S6]{lecalvez-viz}, so the second case from Theorem \ref{th:main-gen} and Corollary \ref{coro:main-gen} holds. 
\end{example} 

\begin{example}[Example with a fixed point and any irrational $\alpha$]\label{exm:invers}
The previous example shows that case (ii) in Theorem \ref{th:main} can hold for smooth maps, but it due to the technique used in the construction of the example, the number $\alpha$ will always be a Liouville number. It is possible that no such example exists if $\alpha$ is Diophantine and $f$ is smooth enough (due to a KAM-type phenomenon). However if one drops the smoothness condition, one may produce homeomorphisms with the same properties as the previous example, for any given irrational $\alpha\in \R/\Z$.

We sketch a way to obtain such example using the Denjoy-Rees technique from \cite{denjoy-rees}. As mentioned in section C2 of \cite{denjoy-rees}, the techniques of said article allow to obtain for any $\alpha\in \R\sm \Q$ a homeomorphism $f\colon \mathbb{S}^2 \to \mathbb{S}^2$ which is semi-conjugate to a rigid rotation by angle $\alpha$ of the sphere (fixing the north and south pole) via a surjection $h\colon \mathbb{S}^2\to \mathbb{S}^2$ which leaves the equator $C$ invariant, such that there is a vertical segment $\ell$ that touches the equator at its base and is mapped by $h$ to a single point on the equator, $\diam(f^k(\ell))\to 0$ as $k\to \pm \infty$ and $h$ is injective outside the orbit of $\ell$ (see Figure \ref{fig:denjoy}), so that $\Lambda =C \cup \bigcup_{k\in \Z} f^k(\ell)$ is an $f$-invariant connnected set. Let $U$ be the upper connected component of $\mathbb{S}^2\sm \Lambda$. Then $f|_U$ preserves a non-atomic Borel probability measure $\mu$ of full support (obtained by pull-back of the normalized Lebesgue measure on the upper half-sphere via $h$), and using the monotone semi-conjugation $h$ it is easy to verify that $\rho(f,U)=\alpha$.

\begin{figure}[h]
\centering
$\begin{array}{c@{\hspace{10pt}}c}
\includegraphics[width=0.5\linewidth]{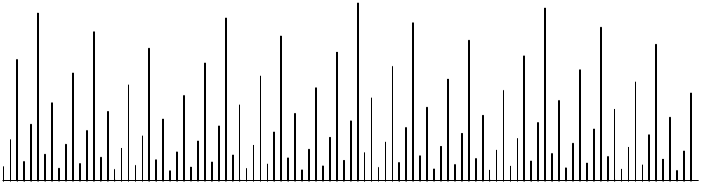} &
	\includegraphics[height=3.5cm]{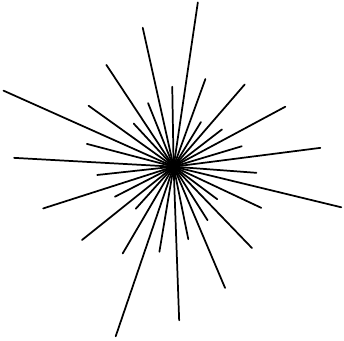} \\ 
\end{array}$
\caption{$\Lambda$ before and after collapsing $C$}\label{fig:denjoy}
\end{figure}

If we collapse the circle $C$ to a single point $p_0$, we obtain a surface $S'\simeq \mathbb{S}^2$, a homeomorphism $f'\colon S'\to S'$ for which $p_0$ is a fixed point, and an $f'$-invariant set $\Lambda'$ which contains $p_0$ and is the boudnary of $U$ in $S'$. Furthermore, by Corollary \ref{coro:prime-rot}, we have that $$\rho(f', U, S) = \rho(f',U,S\sm \{p_0\}) = \rho(f,U, \mathbb{S}^2\sm C) = \rho(f,U,\mathbb{S}^2)=\alpha$$ 
The measure $\mu'$ induced by $\mu$ in $S$ is a non-atomic Borel probability measure of full support, and by the Oxtoby-Ulam theorem \cite{oxtoby-ulam} we may choose a homeomorphism $\phi\colon S\to S$ such that $f'' = \phi f \phi^{-1}$ preserves Lebesgue measure, so $f''$ has the required properties.
\end{example}

\section{Appendix: Proof of Theorem \ref{th:prime-equiv}}

\subsection{Proof of parts (1) and (3)}
We begin by proving (3). First observe that any chain $\mc{C}$ of $U$ as a subset of $S$ is equivalent to some chain $\mc{C}'=(D_i')_{i\in \N}$ such that $\bd_U D_i'$ is a cross-cut of $U$ with endpoints in $S_0$: this is easily done by perturbing the endpoints of the cross-cuts defining the given chain, using the fact that $S\sm S_0$ is totally disconnected (we leave the details for the reader).
Thus a chain of $U$ as a subset of $S$ is always equivalent to some chain of $U$ as a subset of $S_0$. If, in addition, $\mc{C}$ is a prime chain of $U$ as a subset of $S$, then it follows from the definition that $\mc{C}'$ is a prime chain of $U$ as a subset of $S_0$.  

There is a natural way to define an extension $i_*\colon \cPE(U, S)\to \cPE(U, S_0)$ of the inclusion $i\colon U \to \cPE(U,S_0)$, as follows: Let $p\in \bdPE(U,S)$ be a prime end, and let $\mc{C}$ be a representative prime chain. By the previous observation we may choose $\mc{C}$ so that it is also a prime chain of $U$ as a subset of $S_0$. If we let $i_*(p)$ be the prime end in $\bdPE(U, S_0)$ represented by $\mc{C}$, then $i_*(p)$ does not depend on the choice of $\mc{C}$, but only on the choice of $p$, so it is well defined. Thus, letting $i_*(z)=z$ for $z\in U$, we have defined a map $i_*\colon \cPE(U, S)\to \cPE(U, S_0)$. 

To show that $i_*$ is surjective, one needs to verify that if $\mc{C}$ is a prime chain of $U$ as a subset of $S_0$, then it is also a prime chain as a subset of $S$ (for that would imply that the prime end corresponding to $\mc{C}$ in $\bdPE(U, S)$ is mapped by $i_*$ to the prime end corresponding to $\mc{C}$ in $\bdPE(U, S_0)$). But this is true, because if $\mc{C}$ is divided by some chain $\mc{C}'$ of $U$ as a subset of $S$, then from the remark at the beginning of the proof, $\mc{C}'$ is equivalent to some chain $\mc{C}''$ of $U$ as a subset of $S_0$, and so $\mc{C}''$ divides $\mc{C}$. But the latter, together with the fact that $\mc{C}$ is a prime chain of $U$ as a subset of $S_0$, implies that $\mc{C}$ divides $\mc{C}''$. Since $\mc{C}''$ is equivalent to $\mc{C}'$, we conclude that $\mc{C}$ divides $\mc{C}'$, showing that $\mc{C}$ is indeed a prime chain of $U$ as a subset of $S$.  Therefore $i_*$ is surjective.

To see that $i_*$ is injective, note that if $p$ and $p'$ are elements of $\bdPE(U,S)$ such that $i_*(p)=i_*(p')$, the definition of $i_*$ implies that $p$ and $p'$ are represented in $\bdPE(U,S)$ by chains $\mc{C}$ and $\mc{C}'$ (respectively) which are also prime chains of $U$ as a subset of $S_0$ representing the same prime end $i_*(p)=i_*(p')$ in $\bdPE(U,S_0)$. This means that $\mc{C}$ and $\mc{C}'$ are equivalent, and so $p=p'$.

For the continuity of $i_*$, it suffices to show that if $D$ is a cross-section of $U$ as a subset of $S_0$ then $i_*^{-1}(D\cup \iPE_{U,S_0}(D)) = D\cup \iPE_{U, S}(D)$ (recalling that $\iPE_{U,S}(D)$ denotes the prime ends in $\bdPE(U,S)$ which are represented by some chain that divides $D$). But this is clearly true, since as we mentioned at the beginning of the proof, any prime chain $\mc{C}$ of $U$ in $S$ is equivalent to some prime chain $\mc{C}'$ of $U$ in $S_0$, which can be found by perturbing the cross-cuts defining the elements of $\mc{C}$, and if $\mc{C}$ divides $D$ then $\mc{C}'$ divides $D$ as well. 

To conclude that $i_*$ is a homeomorphism, let us show that $i_*^{-1}$ is continuous. The fact that every element of $\bdPE(U,S)$ is represented by some prime chain of $U$ as a subset of $S_0$ implies that if $D$ is a cross-section of $U$ as a subset of $S$ then $\iPE_{U,S}(D)$ is the union of all sets of the form $\iPE_{U,S}(D')$ where $D'\subset D$ is a cross-section of $U$ as a subset of $S_0$. Since we already showed that $i_*(\iPE_{U,S}(D')) = \iPE_{U,S_0}(D')$ for such cross-sections, we have that $i_*(D\cup \iPE(D,S))$ is a union of sets of the form $D'\cup \iPE_{U,S_0}(D')$ where $D'\subset D$ is a cross-section of $U$ in $S_0$. Thus $i_*$ maps open sets to open sets, and we conclude that $i_*$ is a homeomorphism, completing the proof of $(3)$.

Finally, we observe that for the case that $U$ is relatively compact in $S$, part (1) is already known; see for instance \cite{mather-caratheodory} or \cite{mather-area}. In particular, since $\cIB S = S\cup \bdIB S$ is compact and $\bdIB S$ is compact and totally disconnected, applying (3) to the closed surface $\cIB S$ in place of $S$ we conclude that $\cPE(U,S)$ is homeomorphic to $\cPE(U,\cIB S)$, from which (1) follows.
\qed

\subsection{Proof of part (2)}

Note that, since parts (1) and (3) are already proved, we may use Proposition \ref{pro:prime-criterion}, which only depends on those items.
We begin by observing that since $U\subsetneq S_0$, the boundary of $U$ in $S_0$ has more than one point (the only case where $\bd_{S_0}U$ can be a unique point is when $S_0$ is a sphere, which is impossible in our case). Thus $\cPE(U,S_0)$ is well defined. 

\begin{claim*} Any prime chain of $U$ as a subset of $S$ divides some chain which is a prime chain of $U$ as a subset of $S_0$.
\end{claim*}
\begin{proof}
Observe that in view of parts (1) and (3) of Theorem \ref{th:prime-equiv} we may assume that $S$ is a closed surface, replacing $S$ by $\cIB S$ and $S_0$ by $S_0\sm \bdIB S$ (the hypotheses still hold because $\bdIB S$ is closed and totally disconnected in $\cIB S$). Thus, by Proposition \ref{pro:prime-criterion}, any prime end in $\bdPE(U,S)$ is represented by a prime chain $\mc{C} = (D_i)_{i\in \N}$ such that $\bd_U D_i \to x$ in $S$ for some $x\in \bd U$. Hence, to prove our claim it suffices to show that $\mc{C}$ divides some prime chain $\mc{C}'$ of $U$ as a subset of $S_0$.

Suppose first that $x\in S_0$. Then, since $\bd_U D_i\to x$, there is $i_0\in \N$ such that $\ol{\bd_U D_i}\subset S_0$ if $i\geq i_0$, and so $\mc{C} = (D_i)_{i\geq i_0}$ is a chain for $U$ as a subset of $S_0$ as well. The fact that $\mc{C}$ is a prime chain of $U$ as a subset of $S$ also implies that it is a prime chain of $U$ as a subset of $S_0$, and so our claim holds with $\mc{C}'=\mc{C}$.

Now suppose that $x\in S\sm S_0$, and write $S' = \cIB(S_0)$ and $\gamma_i = \bd_U D_i$. Since $\gamma_i\to x$ in $S$ as $i\to \infty$, it follows that for any neighborhood $W$ of $\bdIB(S_0)$ there is $i_0$ such that $\gamma_i\subset W$ if $i\geq i_0$. By compactness of $S'$, if we replace $(D_i)_{i\in \N}$ by a subsequence (which is a chain equivalent to $\mc{C}$), we may assume that $(\cl_{S'} \gamma_i)_{i\in \N}$ converges to some subset of $S'$ in the Hausdorff topology, which must be compact and connected, because so is $\cl_{S'}\gamma_i$ for each $i\in \N$. Moreover, by our previous observations, the Hausdorff limit of such a sequence must be a subset of $\bdIB(S_0)$, which is a totally disconnected set. Thus the limit is a single point; \ie there is $\mathfrak{p}\in \bdIB(S_0)$ such that $\gamma_{i}\to \mathfrak{p}$ in $S'$ as $i\to \infty$. 

Let $(B_i)_{i\in \N}$ be a sequence of closed topological disks in $S'$ such that $B_{i+1}\subset \inter B_i$ and $\bigcap_{i\in \N} B_i = \{\mathfrak{p}\}$. Since $\bdIB (S_0)$ is totally disconnected, one may choose each $B_i$ such that $\bd_{S'} B_i \cap \bdIB(S_0)=\emptyset$, so $\bd B_i = \bd_{S'} B_i \subset S_0$. Since $\bd_{S'} U$ is connected (because $S'$ is compact), it follows that $\bd B_i\cap \bd_{S'} U\neq \emptyset$ for each $i$ (if we assume that $U$ is not entirely contained in $B_1$, which we can). Choose a connected component $\gamma_1'$ of $\bd B_1 \cap U$, and note that $\gamma_1'$ is a cross-section of $U$ in $S_0$. Since $\mc{C}$ is a prime chain of $U$ in $S$, we can find $i_1\in \N$ such that $\gamma_1'\not\subset D_{i_1}$. Moreover, if $i_1$ is large enough then $\bd D_{i_1}=\gamma_{i_1}\subset B_1$, and so it is disjoint from $\gamma_1'$, implying that $\gamma_1' \subset U\sm D_{i_1}$. Thus the connected component $D_1'$ of $U\sm \gamma_1'$ containing $\gamma_{i_1}$ also contains $D_{i_1}$. 

Choose $i_2> i_1$ large enough so that $\gamma_{i_2}\subset B_2$. Since $\bd B_2\cap U$ separates $\gamma_{i_2}$ from $\gamma_1'$ in $U$, it follows easily that there is a connected component $\gamma_2'$ of $\bd B_2\cap U$ that separates $\gamma_{i_2}$ from $\gamma_1'$ in $U$. Since $D_{i_2}\subset D_1'$, one deduces that the connected component $D_2'$ of $U\sm \gamma_2'$ containing $\gamma_{i_2}$ satisfies $D_{i_2}\subset D_2'\subset D_1'$. Repeating this process recursively, we obtain a chain $\mc{C}'=(D_i')_{i\in \N}$ of $U$ in $S_0$ and an increasing sequence $(i_k)_{k\in \N}$ such that $\bd_U D_k' \subset \bd B_k$ and $D_{i_k}\subset D_k'$ for each $k\in \N$. This means that $\mc{C}'$ regarded as a chain of $U$ in $S$ is divided by $\mc{C}$. Moreover, since $\bd_U D_k'\subset B_k$, we have that $\bd_U D_k'\to \mathfrak{p}$ in $S'=\cIB(S_0)$ as $k\to \infty$, so Proposition \ref{pro:prime-criterion2} implies that $\mc{C}'$ is a prime chain of $U$ in $S_0$, completing the proof of the claim.
\end{proof}

Given $p\in \bdPE(U,S)$, let $p'$ be a prime end in $\bdPE(U,S_0)$ such that some prime chain $\mc{C}=(D_i)_{i\in \N}$ of $U$ in $S$ representing $p$ divides a prime chain $\mc{C}'=(D_i')_{i\in \N}$ of $U$ in $S_0$ representing $p'$. The previous claim guarantees that such $p'$ exists. We will show that it is uniquely determined by $p$. Suppose on the contrary that there is $p''\in \bdPE(U,S_0)$ represented by a prime chain $\mc{C}''=(D_i'')_{i\in \N}$, such that $\mc{C}$ divides $\mc{C}''$, and $p''\neq p'$.  Observe that $\cl_{\cPE(U,S_0)} D_i'$ is eventually contained in any neighborhood of $p'$ in $\cPE(U,S_0)\simeq \ol{\D}$, and the analogous fact for $p''$. This implies that $D_i'$ and $D_j''$ are disjoint if $i,j$ are chosen large enough. But this is a contradiction, because both $\mc{C}'$ and $\mc{C}''$ are divided by $\mc{C}$, so for any $i,j\in \N$ the set $D_i'\cap D_j''$ contains some element of $\mc{C}$.  This contradiction shows that $p'$ is uniquely determined by $p$, and so letting $i_*(p)=p'$ we have a well-defined map $i_*\colon \bdPE(U,S)\to \bdPE(U,S_0)$. 

Moreover, if $p'\in \bdPE(U,S_0)$ is represented by a prime chain $\mc{C}' = (D_i')_{i\in \N}$, then $i_*^{-1}(p')$ consists of all prime ends represented by a chain of $U$ in $S$ that divides $\mc{C}'$, which is precisely
$$i_*^{-1}(p') = \bigcap_{i\in \N} \iPE_{U,S}(D_i').$$
This is a nested intersection of open intervals, and since the endpoints of $\bd_U D_i'$ in $S$ are disjoint from the corresponding endpoints of $\bd_U D_{i+1}'$ (from the definition of chain), it follows (by the remarks after Proposition \ref{pro:prime-criterion2}) that the endpoints of $\bd_U D_i'$ in $\cPE(U,S)$ are disjoint from the corresponding endpoints of $\bd_U D_{i+1}$ in $\cPE(U,S)$. This means that the closure of $\iPE_{U,S}(D_{i+1}')$ in $\bdPE(U,S)$ is contained in $\iPE_{U,S}(D_i')$  for each $i\in \N$. Therefore the intersection definining $i_*^{-1}(p')$ is a nonempty closed interval, showing that $i_*$ is surjective and monotone.

Define $i_*(x)=x$ for $x\in U$, so that $i_*\colon \cPE(U,S)\to \cPE(U,S_0)$ is now a monotone surjection extending the inclusion. To see that $i_*$ is continuous, it suffices to show that $i_*^{-1}(D\cup \iPE_{U,S_0} (D)) = D\cup  \iPE_{U,S}(D)$ for each cross-section $D$ of $U$ in $S_0$. Since $i_*^{-1}(D)=D$, it suffices to show that $i_*^{-1}(\iPE_{U,S_0}(D)) = \iPE_{U,S}(D)$. Moreover, it is clear that $i_*^{-1}(\iPE_{U,S_0}(D))\subset \iPE_{U,S}(D)$, so we need to prove that $i_*(\iPE_{U,S}(D))\subset \iPE_{U,S_0}(D)$.
Let $p\in \iPE_{U,S}(D)$, so that $p$ is represented by some prime chain $\mc{C}=(D_i)_{i\in \N}$ of $U$ in $S$ that divides $D$. Let $\mc{C}'=(D_i')_{i\in \N}$ be a prime chain of $U$ in $S_0$ that represents $i_*(p)$ (so that it is divided by $\mc{C}$). We will show that $\mc{C}'$ also divides $D$, which is enough to complete the proof since it implies that $i_*(p)\in \iPE_{U,S_0}(D)$. 

By Proposition \ref{pro:prime-criterion2} we may assume that there is $x\in \cIB(S_0)$ such that $\bd_U D_i' \to x$ in $\cIB(S_0)$ as $i\to \infty$. If $x\in S_0$, then the same proposition implies that $\mc{C}'$ is a prime chain of $U$ as as a subset of $S$, and so it is equivalent to $\mc{C}$, proving that it divides $D$, which is what we wanted. Now suppose that $x\in \bdIB(S_0)$. Since $\cl_{S} \bd_U D \subset S_0$, the fact that $\bd_U D_i'\to x$ as $i\to \infty$ implies that $\bd_U D_i'$ is disjoint from $\bd_U D$ if $i$ is large enough. On the other hand, since $\mc{C}'$ is divided by $\mc{C}$ which in turn divides $D$, it follows that $D_i'\cap D\neq \emptyset$ for all $i$. From these two facts, we have that either $D_i'\subset D$ for some $i\in \N$, or $D \subset D_i'$ for all $i\in \N$. But in the latter case, the endpoints of $\bd_U D$ in $\cPE(U,S_0)$ are accessible prime ends that divide $D_i'$ for each $i\in \N$ (and so they divide $\mc{C}'$), contradicting the fact that $\mc{C}'$ is a prime chain of $U$ in $S_0$. Thus $D_i'\subset D$ for some $i\in \N$, showing that $\mc{C}'$ divides $D$. This completes the proof of Theorem \ref{th:prime-equiv} (2).
\qed

\bibliographystyle{koro} 
\bibliography{kln}

\end{document}